\documentclass[a4paper,10pt]{article}

\usepackage[utf8]{inputenc}
\usepackage{authblk}
\usepackage{amsthm,amsmath,amssymb,latexsym,dsfont}

\usepackage{enumerate}
\usepackage{tikz}
\usetikzlibrary{arrows.meta, positioning}

\usepackage{graphicx}
\usepackage{caption}
\usepackage{subcaption}

\usepackage{xcolor}

\usepackage{url}

\newtheorem{theorem}{Theorem}[section]
\newtheorem{proposition}[theorem]{Proposition}
\newtheorem{lemma}[theorem]{Lemma}
\newtheorem{corollary}[theorem]{Corollary}

\newtheorem{definition}[theorem]{Definition}

\newtheorem{assumption}[theorem]{Assumption}

\newcommand{\N}{\mathbb{N}}
\newcommand{\Z}{\mathbb{Z}}

\newcommand{\R}{\mathbb{R}}

\newcommand{\B}{\mathcal{B}}

\newcommand{\F}{\mathcal{F}}

\newcommand{\X}{\mathcal{X}}
\newcommand{\Y}{\mathcal{Y}}

\renewcommand{\P}{\mathbb{P}}
\newcommand{\E}{\mathbb{E}}

\newcommand{\law}[1]{\text{Law}(#1)}

\newcommand{\Pas}{\text{a.s.}}
\newcommand{\ind}{\mathds{1}}

\newcommand{\eps}{\varepsilon}

\newcommand{\ga}{\gamma}

\newcommand{\dtv}{d_{\text{TV}}}
\newcommand{\cov}{\mathrm{Cov}}
\newcommand{\var}{\mathrm{Var}}

\newcommand{\dint}{\mathrm{d}}  
\newcommand{\leb}{\mathrm{Leb}}

\usepackage[margin=30mm]{geometry}

\newcommand{\lfrf}[1]{\left\lfloor #1\right\rfloor}

\title{Mixing properties of some Markov chains models in random
	environments\thanks{The first author was supported by the János Bolyai Research Scholarship of the Hungarian Academy of Sciences, the Thematic Excellence Program 2021 (National Research, Development and Innovation Office, subprogram “Artificial intelligence, large networks, data security: mathematical foundation and applications”), and the grant K 143529.}}
\author[1,2]{Attila Lovas}
\author[3]{Lionel Truquet}

\affil[1]{HUN-REN Alfr\'ed R\'enyi Institute of Mathematics, Budapest, Hungary}
\affil[2]{Budapest University of Technology and Economics, Budapest, Hungary}
\affil[3]{Univ Rennes, Ensai, CNRS, CREST -- UMR 9194, F-35000 Rennes, France}
\date{\today}

\begin{document}
	
	\maketitle
	
	\begin{abstract}	
		Markov chains in random environments (MCREs) have recently attracted renewed interest, as these processes naturally arise in many applications, such as econometrics and machine learning. Although specific asymptotic results, such as the law of the large numbers and central limit theorems, have been obtained for some of these models, their annealed dependence properties, such as strong mixing properties, are not well understood in general. We derive strong mixing properties for a wide range of MCREs that satisfy some drift/small set conditions, with general assumptions on the corresponding stochastic parameters and the mixing properties of the environments. We then demonstrate the wide range of applications of our results in time series analysis and stochastic gradient Langevin dynamics, with fewer restrictions than those found in existing literature.
		
	\end{abstract}

	\section{Introduction}
	
	Finite-memory information processing systems arise in various domains, including communication networks, control systems, and machine learning applications. In particular, learning from data streams has become a prominent and actively researched area. Markov Chains in Random Environments (MCREs) offer a mathematical framework for modeling information processing systems with finite memory but for which an additional stochastic disorder has an influence on the dynamics.

Though many existing articles \cite{stenflo}, \cite{orey1991markov}, \cite{kifer1} or \cite{cogburn1984ergodic} have already investigated some theoretical properties of these objects, there are some important new aspects concerning the case of unbounded and continuous state spaces which have been only studied recently. For instance \cite{lovas2021ergodic}, \cite{Truquet1}, \cite{doukhan2023stationarity} derived many sufficient conditions on the transition mechanism of the MCRE to ensure the existence of a stationary and ergodic solution, from which various applications can be obtained such as the study of the long-time behavior of stochastic gradient Langevin algorithms, some limit properties of queing systems with time dependencies or the existence of stationary autoregressive time series with exogenous covariates. 

Beyond the use of stationarity and ergodic properties which are sufficient to deal with some specific applications, many other limit theorems such as central limit theorems for partial sums or convergence of the marginal probability measure for some stochastic iterations require an extra amount of work. Usually, there are two points of view for analyzing processes in random environments. The first one, called quenched asymptotics, allows to consider the limiting properties for some specific statistics of the process, conditionally on the environment. See for instance \cite{kifer1998limit} for quenched central limit theorems obtained for MCRE satisfying Doeblin's type minorization condition and \cite{cogburn91} for a functional CLT when the state space is discrete. 
In this work, we will investigate the second kind of behavior, called annealed, for which the convergence of some statistics is obtained unconditionally on the environment. Such annealed results can be easily obtained by integrating quenched results but we will use another approach which consists in controlling the decay of some classical dependence coefficients used for deriving an asymptotic theory for stochastic processes. For instance, the so-called $\alpha-$mixing coefficients (sometimes called strong-mixing coefficients) are widely used in this context. Controlling strong mixing coefficients of such models has also its own interest, since a battery of asymptotic results and statistical procedures have been developed for such processes. See for instance \cite{doukhan1995mixing}, \cite{Rio} or \cite{Bradley2005}. In this context, it is natural to assume that the environment itself satisfies some strong mixing properties and
to study the so-called "tranfert mixing property" in the spirit of the recent works \cite{lovas2024transition} or \cite{TRUQUET2023294}. More precisely, if $(Y_t)_t$ denotes the environment or the sequence of exogenous covariates and $(X_t)_t$ the MCRE, the transfert mixing property allows to prove a mixing property for the pair $(X_t,Y_t)_t$ from the mixing property of $(Y_t)_t$. 
    
	Our work is motivated by several applications. 
    A first one concerns time series analysis and the stability properties of autoregressive models which are widely encountered in Econometrics. Beyond the use of linear processes such as ARMA models, the stability of non-linear autoregressive models is often studied from Markov chains techniques. 
See for instance \cite{tjostheim1990non} for some general results obtained in that way. The textbook \cite{douc2014nonlinear} also provides many interesting complements. When an exogenous times series is incorporated in the dynamic of the quantity of interest, the resulting time series is no more Markovian and results based on Markov chains techniques are useless. However Markov chains in random environments are meaningful when the time series $(Y_t)_t$ is assumed to be strictly exogenous in the sense of \cite{chamberlain1982general}, see Definition $3$ or equivalently Condition $S'$. Such a strong assumption holds true in the dynamic 
\begin{equation}\label{recur}
X_t=F\left(X_{t-1},Y_{t-1},\varepsilon_t\right),\quad t\geq 1,
\end{equation}
as long as the time series $(Y_t)_t$ is independent from the noise sequence $(\varepsilon_t)_t$ which is assumed to have i.i.d. coordinates. Connections between time series with strictly exogenous covariates and Markov chains in random environements
has been exploited recently in order to get sufficient conditions for the existence of stationary sequences satisfying (\ref{recur}). See in particular \cite{Truquet1} or \cite{doukhan2023stationarity} and the references therein. 
With respect to these references, we obtain strong mixing properties for a wide range of unbounded autoregressive models and non-necessarily stationary covariates. 
We also complement the resuts of \cite{TRUQUET2023294} which mainly concern stationary environments and bounded state spaces with a uniform minorization condition, often called Doeblin's condition in Markov chains theory. Such a uniform condition is often too restrictive when the support of the time series is not compact. Applications to parametric or non-parametric statistics will be then considered.

Another application concerns stochastic gradient Langevin algorithms as in \cite{5author}, \cite{6} or \cite{lovas}. In this case, we will derive asymptotic properties of such algorithms under less restrictive conditions on the parameters of the model. More generally, our work considers the mixing properties of MRCE when the drift/small set parameters of the model, which depend on the environment, are not uniformly bounded. This extension involves a substantial technical difficulty but it allows to consider a large class of examples and many statistical applications for MCRE with unbounded exogenous processes.

The paper is organized as follows. In Section \ref{sec:main}, we present our main results. 
Applications to time series analysis are presented in Section \ref{sec:econometrics} while convergence results for Langevin type dynamics are given in Section \ref{sec:langevin}. Finally most of our proofs are postponed to Section \ref{ap:proofs} which also contains many technical lemmas of independent interest.

	\bigskip
	\noindent
	{\bf Notations and conventions.} Let $\R_{+}:=\{x\in\R:\, x\geq 0\}$
	and $\N_{+}:=\{n\in\N:\ n\geq 1\}$. Let $(\Omega,\F,\P)$ be a probability space. We denote by $\E[X]$ the expectation of a random variable $X$. For $1\le p<\infty$, $L^p$ is used to denote the usual space of $p$-integrable real-valued random variables and $\Vert X \Vert_p$ stands for the $L^p$-norm of a random variable $X$.
	
	In the sequel, we employ the convention that $\inf \emptyset=\infty$, $\sum_{k}^{l}=0$ and $\prod_{k}^{l}=1$ whenever $k,l\in\Z$, $k>l$. 
	Lastly, $\langle \cdot\mid \cdot\rangle$ denotes the standard Euclidean inner product
	on finite dimensional vector spaces. For example, on $\R^d$, $\langle x\mid y\rangle = \sum_{i=1}^{d} x_i y_i$.
	
	\section{Main results}\label{sec:main}
	
	Let $\X$ and $\Y$ be complete and separable metric spaces.
	The function \( Q:\Y \times \X \times \B(\X) \to [0,1] \) is a parametric probabilistic kernel, meaning that
	\begin{enumerate}[i]
		\item for each pair \( (y, x) \in \Y \times \X \), the mapping \( B \mapsto Q(y, x, B) \) defines a Borel probability measure on the Borel sigma-algebra \( \B(\X) \), and
		
		\item for any set \( B \in \B(\X) \), the mapping \( (x, y) \mapsto Q(y, x, B) \) is measurable with respect to the product sigma-algebra \( \B(\X) \otimes \B(\Y) \).
	\end{enumerate}
	
	\begin{definition}\label{def:MCRE}
	Let $(Y_n)_{n \in I}$ be a $\Y$-valued process, where $I = \Z$ or $I=[N,\infty)$ for some $N\in\Z$. The process $(X_n)_{n \in I}$ is a Markov chain in a random environment if there exists a parametric probabilistic kernel $Q: \Y \times \X \times \B(\X) \to [0,1]$ such that
	$$
	\P(X_{n+1} \in B \mid (X_k)_{k \le n}, (Y_k)_{k \in I}) = Q(Y_n, X_n, B), \quad B \in \B(\X),\,\, n \in I.
	$$
	\end{definition}
	The process $(X_n)_{n\in I}$ in Definition \ref{def:MCRE} defines a time-inhomogeneous Markov chain conditionally on the process $(Y_n)_{n\in I}$ being interpreted as random environment. This characterization leads us to term this process a Markov chain in a random environment (MCRE).
    In the remainder of the paper, we focus on two settings: (i) when \(I = \mathbb{N}\), and the initial value \(X_0\) is independent of the environment \((Y_n)_{n \in \mathbb{N}}\); and (ii) when \(I = \mathbb{Z}\), and the environment \((Y_n)_{n \in \mathbb{Z}}\) is a stationary process. In the latter case, it was shown in \cite{Truquet1} that under mild conditions, there exists a process \((X_n^*)_{n \in \mathbb{Z}}\) such that the joint process \((X_n^*, Y_n)_{n \in \mathbb{Z}}\) is stationary, and \((X_n^*)_{n \in \mathbb{Z}}\) is a Markov chain in the random environment \((Y_n)_{n \in \mathbb{Z}}\), moreover the distribution of \((X_n^*, Y_n)_{n \in \mathbb{Z}}\) is unique.

    \begin{definition}\label{def:act}
		Let \( Q:\Y\times \X \times \B(\X) \to [0,1] \) be a parametric probabilistic kernel. For a bounded measurable function \( \phi:\X \to \mathbb{R} \), we define
		\[
		[Q(y)\phi](x) = \int_\X \phi(z) Q(y, x, \dint z), \quad (y, x) \in \Y\times\X.
		\]
		This definition also applies to any non-negative measurable function \( \phi \).
			
		For	any collection $y_1, \ldots, y_{p} \in \Y$ with $p\ge 1$, the successive application of the kernels is interpreted in the order corresponding to the composition of conditional expectations:
		\[
		[Q(y_{1})\cdots Q(y_p)\phi] = [Q(y_1)[\ldots [Q(y_{p})\phi]]].
		\]
	\end{definition}
   Let \((X_n)_{n\in I}\) be a Markov chain in a random environment (MCRE) \((Y_n)_{n\in I}\), with \(I=\N\) or \(\Z\), and parametric kernel \(Q:\Y\times\X\times\B (\X)\to [0,1]\).  
   For any \(p \ge 1\), the subsequence \(X_n^p := X_{np}\), \(n\in I\), is again an MCRE, now in the environment  
    \[
    Y_n^p := (Y_{np},\ldots,Y_{np+p-1}) \in \Y^p, \quad n\in I,
    \]
    with associated parametric kernel
    \begin{equation}\label{eq:Qp}
    Q^p((y_1,\ldots,y_p),x,B) := [Q(y_1)\cdots Q(y_p)\ind_B](x),
    \quad (y_1,\ldots,y_p) \in \Y^p,\, x \in \X,\, B \in \B(\X).
    \end{equation}

	Markov chain in random environments and nonlinear autoregressive models incorporating exogenous regressors are closely related. If $f: \X \times \Y \times [0,1] \to \X$ is a measurable function, and $\nu$ is a probability measure on $\B ([0,1])$, then
	\begin{equation}\label{eq:f_to_Q}
	Q(y, x, B) = \int_{[0,1]} \ind_{\left\{f(x, v, z) \in B \right\}} \, \nu(\dint z)
	\end{equation}
	is a parametric probabilistic kernel. The converse of this statement is also true: any parametric kernel $Q$ admits a representation of the form \eqref{eq:f_to_Q}. Of course, this representation is not unique.
	
	\subsection{Transition of mixing properties}\label{sec:trans_mixing}
    
	There has been extensive research on random sequences with well-defined structures, such as Markov chains, martingales, and Gaussian processes. However, by the mid-20th century, it became evident that a more general framework for statistical inference was needed—one that could support the analysis of time series that did not fit into these classical structures but exhibited certain asymptotic independence properties. 
	
	This need led to the development of the broad theory of weakly dependent sequences, as well as the formulation of various strong mixing conditions, which provided a foundation for handling such cases.
	Weak dependence models offer a powerful approach to capturing long-range interdependencies while allowing for non-stationarity. For instance, in a Markov chain in a random environment (MCRE), the influence of the environment $(Y_n)_{n \in I}$ can introduce dependencies in the time series that persist over time. Therefore, understanding how the process $(X_n)_{n \in I}$ inherits the mixing properties of the environment $(Y_n)_{n \in I}$ opens the door to the statistical analysis of MCREs.
	
	Several notions of mixing are discussed in the literature. The interested reader is encouraged to consult the excellent survey by Bradley \cite{Bradley2005} and references therein, for instance. In this work, we focus on three types of strong mixing conditions: $\alpha$-mixing, $\phi$-mixing, and $\psi$-mixing. For any two sub-$\sigma$-algebras $\mathcal{G}, \mathcal{H} \subset \mathcal{F}$, we define the following measures of dependency:
	
	\begin{align}\label{eq:dep}
		\begin{split}
			\alpha(\mathcal{G}, \mathcal{H}) &= \sup\left\{\left|\P(G \cap H) - \P(G)\P(H)\right| : G \in \mathcal{G}, H \in \mathcal{H}\right\} \\
			\phi(\mathcal{G}, \mathcal{H}) &= \sup\left\{\left|\P(H \mid G) - \P(H)\right| : G \in \mathcal{G}, H \in \mathcal{H}, \P(G) > 0\right\} \\
			\psi(\mathcal{G}, \mathcal{H}) &= \sup\left\{\left|\frac{\P(G \cap H)}{\P(G)\P(H)} - 1\right| : G \in \mathcal{G}, H \in \mathcal{H}, \P(G) > 0, \P(H) > 0\right\}
		\end{split}
	\end{align}
	
	Furthermore, consider an arbitrary sequence of random variables $(W_t)_{t \in I}$. We define the $\sigma$-algebras $\mathcal{F}_{t,s}^W := \sigma \left(W_k,\, k \in I,\, t \le k \le s \right)$, where $-\infty \le t \le s \le \infty$. If $I \cap [t,s] = \emptyset$, then $\mathcal{F}_{t,s}^W$ is defined as $\{\emptyset, \Omega\}$. Additionally, we define the mixing coefficients as follows:
	\begin{align*}
		\alpha^W (n) &= \sup_{j \in \mathbb{Z}} \alpha \left(\mathcal{F}_{-\infty,j}^W, \mathcal{F}_{j+n,\infty}^W \right), \\
		\phi^W (n) &= \sup_{j \in \mathbb{Z}} \phi \left(\mathcal{F}_{-\infty,j}^W, \mathcal{F}_{j+n,\infty}^W \right),\,\,
		\text{and}\,\, 
		\psi^W (n) = \sup_{j \in \mathbb{Z}} \psi \left(\mathcal{F}_{-\infty,j}^W, \mathcal{F}_{j+n,\infty}^W \right).
	\end{align*}
	
	The sequences of strong mixing coefficients $\alpha^W(n)$, $\phi^W(n)$, and $\psi^W(n)$, for $n \ge 1$, are evidently non-increasing. The sequence $(W_n)_{n \in I}$ is $\alpha$-mixing if $\alpha^W(n) \to 0$, $\phi$-mixing if $\phi^W(n) \to 0$, and $\psi$-mixing if $\psi^W(n) \to 0$ as $n \to \infty$. The following implications hold between these strong mixing concepts:
	\[
	\psi\text{-mixing} 
	\Longrightarrow 
	\phi\text{-mixing} 
	\Longrightarrow 
	\alpha\text{-mixing}.
	\]
	
	\smallskip 
	In \cite{lovas2024transition}, we introduced the \emph{coupling condition} for nonlinear autoregressive processes with exogenous covariates. 
Coupling inequalities of this type have already been employed in \cite{TRUQUET2023294} and \cite{lovasCLT} to derive upper bounds for the strong mixing coefficient \(\alpha^X(n)\), \(n \in \mathbb{N}\). 
The following lemma refines the formulation in \cite{lovas2024transition}: here, the coupling condition pertains not to the iteration mechanism itself, but to a property of the distribution of the process \((X_n)_{n \in \mathbb{N}}\). 
Furthermore, it can be regarded as a generalization of Lemma~1.2 in \cite{lovas2024transition}, since it suffices to impose the coupling condition on a subsequence \((X_{np})_{n \in \mathbb{N}}\) for an arbitrary \(p \ge 1\), rather than on the full sequence.
	\begin{lemma}\label{lem:coupl_trans_mix}
    Let $(X_n)_{n \in \N}$ be a Markov chain in the random environment $(Y_n)_{n\in \N}$ with parametric kernel $Q$, and suppose that the initial state $X_0$ is independent of $(Y_n)_{n\in \N}$.  
Assume that for some $p \ge 1$ there exist a measurable function
\[
f^p : \X \times \Y^p \times [0,1] \to \X
\]
and an i.i.d.\ process $(\eps_n)_{n \in \N}$ on $[0,1]$, independent of both $(Y_n)_{n\in \N}$ and $(X_n)_{n\in\N}$, such that
\[
Q^p((y_1,\ldots,y_p), x, B) = \E\left[ \ind_{\{ f^p(x, (y_1,\ldots,y_p), \eps_0) \in B \}} \right],
\quad (y_1,\ldots,y_p) \in \Y^p, \, x \in \X, \, B \in \B (\X).
\]
Furthermore, suppose that there exists a deterministic $x_0 \in \X$ such that
\begin{equation}\label{eq:coupling_cond}
b(n) := \sup_{j \in \N} \P \left( Z_{jp,(j+n)p}^{X_{jp}} \neq Z_{jp,(j+n)p}^{x_0} \right) \to 0,
\quad n \to \infty,
\end{equation}
where, for any (possibly random) $x \in \X$, the process $Z_{sp,tp}^{x}$, $t \ge s$, denotes the trajectory starting from $x$ at time $sp \in \N$:
\begin{equation}\label{eq:iter}
Z_{sp,tp}^{x} =
f^p\left( Z_{sp,(t-1)p}^{x}, Y_{t-1}^p, \eps_{tp} \right),
\quad t > s, \qquad
Z_{sp,sp}^{x} = x.
\end{equation}

Then, for $n\ge p$ and \(0 \le m < \lfloor \frac{n}{p} \rfloor\),we have the following inequality:
\begin{equation}\label{eq:trans_mix}
\alpha^{X,Y}(n) \le \alpha^{Y}(mp) + b\!\left(\lfloor \tfrac{n}{p} \rfloor - 1 - m\right).
\end{equation}
\end{lemma}

	\begin{proof}
		To maintain readability, the rather technical proof is deferred to Appendix~\ref{ap:proof_of_coupl_trans_mix}.
	\end{proof}
	 
	The extension of the Foster-Lyapunov theory of general state space Markov chains equips us with the necessary analytical tools for establishing the transition of $\alpha$-mixing properties from the environment to the process by verifying that the conditions of Lemma \ref{lem:coupl_trans_mix} are met.
	
	We say that \( Q \) satisfies the \textit{drift (or Lyapunov) condition} if there exists a measurable mapping \( V:\X \to [0,\infty) \) (referred to as a Lyapunov function) and measurable functions \( \gamma, K:\Y \to (0,\infty) \), such that for all \( (y, x) \in \Y \times \X \),
	\begin{equation}\label{eq:drift}
	[Q(y)V](x) \leq \gamma(y)V(x) + K(y).
	\end{equation}
	
	The kernel \( Q \) satisfies the \textit{minorization condition} with parameter \( R > 0 \) if there exists a probability kernel \( \kappa_R: \Y \times \B(\X) \to [0,1] \) and a measurable function \( \beta:[0,\infty) \times \Y \to [0,1) \) such that for all \( (y, x, A) \in \Y \times \stackrel{-1}{V} ([0,R]) \times \B(\X) \),
	\begin{equation}\label{eq:smallset}
	Q(y, x, A) \geq (1 - \beta(R, y))\kappa_R(y, A).
	\end{equation}
	The minorization condition guarantees the existence of ``small sets.'' Hence, it is also known as a ``small set''-type condition.
	
	\medskip 
	If $\ga,K$ are independent of $y$ and $\ga<1$ then \eqref{eq:drift} is the standard drift condition for geometrically ergodic Markov chains, see Chapter 15 of \cite{mt}. In \cite{lovas}, authors studied the ergodic properties of MCREs when the environment $(Y_n)_{n\in\Z}$ is stationary and $\gamma(y) \geq 1$ may occur.  
	In their analysis, alongside with some minor technical conditions, the following \emph{long-term contractivity} assumption played a crucial role:
	\begin{equation}\label{eq:LT}
		\limsup_{n\to\infty}\E^{1/n}\left(K(Y_0)\prod_{k=1}^{n}\gamma (Y_k)\right) < 1.
	\end{equation}
	
	To the best of our knowledge, the strongest results on the existence, stability, and ergodicity of stationary solutions for Markov chains in random environments, when the environment $(Y_n)_{n\in\Z}$ is itself stationary, are presented in \cite{Truquet1}.	
	Under the assumption that $\E \left[\log (\ga (Y_0))_+\right]+\E \left[\log (K (Y_0))_+\right]<\infty$ and
	\begin{equation}\label{eq:Truq}
		\limsup_{n\to\infty}\prod_{k=1}^{n}\ga (Y_{-k})^{1/n}<1,\,\,\P-\Pas,
	\end{equation}
	which is notably weaker than \eqref{eq:LT}, in \cite{Truquet1} Truquet proved that there exists a stationary process $((Y_n, X_n^\ast))_{n\in\Z}$, where $(X_n^\ast)_{n\in\Z}$ is a MCRE, and the distribution of $((Y_n, X_n^\ast))_{n\in\Z}$ is unique.
	
	If, in addition, the environment $(Y_n)_{n\in\Z}$ is ergodic, the process $((Y_n,X_n^\ast))_{t\in\Z}$ is also ergodic. As a result, the strong law of large numbers holds for any measurable function $\Phi:\Y\times\X\to\R$ with $\E (|\Phi (Y_0,X_0^\ast)|)<\infty$ i.e.
	$$
	\frac{1}{n}\sum_{k=1}^{n} \Phi (Y_k, X_k^\ast)\to \E (\Phi (Y_0,X_0^\ast)),\,\,\text{as}
	\,\,n\to\infty,\,\,\P-\Pas
	$$
	In this case, the condition \eqref{eq:Truq} boils down to $\E \left[\log (\ga (Y_0))\right]<0$.

	\medskip 
	The very recent work \cite{lovas2024transition} establishes the transition of $\alpha$-mixing properties from the potentially non-stationary environment $(Y_n)_{n\in\N}$ to the process $(X_n)_{n\in\N}$ under the following assumptions:
	\begin{enumerate}[i.]
		\item Beyond the integrability condition $\E \left(K(Y_j)\prod_{k=1}^{n}\ga(Y_{k+j})\right)<\infty$, for all $j\in\N$, $n\ge 1$, the following uniform version of the long-term contractivity condition holds:
		\[
		\bar{\ga}:=\limsup_{n\to\infty}\sup_{j\ge -1} \E^{1/n}\left[K(Y_j)\prod_{k=1}^{n}\ga(Y_{k+j})\right]<1,
		\]
		where $K(Y_{-1}):=1$.
		
		\item For some $0<r<1/\bar{\ga}-1$, the parametric kernel $Q$ satisfies the minorization condition with
		\[
		R(y)=\frac{2K(y)}{r\ga(y)},\quad \text{and}\quad \bar{\beta}:=\sup_{y\in\Y}\beta\left(R(y),y\right)<1.
		\]
	\end{enumerate}
	In concrete models, such as single-server queuing systems, for example, condition ii.\ is typically satisfied via uniform minorization, meaning that the parametric kernel $Q$ satisfies the minorization condition for every $R>0$ with a constant $\bar{\beta}$ which may depend on $R$, but is independent of $y$.
	
	Under condition i., there exist constants $L>0$ and $\chi\in(0,1)$ such that for all $j\in\N$,
    \begin{equation}\label{contract}
	\E\left[K(Y_j)\prod_{k=1}^{n}\ga(Y_{k+j})\right] \leq L\chi^n.
	    \end{equation}
	However, requiring (\ref{contract}) would impose a too strong restriction, since condition i.\ does not follow from the $\alpha$-mixing property of the environment process $(Y_n)_{n\in\N}$ (see the counterexample presented in Appendix B of \cite{lovas2024transition}).
However, we show in this paper that such a long-term contractivity condition can be valid for our models for $\phi$ or $\psi-$mixing environments. See Lemma \ref{lem:ga} for details.

	The aim is to investigate the general setting where the environment $(Y_n)_{n\in\N}$ is not necessarily stationary, minorization is not uniform, and a long-term contractivity condition with a possibly sub-geometric convergence rate holds.
       Our working assumptions are as follows:
	\begin{assumption}\label{as:main_assumptions}
	We assume that the parametric kernel $Q$ satisfies the drift \eqref{eq:drift} condition and the minorization
	condition \eqref{eq:smallset}, with every $R>0$, such that the following hold:
	\begin{itemize}
		\item[\bf A1] 
        We have
        $$r_0:=\sum_{\ell\geq 0}d_{\ell}<\infty,\mbox{ with } d_0=\sup_{t\geq 0}\E\left[K(Y_t)+\gamma(Y_t)\right]<\infty \mbox{ and }$$
        $$d_{\ell}:=\sup_{t\geq -1}\E\left[K(Y_t)\prod_{i=1}^{\ell}\gamma(Y_{t+i})\right],\quad \ell\geq 1.$$
        		
		\item[\bf A2] For any $R>0$, $\lim_{\overline{\beta}\uparrow 1}\sup_{t\in\N}\P\left(\beta(R,Y_t)>\overline{\beta}\right)=0$.
	\end{itemize}
	\end{assumption}

\paragraph{Note.} When the process $(Y_t)_{t\in\Z}$ is stationary, we simply have $d_{\ell}=\E\left[K(Y_0)\gamma(Y_1)\cdots \gamma(Y_{\ell})\right]$ and Assumption {\bf A2} holds true as soon as $\P\left(\beta(R,Y_0)=1\right)=0$ for any $R>0$.

    The following lemma shows that one can also use the drift function $V^{\delta}$ instead of $V$ for any $\delta\in (0,1)$. Choosing $\delta$ smaller than $1$ can be useful for checking Assumption \ref{as:main_assumptions} when the sequence $\left(d_{\ell}\right)_{\ell\geq 1}$ is not summable or for getting better rates of convergence for the new sequence. 
    See for instance Lemma \ref{products} for some specific cases.
    For simplicity, we only consider the drift function $V$ in our statements, the choice of a specific power will be a task related to specific examples.
	\begin{lemma}\label{lem:V^delta}
		Assume that the parametric kernel $Q:\Y\times\X\times\B (\X)\to [0,1]$ satisfies the drift condition \eqref{eq:drift} with $\ga, K:\Y\to (0,\infty)$. Then for any $\delta\in (0,1)$,
		$$
		[Q(y)V^\delta](x)\le \ga (y)^\delta V(x)^\delta + K(y)^\delta.
		$$
	\end{lemma}
	\begin{proof}
		For $\delta \in (0,1)$, the function $x \mapsto x^\delta$, $x \geq 0$ is concave, monotonically increasing, and also subadditive. Hence, by Jensen's inequality, we immediately obtain  
		\begin{equation*}
			[Q(y)V^\delta](x) \leq \left([Q(y)V](x)\right)^\delta
			\le
			(\gamma (y)V(x) + K(y))^\delta
			\le  \gamma (y)^\delta V(x)^\delta + K(y)^\delta.
		\end{equation*}   
	\end{proof}

We next show that Assumption \ref{as:main_assumptions} guarantees integrability of $V$ in the annealed case. 
\begin{lemma}\label{mom}
Suppose that Assumption \ref{as:main_assumptions} holds true. Suppose furthermore that either $X_0$ is independent from $(Y_t)_{t\in\Z}$
with $\E\left[V(X_0)\right]<\infty$ or $\left((X_t,Y_t)\right)_{t\in \Z}$ is a stationary process. Then $\sup_{j\geq 0}\E\left[V(X_j)\right]<\infty$.
\end{lemma}

\paragraph{Proof of Lemma \ref{mom}.} 
To simplify the notations, we denote $K(Y_t)$ or $\gamma(Y_t)$ by $K_t$ or $\gamma_t$ respectively.
Our drift condition yields to
$$\E\left[V(X_j)\vert X_0,Y\right]\leq \left(\prod_{i=0}^{j-1}\gamma_i\right) V(X_0)+K_{j-1}+\sum_{i=0}^{j-2} K_i\gamma_{i+1}\cdots \gamma_{j-1}.$$
When $X_0\equiv x_0$, using the inequality $K\geq 1$, we get 
$$\E\left[V(X_j)\right]\leq d_j V(x_0) +r_0\leq r_0(1+V(x_0))<\infty.$$
Let us now consider the stationary case. From Assumption \ref{as:main_assumptions}, we have $\log d_n< 0$ for $n$ sufficiently large. On the other hand, Jensen's inequality entails
$$\log d_n\geq \E\left[\log K_0\right]+n \E\left[\log \gamma_0\right]\geq n \E\left[\log \gamma_0\right].$$
We deduce that $\E\left[\log \gamma_0\right]<0$. We know that under this condition, there exists a unique stationary solution as well as a sequence of stationary random probability measure $\left(\pi_t\right)_{t\in\Z}$ such that $\E\left[h(X_0)\right]=\E\int h(x) \pi_0(dx)$
and 
$$\pi_0=\lim_{n\rightarrow \infty}\delta_x Q\left(Y_{t-n}\right)\cdots Q\left(Y_{-1}\right)\mbox{ a.s.},$$
where the convergence of measures holds true for to the total variance distance and the limit does not depend on the intial state $x$.

Setting $V_M=V\wedge M$ for some positive constant $M$, we have 
$$\delta_x Q\left(Y_{-n}\right)\cdots Q\left(Y_{-1}\right)V_M\leq \delta_x Q\left(Y_{-n}\right)\cdots Q\left(Y_{-1}\right)V\leq V(x)\prod_{i=-n}^{-1}\gamma_i+K_{-1}+\sum_{i\geq 0}K_{-i}\gamma_{-i+1}\cdots \gamma_{-1}.$$
By taking the almost sure limit and using the Lebesgue theorem, we get 
$$\E\left[V_M(X_0)\right]=\E\left[\pi_0 V_M\right]\leq r_0<\infty.$$
Letting $M\nearrow \infty$ and using the monotone convergence theorem, we deduce that $\E\left[V(X_0)\right\vert]<\infty$ which concludes the proof.$\square$

\begin{corollary}\label{cor:moment_bound}
   Let \(\Phi: \mathcal{X} \to \mathbb{R}\) be a measurable function. Most limit theorems for mixing processes (e.g., law of large numbers, central limit theorems, etc.) require that the process \((\Phi(X_n))_{n \in \mathbb{N}}\) has uniformly bounded moments of sufficiently high order. Lemma~\ref{mom} provides an easily verifiable condition to ensure this. Suppose that for some constant \(C > 0\) and exponent \(p > 1\), we have
\[
|\Phi(x)|^p \le C (1 + V(x)).
\]
Then, by Lemma~\ref{mom}, it immediately follows that
\[
\sup_{j \ge 0} \|\Phi(X_j)\|_p < \infty.
\]
\end{corollary}

    We now state the main result of the paper. 
    
    \begin{theorem}\label{thm:A12fromMixing}
	 Suppose that Assumption \ref{as:main_assumptions} holds true and that $Y$ is $\alpha-$mixing. Set $r_i=\sum_{\ell\geq i}d_{\ell}$ for some positive integer $i$. Then there exist a function $f$ in Lemma \ref{lem:coupl_trans_mix}, $\kappa\in (0,1)$ and a positive constant $c$ such that 
     \begin{equation}\label{boundmain}
     b(n)\leq c \inf_{1\leq i\leq q\leq n}\left\{r_i+\kappa^{n/q}+\alpha^Y(q+1-i)\right\}.
     \end{equation}
    \end{theorem}

We immediately deduce a bound for strong mixing coefficients. For clarity of the presentation, we state a separate result.

\begin{corollary}\label{cor:A12fromMixing}
 Suppose that Assumption \ref{as:main_assumptions} holds true and that $Y$ is $\alpha-$mixing. Set $r_i=\sum_{\ell\geq i}d_{\ell}$ for some positive integer $i$. Then there exist $\kappa\in (0,1)$ and a positive constant $c$ such that 
     \begin{equation}\label{boundmain}
     \alpha^{X,Y}(n)\leq c \inf_{1\leq i\leq q\leq n}\left\{r_i+\kappa^{n/q}+\alpha^Y(q+1-i)\right\}.
     \end{equation}

\end{corollary}

	\paragraph{Notes}
    \begin{enumerate}
        \item
    From a suitable choice of the pair $(i,q)$, Corollary \ref{cor:A12fromMixing} ensures that the pair $(X,Y)$ is strongly mixing. The decays of the mixing coefficients depends in particular on that of the sequence $(r_i)_{i>0}$. Getting such a rate is non-trivial and depends on some assumptions on the drift parameters. Many possible rates are given in Lemma \ref{conc} in the Appendix. In particular, under the long-term contractivity assumption (e.g. when $\sup_{y\in \mathcal{Y}}\gamma(y)<1$), the following transfer mixing properties are valid. If $\alpha^Y(n)=O\left(\rho^n\right)$ for some $\rho\in (0,1)$, then the choice $i=[q/2]$ and $q\sim \sqrt{n}$ yield to $\alpha^{X,Y}(n)=O\left(\overline{\rho}^{\sqrt{n}}\right)$ for some $\overline{\rho}\in (0,1)$. On the other hand, if $\alpha^Y(n)=O\left(n^{-a}\right)$, then choosing $i=[q/2]$ and $q\sim c n/\log n$ yield to $\alpha^{X,Y}(n)=O\left(n^{-a}\log^a n\right)$.
    \item 
    Note that in the case where $\sup_{y\in \mathcal{Y}}\gamma(y)<1$, $\sup_{y\in \mathcal{Y}}K(y)<\infty$ and $\inf_{y\in \mathcal{Y}}\beta(R,y)>0$, it is possible to obtain a better rate for $b(n)$, typically an exponential rate. See for instance \cite{lovasCLT} for stochastic gradient Langevin dynamics but the principle is more general since some standard coupling methods such as that introduced in \cite{rosenthal1995minorization} can be extended to such time-inhomogeneous Markov chains. In contrast, from the previous point, we observe that the bound given in Theorem \ref{thm:A12fromMixing} cannot give a better rate than $\beta^{\sqrt{n}}$ with $\beta\in (0,1)$, even when the environment is geometrically mixing. The reason is that our assumptions are very general and allow non-uniform drift or minorization conditions. But inspecting the proof of Theorem \ref{thm:A12fromMixing} shows that some simplifications can be made. In particular, the number of random time points $L_{n,j}$ given in Lemma \ref{lem:tau_chain}, which counts the number of return times of the process $(Y_t)_t$ in a suitable stability region, automatically equals to $n$ in this case. Then one can also recover the exponential rate.
\end{enumerate}
\subsection{Forward coupling with the stationary solution}\label{sec:stac_environment}

Let $(Y_n)_{n\in\Z}$ be a strongly stationary random environment.  
In the proof of Lemma~\ref{mom}, we established that under Assumption~\ref{as:main_assumptions},  
the conditions of Theorem~1 in \cite{Truquet1} are satisfied.  
Consequently, there exists a stationary solution $(X_n^\ast,Y_n)_{n\in\N}$ whose distribution is unique, and moreover  
\(\E [V(X_0^\ast)] < \infty\).

We now show that under Assumption~\ref{as:main_assumptions},  
if $(X_n)_{n\in\N}$ is an MCRE starting from an arbitrary initial value \(X_0\) independent of \((Y_n)_{n\in\N}\) and satisfying \(\E[V(X_0)]<\infty\),  
then $(X_n)_{n\in\N}$ and $(X_n^\ast)_{n\in\N}$ admit versions of the form \eqref{recur} which are \emph{forward coupled}.  
That is, there exists an almost surely finite random time \(\tau\) such that  
\[
Z_{0,n}^{X_0} = Z_{0,n}^{X_0^\ast}, \quad n \ge \tau,
\]
where \(Z_{0,n}^{X_0}\) and \(Z_{0,n}^{X_0^\ast}\), \(n \ge 0\), denote the respective versions of  
\((X_n)_{n\in\N}\) and \((X_n^\ast)_{n\in\N}\) in question.
This stronger notion of stability, introduced by Györfi and Morvai~\cite{gyorfi2002} following Lindvall~\cite{lindvall2002lectures},  
has proved particularly useful in the analysis of queueing systems (See e.g \cite{gyorfi2002} and Section 3 in \cite{lovas2024transition}).

\begin{theorem}\label{thm:stac_coupling}
Suppose that Assumption~\ref{as:main_assumptions} holds, and that $(Y_n)_{n\in\Z}$ is strongly stationary and $\alpha$-mixing.  
Let $(X_n^*, Y_n)_{n\in\Z}$ denote the associated stationary process, and let $X_0$ be a random initial state independent of $(Y_n)_{n\in\Z}$.  
Then there exist versions of $(X_n)_{n\in\N}$ and $(X_n^*)_{n\in\N}$ that are forward coupled.  
Moreover, the tail probability of the coupling time $\tau$ satisfies
\begin{equation}\label{boundmain}
     \P(\tau > n) \le
     c \inf_{1 \le i \le q \le n} \left\{ r_i + \kappa^{n/q} + \alpha^{Y}(q-i) \right\},
\end{equation}
where \( r_i := \sum_{\ell \ge i} d_{\ell} \) for \( i \ge 1 \), and \( c > 0 \), \( \kappa \in (0,1) \) are constants independent of \(n\).
\end{theorem}
\begin{proof}
The proof follows similar lines as the proof of Theorem \ref{thm:A12fromMixing}. For readability, it is provided in Appendix \ref{ap:stac_coupling:proof}.
\end{proof}

The following corollary to the above theorem yields an explicit and tractable upper bound for the total variation distance between $\law{X_n}$ and $\law{X_n^\ast}$, $n \in \N$.  
This bound is considerably sharper than those obtained in our earlier work~\cite{lovas}.  
A similar result was established in~\cite{lovas2024transition} under substantially stronger assumptions than Assumption~\ref{as:main_assumptions}.
\begin{corollary}\label{cor:fwd_dtv}
	Under the conditions of Theorem~\ref{thm:stac_coupling}, the following rate estimate holds:
	$$
	\lVert \law{X_n} - \law{X_n^\ast} \rVert_{TV}
	\le
    2c \inf_{1\leq i\leq q\leq n}\left\{r_i+\kappa^{n/q}+\alpha^Y(q-i)\right\}
	$$
	with the same constants $c$ and $\kappa$ as in Theorem~\ref{thm:stac_coupling}.
\end{corollary}

\begin{proof}
By the optimal transportation cost characterization of the total variation distance,  
\begin{align*}
	\frac{1}{2}\lVert \law{X_n} - \law{X_n^\ast} \rVert_{TV}
	\le
	\inf_{\kappa \in \mathcal{C}(X_n, X_n^\ast)}
	\int_{\X \times \X} \ind_{x \neq y} \, \kappa(\dint x, \dint y),
\end{align*}
where $\mathcal{C}(X_n, X_n^\ast)$ denotes the set of couplings of $\law{X_n}$ and $\law{X_n^\ast}$.   
Invoking Theorem~\ref{thm:stac_coupling}, the right-hand side can be bounded as
\begin{align*}
	\inf_{\kappa \in \mathcal{C}(X_n, X_n^\ast)}
	\int_{\X \times \X} \ind_{x \neq y} \, \kappa(\dint x, \dint y)
	&\le \P(\tau > n)
	\le c \inf_{1\leq i\leq q\leq n}\left\{r_i+\kappa^{n/q}+\alpha^Y(q-i)\right\}
\end{align*}
which yields the claimed inequality.
\end{proof}

An important consequence of Theorem~\ref{thm:stac_coupling} is that the inequality \eqref{boundmain} stated in Corollary~\ref{cor:A12fromMixing} also holds for the mixing coefficients \(\alpha^{X^*,Y}(n)\), \(n \in \mathbb{N}\), of the stationary solution. We state this result separately in the following theorem.
\begin{theorem}\label{thm:mix_stac_sol}
Under the assumptions and notations of Corollary~\ref{cor:A12fromMixing}, if the process \((Y_n)_{n \in \mathbb{Z}}\) is stationary and \(((X_n^*, Y_n))_{n \in \mathbb{Z}}\) denotes the stationary solution, then 
\[
\alpha^{X^*,Y}(n) \le c \inf_{1\leq i\leq q\leq n}\left\{r_i+\kappa^{n/q}+\alpha^Y(q-i)\right\}
\]
holds.
\end{theorem}
\begin{proof}
Let \(X_0 = x_0\) be an arbitrary deterministic initial value with \(V(x_0) < \infty\), and let \((X_n)_{n \in \mathbb{N}}\) be a Markov chain in the random environment \((Y_n)_{n \in \mathbb{Z}}\) with initial state \(X_0 = x_0\). By Theorem~\ref{thm:stac_coupling}, there exist suitable versions of \((X_n)_{n \in \mathbb{N}}\) and \((X_n^*)_{n \in \mathbb{N}}\) that are forward coupled. By a slight abuse of notation, these forward coupled versions will also be denoted by \((X_n)_{n \in \mathbb{N}}\) and \((X_n^*)_{n \in \mathbb{N}}\) in the sequel. Denote by \(\tau\) the random coupling time, that is, for all \(n \ge \tau\), we have \(X_n^* = X_n\).

Let $m \in \Z$ and let $(C_k)_{k \ge m} \subset \B(\X \times \Y)$ be an arbitrary collection of Borel sets.  
Using the forward coupling of the processes $X$ and $X^*$ together with the strong stationarity of the process $(X^*,Y)$, we obtain
\[
\big|\P\big((X_{k+n},Y_{k+n}) \in C_k,\ k \ge m\big)
 - \P\big((X_{k}^*,Y_{k}) \in C_k,\ k \ge m\big)\big|
 \le 2\,\P(\tau > m+n) \ \xrightarrow[n \to \infty]{}\ 0.
\]

Let $j \in \Z$ and $n \in \N$ be fixed, and let $A \in \F_{-\infty,j}^{X^*,Y}$ and $B \in \F_{n+j,\infty}^{X^*,Y}$ be arbitrary events.  
Following the argument in the proof of Lemma~\ref{lem:coupl_trans_mix},  
there exist collections of Borel sets $(C_k)_{k \in \Z} \subseteq \B(\X \times \Y)$ such that  
\begin{equation}\label{eq:defAB1}
A = \big\{ (X_k^*, Y_k) \in C_k \ \text{for all} \ -\infty \le k \le j \big\},  
\quad  
B = \big\{ (X_k^*, Y_k) \in C_k \ \text{for all} \ k \ge j+n \big\}.
\end{equation}

By the above, for any $\eps > 0$ there exist $m \in \Z$ and $N \ge -m$ such that
\begin{equation*}
    \big| \P(A \cap B) - \P(A)\P(B) \big|
    \le \eps + \big| \P(A'_N \cap B'_N) - \P(A'_N)\P(B'_N) \big|,
\end{equation*}
where
\[
A'_N = \big\{ (X_{k+N}, Y_{k+N}) \in C_k \ \text{for all} \ m \le k \le j \big\},  
\quad  
B'_N = \big\{ (X_{k+N}, Y_{k+N}) \in C_k \ \text{for all} \ k \ge j+n \big\}.
\]
Since $\eps > 0$ was arbitrary, it follows that
\[
\big| \P(A \cap B) - \P(A)\P(B) \big| \le \alpha^{X,Y}(n).
\]
Taking the supremum over all $A \in \F_{-\infty,j}^{X^*,Y}$, $B \in \F_{n+j,\infty}^{X^*,Y}$ and $j\in\Z$ yields  
$\alpha^{X^*,Y}(n) \le \alpha^{X,Y}(n)$, from which the desired statement follows by Corollary~\ref{cor:A12fromMixing}.

\end{proof}

\subsection{Ramifications}\label{sec:ramifications}

Although a wide range of interesting models fall within the scope of the framework presented in the previous section, some important processes lie outside its reach. For instance, certain classes of vector autoregressive models do not satisfy the standard (one-step) drift and minorization conditions. To address this limitation, following the approach presented in \cite{Truquet1} and \cite{lovas}, we consider a setting in which the drift and small set conditions hold only after a fixed number of iterations of the parametric kernels. 

We close the section with a very simple example of a vector autoregressive model that satisfies the assumed \(p\)-step drift and minorization conditions, but not the one-step versions. 

\begin{assumption}\label{as:pkernel}
We assume the existence of a positive integer $p$ such that the following conditions hold:    
\begin{itemize}
\item[\bf A3] There exists a measurable mapping \( V: \mathcal{X} \to [0, \infty) \), called a Lyapunov function, and measurable functions \( \gamma, K: \Y^p \to (0, \infty) \) such that for all \( y_1, \ldots, y_p \in \Y \) and \( x \in \X \),
\begin{equation}\label{eq:pdrift}
    [Q(y_p) \cdots Q(y_1)V](x) \le \gamma (y_1, \ldots, y_p) V(x) + K(y_1, \ldots, y_p).
\end{equation}

Moreover, similarly to part A1 of Assumption \ref{as:main_assumptions}, we suppose that 
\begin{align*}
r_0:=\sum_{l\ge 0} d_l&<\infty, 
\text{ with }
d_0 = \sup_{t\ge 0} \E\left[K(Y_{tp},\ldots,Y_{tp+p-1})+\gamma (Y_{tp},\ldots,Y_{tp+p-1})\right]<\infty
\text{ and }
\\
d_l&:=\sup_{t\ge -1} \E\left[
K(Y_{tp},\ldots,Y_{tp+p-1})
\prod_{i=1}^l  \gamma (Y_{(t+i)p},\ldots,Y_{(t+i)p+p-1})
\right],\quad l\ge 1.
\end{align*}

\item[\bf A4] 
For any $R>0$, there exists a probability kernel $\kappa_R:\Y^p\times\B (\X)\to [0,1]$ and a measurable function $\beta: [0,\infty)\times\Y^p\to [0,1)$ such that
\begin{equation}\label{eq:pminor}
[Q(y_1) \cdots Q(y_p)\ind_A](x) \ge (1-\beta (R, (y_1,\ldots,y_p))) \kappa_R ((y_1,\ldots,y_p), A)
\end{equation}

for all $y_1,\ldots,y_p\in\Y$, $x\in \stackrel{-1}{V}([0,R])$, and $A\in\B (\X)$.

Furthermore, we assume that for any $R>0$,
$$
\lim_{\bar{\beta}\uparrow 1} \sup_{t\in\N} \P \left(\beta (R,(Y_{tp},\ldots,Y_{tp+p-1}))>\bar{\beta}\right)=0.
$$

\end{itemize}
\end{assumption}

Under these conditions, for the strong mixing coefficient sequence $\alpha^{X,Y}(n)$, $n \in \N$, we obtain the following bound, analogous to Corollary~\ref{cor:A12fromMixing}.
\begin{theorem}\label{thm:extended}
Let Assumption~\ref{as:pkernel} hold and suppose that the process $Y$ is $\alpha$-mixing. Define $r_i = \sum_{l \ge i} d_l$ for $i \ge 0$. Then there exist constants $\tilde{\kappa} \in (0,1)$ and $c' > 0$, such that
\begin{equation}\label{eq:palpha}
\alpha^{X,Y}(n) \le c' \inf \left\{r_i + \tilde{\kappa}^{n/q}+\alpha^Y (q-i)
\middle| 1\le i\le q\le n\right\},\quad n\ge 1.
\end{equation}

\end{theorem}
\begin{proof}

Under Assumption~\ref{as:pkernel}, the sequence $(X_{np})_{n \in \mathbb{N}}$ forms a Markov chain in the random environment  
\[
Y^p_n := (Y_{np}, Y_{np+1}, \ldots, Y_{np+p-1}), \quad n \in \mathbb{N},
\]  
to which Theorem~\ref{thm:A12fromMixing} applies. In particular, there exists a measurable function  
\[
f^p : \mathcal{X} \times \mathcal{Y}^p \times [0,1] \to \mathcal{X}
\]  
and an i.i.d.\ sequence $(\varepsilon_n)_{n \in \mathbb{N}}$, independent of both $(Y^p_n)_{n \in \mathbb{N}}$ and $(X_n)_{n\in\N}$, such that for all $y_1, \ldots, y_p \in \mathcal{Y}$, $x \in \mathcal{X}$, and $B \in \mathcal{B}(\mathcal{X})$,  
\[
\big[ Q(y_1) \cdots Q(y_p) \mathbf{1}_B \big](x)  
= \mathbb{E} \!\left[ \mathbf{1}_{\{ f^p(x, (y_1, \ldots, y_p), \varepsilon_0) \in B \}} \right].
\]

Moreover, for any $x_0 \in \mathcal{X}$ with $V(x_0) < \infty$, there exist constants $\kappa \in (0,1)$ and $c > 0$ such that  
\[
b(n) := \sup_{l \in \mathbb{N}} \mathbb{P} \!\left( Z_{lp, (l+n)p}^{X_{lp}} \neq Z_{lp, (l+n)p}^{x_0} \right)  
\le c \inf_{1 \le i \le q \le n} \left\{ r_i + \kappa^{n/q} + \alpha^{Y^p}(q-i) \right\}.
\]

It is straightforward to check that for the mixing coefficients
\[
\alpha^{Y^p}(n)\le \alpha^Y(np-p+1),\qquad n\ge 1,
\]
and hence
\[
b(n)\le c \inf_{1 \le i \le q \le n} \big\{ r_i + \kappa^{n/q} + \alpha^{Y}(p(q-i-1)+1) \big\}.
\]
By Lemma~\ref{lem:coupl_trans_mix}, for \(n\ge 2p\) and with the choice \(m=\lfrf{\frac{n}{2p}}\) there exist constants \(c'>0\) and \(\tilde{\kappa}\in(0,1)\) such that
\begin{align*}
\alpha^{X,Y}(n) &\le \alpha^Y \left(\lfrf{\frac{n}{2p}}p\right)
+ b\left(\lfrf{\frac{n}{p}}-1-\lfrf{\frac{n}{2p}}\right)
\\
&\le c' \inf \left\{r_i + \tilde{\kappa}^{n/q}+\alpha^Y (p(q-i-1)+1)
\middle| 1\le i\le q\le \lfrf{\frac{n}{p}}-1-\lfrf{\frac{n}{2p}}
\right\}.
\end{align*}
Since $\alpha^Y(p(q-i-1)+1)\leq \alpha^Y(q-i)$ as soon as $q-i\geq 1$ and $\widetilde{\kappa}^{n/q}+\alpha^Y(p(q-i-1)+1)$ does not converge to $0$ when 
$[n/p]-1-[n/(2p)]\leq q\leq n$ or when $q=i$,  one can simply change the constant $c'$ so that the previous infimum can be extended to the set of pair of integers $(i,q)$ such that $1\leq i\leq q\leq n$. The result in then valid for $n\geq 2p$ but it can be extended to any $n\geq 1$ using a suitable constant $c'>0$. This yields the desired result.

\end{proof}


\begin{lemma}\label{lem:pVbound}
Assume that Assumption~\ref{as:pkernel} holds. Furthermore, suppose that either  
\begin{enumerate}
\item \(X_0\) is independent of \((Y_n)_{n\in\Z}\), satisfies \(\E[V(X_0)]<\infty\), and the parametric kernel \(Q\) obeys  
\begin{equation}\label{eq:OneStepC}
[Q(y)V](x) \le C\big(V(x)+1\big)
\end{equation}
for some constant \(C>1\), or  
\item the two–sided process \(((X_n,Y_n))_{n\in\Z}\) is stationary.  
\end{enumerate}
Then  
\begin{equation}
\sup_{j\ge 0} \E[V(X_j)] < \infty.
\end{equation}
\end{lemma}
\begin{proof}
We again use that under Assumption~\ref{as:pkernel}, the sequence \((X_{np})_{n \in \mathbb{N}}\) forms a Markov chain in the random environment \((Y^p_n)_{n\in\N}\), which satisfies Assumption~\ref{as:main_assumptions}. Hence, by Lemma~\ref{mom}, we have
\begin{equation}\label{eq:supjE}
M:=\sup_{j \ge 0} \E[V(X_{jp})] < \infty.
\end{equation}

In the first case, if \(n = pm + r\) with \(0 \le r < p\), then by the tower rule and applying inequality~\eqref{eq:OneStepC} \(r\) times, we obtain
\begin{align*}
\E [V(X_n)] 
&= \E \big[ [Q(Y_{pm+1}) \cdots Q(Y_{pm+r}) V](X_{pm}) \big] \\
&\le C^r \big( \E [V(X_{pm})] + r \big) \\
&\le C^{p-1} (M + p - 1),
\end{align*}
which is an upper bound independent of \(n\). This completes the proof in the first case.

In the second case, the stationarity of \(((X_n,Y_n))_{n\in\Z}\) directly implies that 
\(M = \E[V(X_n)]\) is independent of \(n\), from which the desired claim follows trivially.
\end{proof}

\paragraph{Mixing properties of a simple linear autoregressive model.}

We consider linear vector autoregressive models of the form
\begin{equation}\label{eq:myVARX}
    X_{n+1} = A X_n + B Y_n + \varepsilon_{n+1},
\end{equation}
where $X_n, Y_n \in \mathbb{R}^d$ for $n \in \mathbb{N}$. 
Here $A$ and $B$ are fixed $d \times d$ matrices such that the spectral radius of $A$ satisfies $r(A) < 1$, while $\|A\| > 1$. 
The noise sequence $(\varepsilon_n)_{n \ge 1}$ is i.i.d.\ with values in $\mathbb{R}^d$, independent of the process $(Y_n)_{n \in \mathbb{N}}$. 
Finally, we assume that $X_0$ and $(\varepsilon_n)_{n \ge 1}$ are conditionally independent given $Y$, that $\mathbb{E}[|\varepsilon_0|] < \infty$, and that the distribution of $\varepsilon_0$ is absolutely continuous with respect to the Lebesgue measure, with density $\varphi_{\varepsilon_0}:\mathbb{R}^d \to (0,\infty)$ bounded away from zero on compact sets.

Choose the Euclidean norm as the Lyapunov function, that is, let $V(x)=\|x\|$. 
The parametric kernel corresponding to the recursion \eqref{eq:myVARX} is given by
\[
Q(y,x,H)=\mathbb{P}\big(Ax+By+\varepsilon_0\in H\big),
\qquad x,y\in\mathbb{R}^d,\; H\in\mathcal{B}(\mathbb{R}^d).
\]
Assume that $Q$ satisfies the drift condition \eqref{eq:drift} with some $\ga(.)$ and $K(.)$. Fix $y\in\mathbb{R}^d$ and a unit vector $x_*\in\mathbb{R}^d$ such that $\|Ax_*\|=\|A\|$. By the reverse triangle inequality, for any $t>0$ we obtain
\[
\mathbb{E}\!\left[\big|\,t\|A\|-\|By+\varepsilon_0\|\,\big|\right]
\;\le\;[Q(y)V](tx_*)=\mathbb{E}\big[\|A(tx_*)+By+\varepsilon_0\|\big]
\;\le\;t\,\gamma(y)+K(y).
\]
Dividing through by $t$ and letting $t\to\infty$, we conclude that $\gamma(y)\geq\|A\|>1$ for all $y\in\mathbb{R}^d$. Hence, part~A1 of Assumption~\ref{as:main_assumptions} cannot hold.

Let $M \ge 1$ be chosen large enough so that $\max\left(\|A\|,\|B\|, \mathbb{E}[|\varepsilon_0|]\right)\le M$. Since $r(A)<1$, there exists $p\in\mathbb{N}$ such that $\|A^p\|<1$. It is straightforward to verify that the $p$-step parametric transition kernel is given by
\[
Q((y_1,\ldots,y_p),x,H) 
= \mathbb{P}\!\left(A^p x + \sum_{k=1}^p A^{p-k}(B y_k + \varepsilon_{k+1}) \in H\right),
\quad x,y_1,\ldots,y_p \in \mathbb{R}^d,\; H\in\mathcal{B}(\mathbb{R}^d).
\]
Regarding the $p$-step drift condition, we obtain
\begin{align*}
[Q((y_1,\ldots,y_p))V](x) 
&= \mathbb{E}\!\left[\left\|A^p x + \sum_{k=1}^p A^{p-k}(B y_k + \varepsilon_{k+1})\right\|\right] \\
&\le \|A^p\|\,\|x\| + M^p\!\left(p+\sum_{k=1}^p \|y_k\|\right).
\end{align*}
Thus, the $p$-step drift condition \eqref{eq:pdrift} holds with 
\[
\gamma(y_1,\ldots,y_p)=\|A^p\|<1,
\quad
K(y_1,\ldots,y_p)=M^p\!\left(p+\sum_{k=1}^p \|y_k\|\right).
\]
In particular, if $\sup_{n\in\mathbb{N}}\E\left[\|Y_n\|\right]<\infty$, then part~A3 of Assumption~\ref{as:pkernel} is satisfied.

\medskip 
Now, let $x\in\mathbb{R}^d$ satisfy $\|x\|\le R$, let $y_1,\ldots,y_p\in\mathbb{R}^d$ be arbitrary, and let $H\in\mathcal{B}(\mathbb{R}^d)$ be an arbitrary Borel set. Define
\[
w = A^p x + \sum_{k=1}^{p-1} A^{p-k}(B y_k+\varepsilon_{k+1}) + B y_p,
\]
and choose $L=L(R,y_1,\ldots,y_p)>0$ such that $\mathbb{P}(\|w\|\le L)\ge 1/2$.  
By independence, we can estimate
\begin{align}\label{eq:alma}
\begin{split}
Q((y_1,\ldots,y_p),x,H) 
&= \mathbb{P}(w+\varepsilon_{p+1}\in H) \\
&\ge \mathbb{P}(\|w\|\le L)\times \inf_{\|v\|\le L}\mathbb{P}(v+\varepsilon_0\in H) \\
&\ge \tfrac{1}{2}\inf_{\|v\|\le L}\mathbb{P}(v+\varepsilon_0\in H).
\end{split}
\end{align}

Since the density $\varphi_{\varepsilon_0}$ is bounded away from zero on compact sets, for $\|v\|\le L$ we have
\begin{align*}
\mathbb{P}(v+\varepsilon_0\in H)
&= \int_{\mathbb{R}^d}\mathbf{1}_{\{v+z\in H\}}\varphi_{\varepsilon_0}(z)\,\mathrm{d}z \\
&\ge \int_{\|s\|\le 1}\mathbf{1}_{\{s\in H\}}\varphi_{\varepsilon_0}(s-v)\,\mathrm{d}s \\
&\ge C\,\mathrm{Leb}(H\cap B_{L+1}(0)),
\end{align*}
where $C>0$ is a normalization constant.  

Combining this with \eqref{eq:alma}, we obtain
\[
Q((y_1,\ldots,y_p),x,H)\;\ge\;\tfrac{C}{2}\,\mathrm{Leb}(H\cap B_{L+1}(0)),
\]
which shows that part~A4 of Assumption~\ref{as:pkernel} is also satisfied.

Note that for this system, $d_l \to 0$ exponentially fast. Hence, by Theorem~\ref{thm:extended}, it immediately follows that there exist constants $c''>0$ and $r\in (0,1)$ such that
$$
\alpha^{X,Y} (n)\le c'' \left[r^{n^{1/2}}+\alpha^Y \left(\lfrf{\frac{n^{1/2}}{2}}\right)\right].
$$
It is straightforward to verify that the parametric kernel $Q$ satisfies part~1 of Lemma~\ref{lem:pVbound}. Consequently, for any function $\Phi:\mathbb{R}^d \to \mathbb{R}$ satisfying
\[
|\Phi(x)| \le c' \,(1+\|x\|)
\]
for some constant $c'>0$, the moments are uniformly bounded:
\[
\sup_{n \in \mathbb{N}} \mathbb{E}\big[ |\Phi(X_n)| \big] < \infty.
\]
 
In Section~6 of \cite{lovas} and Section~4.2 of \cite{Truquet1}, the authors study linear systems in random environments that are considerably more complex than the ones considered here. Due to space constraints, we do not address these more general systems; however, the analysis presented above can be extended to them in a straightforward manner.

\section{Applications in Econometrics}\label{sec:econometrics}

In this section, we provide many applications of our results to time series analysis. For simplicity, we focus on stationary sequences for which statistical inference
is easier to investigate, though many limit theorems are still valid under non-stationary strong mixing sequences. 
Some standards textbooks such as \cite{Bradley2005}, \cite{Rio} or \cite{doukhan1995mixing} provide various convergence results for partial sums of $\alpha-$mixing processes.

\subsection{Mixing properties of a location-scale autoregressive model}

 Let $\left(\varepsilon_t\right)_{t\in\Z}$ be a sequence of i.i.d. random variables, independent of a stationary processes $(Y_t)_{t\in\Z}$ which takes values in $\R^d$.
 We assume here that
$$X_t=r\left(Y_{t-1},X_{t-1}\right)+\varepsilon_t\sigma\left(Y_{t-1},X_{t-1}\right),\quad X_0=x,$$
with $\E\left(\varepsilon_0\right)=0$, $\E\varepsilon_0^2=1$.
The two measurable mappings $r:\R^d\times \R\rightarrow \R$ and $\sigma:\R^d\times \R\rightarrow \R_+$ satisfy for $(x,y)\in\R^{d+1}$,
$$\left\vert r(y,x)\right\vert \leq a(y)\vert x\vert+b(y),$$
$$\left\vert \sigma(y,x)\right\vert\leq c(y)\vert x\vert+d(y),$$
where $a,b,c,d:\R^d\rightarrow \R_+$ are four measurable mappings.
Here we set $\gamma(y)=a(y)+c(y)$ and $K(y)=b(y)+d(y)$.
It is clear that Jensen's inequality leads to 
$$Q(y)V\leq \gamma(y)V+K(y),$$
with $V(x)=\vert x\vert$.
The strong mixing properties of the process $\left((X_t,Y_t)\right)_{t\geq 0}$ can be obtained from several assumptions on the process $(Y_t)_{t\in\Z}$ and on the mappings $\gamma$ and $K$. We give here two specific cases of transfert mixing properties which can be obtain directly from Lemma \ref{conc}, points $2.$ and $4.$, and (\ref{eq:trans_mix}) with $m=[n/2]$. Other assumptions are possible.

\begin{proposition}\label{illus}
Suppose that $(Y_t)_{t\in\Z}$ is a stationary process with a power decay for the mixing rates, i.e. $\alpha_Y(m)=O\left(m^{-a}\right)$ for some $a>0$.
Suppose furthermore that the probability distribution of $\varepsilon_0$ has a density $f$ with respect to the Lebesgue measure and which is lower-bounded by a positive constant on any compact interval. Suppose furthermore that for any $R>0$, $$\inf_{\vert x\vert\leq R}\sigma(y,x)>0.$$
\begin{enumerate}
\item 
Suppose that there exists $\gamma_+\in (0,1)$ such that $\gamma(Y_0)\leq \gamma_+$ a.s. and $\E K(Y_0)<\infty$.
Then the process $\left((X_t,Y_t)\right)_{t\in\Z}$ is $\alpha-$mixing with 
$$\alpha_{X,Y}(m)=O\left(\left(\log m\right)^a m^{-a}\right).$$
\item 
Suppose that $\gamma(Y_0)\in [0,1]$ a.s. with $\P\left(\gamma(Y_0)<1\right)>0$ and $\E K(Y_0)^k<\infty$ for some $k>1$. Then if $b:=a(k-1)/k>1$ the process $\left((X_t,Y_t)\right)_{t\in\Z}$ is $\alpha-$mixing with 
$$\alpha_{X,Y}(m)=O\left(\left(\log m\right)^{2b-1} m^{-b+1}\right).$$
\end{enumerate}
\end{proposition}

\paragraph{NoteS}
\begin{enumerate}
    \item
For instance, if $\varepsilon_0$ is a standard Gaussian and $d=1$, the model
$$X_t=a_0 Y_{t-1}+a_1 X_{t-1}\mathds{1}(X_{t-1}<r)+a_2 X_{t-1}\mathds{1}(X_{t-1}\geq r)+\varepsilon_t\sqrt{b_0^2+b_1^2 X_{t-1}^2+b_2^2 Y_{t-1}^2},$$
which is a version of the threshold autoregressive model of \cite{tong1980threshold} with exogenous covariate and conditional heteroscedasticity satisfies the assumptions of point $1.$ of Proposition \ref{illus} as soon as $b_0^2>0$, $\max\left(\vert a_1\vert,\vert a_2\vert\right)+\vert b_1\vert<1$ and $Y_0$ is integrable.
\item 
Autoregressive models for which $a(y)\in [0,1]$ can be arbitrarily close to $1$ has some interest to introduce more persistence in the model. For instance, if $r(y,x)=\left(1-\exp(-y)\right)x$ for positive $y$, the autoregressive process is close to a unit root behavior when the covariate takes large values.
\end{enumerate}

	\subsection{Parametric estimation with miss-specified models}
	
	In this section, we consider a parametric model of Markov chains in random environments as a model for time series with strictly exogenous covariates. Suppose that $Y_1,\ldots,Y_T$ is a sample of a stationary time series ${\bf Y}:=(Y_t)_{t\in\Z}$ taking values in $\R^d$. Conditionally on ${\bf Y}$, we assume that 
	$(X_t)_{t\in \Z}$ follows the dynamic of a MCRE, taking values in $\R^k$ and with transition kernel $Q\left(y,\cdot,\cdot\right)$, $y\in \R^d$. Let us consider a parametric model
	$$\left\{Q_{\theta}(y,x,A)=\int_A q_{\theta}(y,x,x')\mu(dx'): (\theta,x,y,A)\in \Theta\times \R^k\times\R^d\times \mathcal{E}\right\}$$ of transition kernels,
	where $\Theta$ is a measurable subset of $\R^m$, $\mathcal{E}$ denotes the Borel sigma field on $\R^k$ and $\mu$ is a reference measure.

	Let us define the estimator 
	$$\hat{\theta}_n=\arg\max_{\theta\in \Theta} n^{-1}\sum_{t=1}^n h_t(\theta),\quad h_t(\theta)=\log q_{\theta}\left(Y_{t-1},X_{t-1},X_t\right).$$
	We also denote by 
	$$\theta_0=\arg\max_{\theta\in \Theta}\E\left(h_0(\theta)\right),$$
	provided this quantity is well defined and unique.
	
	Here we do not assume that the model is well specified, i.e. there exists $\theta\in \Theta$ such that $Q=Q_{\theta}$. In a well specified scenario,  maximum likelihood estimators for smooth parametric Markov models are often shown to be consistent and asymptotic normal only using ergodic properties of the process. Indeed, for the asymptotic normality, the score process $\left(\nabla_{\theta}h_t(\theta_0)\right)_{t\in\Z}$ is a martingale difference and a central limit theorem for ergodic martingale differences can be used. In the miss-specified case, the later property is not necessarily true and an alternative is to use central limit theorems for mixing processes.
	
	In what follows, for a mapping $h:\Theta\rightarrow \R$, we denote by $\nabla h$ its gradient and $\nabla^{(2)}h$ its Hessian matrix.
	
	\begin{theorem}\label{param}
	Suppose that $(X_t)_{t\in\Z}$ is a stationary and ergodic process for which Assumption \ref{as:main_assumptions} is satisfied. Suppose furthermore that $\Theta$ is a compact and convex set and $\theta\mapsto \E h_0(\theta)$ has a unique argmax denoted by $\theta_0$.
	
	\begin{enumerate}
	    \item 
	    Suppose that $\theta\mapsto \E h_0(\theta)$ is continuous on $\Theta$, with $\sup_{\theta\in\Theta}\left\vert h_0(\theta)\right\vert$ integrable. Then $$\lim_{n\rightarrow \infty}\hat{\theta}_n=\theta_0\mbox{ a.s.}.$$
	    \item
	    Suppose additionally that $\theta_0$ is an interior point of $\Theta$,  $\theta\mapsto h_0(\theta)$ is two times continuously differentiable mapping, $\E\Vert \nabla h_0(\theta_0)\Vert^{2r}<\infty$ for some $r>2$,  $\sup_{\theta\in \Theta}\Vert \nabla^{(2)}h_0(\theta)\Vert$ has a finite moment and $M_{\theta_0}:=\E\nabla^{(2)}h_0(\theta_0)$ is an invertible matrix.
        Suppose furthermore that $\sum_{k\geq 1}k^{\frac{1}{r-1}}\alpha_{X,Y}(k)<\infty$. Then
	    $$\sqrt{n}\left(\hat{\theta}_n-\theta_0\right)\Rightarrow \mathcal{N}_d\left(0,M_{\theta_0}^{-1}N_{\theta_0}M_{\theta_0}^{-1}\right),$$
	    where 
	    $$N_{\theta_0}=\sum_{\ell\in\Z}\cov\left(\nabla h_0(\theta_0),\nabla h_{\ell}(\theta_0)\right).$$ 
	    	\end{enumerate}
		\end{theorem}

\paragraph{Notes}
\begin{enumerate}
    \item
If $Q=Q_{\theta'}$ for some $\theta'\in\Theta$, it is easy to show that $\theta'$ is an arg max of $\theta\mapsto \E h_0(\theta)$. Under the assumptions of Theorem \ref{param}, we have $\theta'=\theta_0$. Moreover
$$\E\left[\nabla q_t(\theta_0)\vert {\bf Y},X_{t-1}\right]=\int \dot{q}_{\theta_0}(Y_{t-1},X_{t-1},x')\mu(dx')=0,$$
where $\dot{q}_{\theta_0}$ is the gradient vector of $\theta\mapsto h_{\theta}$ at point $\theta_0$. The last equality corresponds to the inversion between derivative and integral. The score $\left(\nabla h_t(\theta_0)\right)_{t\in\Z}$ is then a martingale difference and $N_{\theta_0}$ simplifies to $\var\left(\nabla h_0(\theta_0)\right)$. However, under model miss-specification, these properties are no longer valid, though $\nabla h_t(\theta_0)$ has still zero mean. Then $N_{\theta_0}$ has a more complicated form due to non-vanishing covariances for the first-order derivatives of the mappings 
$\theta\mapsto h_t(\theta)$.
	\item 
Suppose that $\alpha_Y(n)=O\left(n^{-a}\right)$ for some $a>0$. Using Lemma \ref{conc} and Inequality (\ref{eq:trans_mix}) with $m=[n/2]$, we get $\alpha_{X,Y}(n)=O\left(n^{-b+1}\right)$ for any $b\in (0,a)$. We then deduce  that condition $\sum_{k\geq 1}k^{\frac{1}{r-1}}\alpha_{X,Y}(k)<\infty$ holds true as soon as $a>2+\frac{1}{r-1}$.

    \end{enumerate}
\paragraph{Example.} 
Suppose that $Q_{\theta}(y,x,\cdot)$ is the Poisson distribution with intensity $\lambda_{\theta}(y,x)=\theta_1+\theta_2 x+\theta_3^Ty$ where $\theta=(\theta_1,\theta_2,\theta_3)$ is an element of $\R_+\times \R_+\times \R_+^d$ with $\theta_1>0$. The true kernels can correspond to a different count distribution, to a Poisson distribution with a different intensity (e.g. $\lambda(y,x)$ is not linear in $(x,y)$) or even to a Poisson distribution with intensity	
$$\lambda(y',x)=\eta_1+\eta_2 x+\eta_3^T y'$$
but with $y'$ a realization of a process $Y_t'=(Y_t,Z_t)$ taking values in $\R^m\times \R^{m'}$. The latter case corresponds to missing covariates (i.e. the process $(X_t)_{t\in\Z}$ is a MCRE but defined conditionally on $Y'$ and not only $Y$).

\paragraph{Proof of Theorem \ref{param}.}

The proof is based on standard arguments. For asymptotic normality, the regularity assumptions guarantee that $\nabla q_t(\theta_0)$ is integrable and has zero mean.
One can then apply the central limit theorem for $\alpha-$mixing variables to 
the partial sums $U_n:=n^{-1/2}\sum_{t=1}^n \nabla q_t(\theta_0)$. Under our assumptions, the limiting distribution of $U_n$ is Gaussian with mean $0$ and variance matrix $N_{\theta_0}$. See Theorem $1$ in \cite{doukhan1994functional} and remark $4$ just after the statement of the theorem.

	\subsection{Non-parametric kernel regression}
	
	In this section, we consider the non-parametric estimation of the regression function 
	$$r(x,y)=\E\left[X_t\vert X_{t-1}=x,Y_{t-1}= y\right]=\int x' Q(y,x,dx')$$
	for a stationary Markov chain in random environments satisfying our assumptions. In what follows, for $t\in\Z$, we set $Z_t=(X_t,Y_t)$ which is a random vector taking values in $\R\times \R^d$. 
	 Nadaraya-Watson kernel estimators are widely popular for such problems. This estimator, denoted by $\hat{r}$ in what follows, is given by the formula
	$$\hat{r}(x,y)=\frac{\displaystyle\sum_{t=2}^n X_t K\left(\frac{x-X_{t-1}}{h},\frac{y-Y_{t-1}}{h}\right)}{\displaystyle\sum_{t=2}^n K\left(\frac{x-X_{t-1}}{h},\frac{y-Y_{t-1}}{h}\right)},$$
	where $K:\R^{d+1}\rightarrow \R_+$ is a probability density and $h>0$ is a bandwidth parameter. For simplicity, we will assume that the kernel $K$ is Lipschitz, symmetric and has compact support. These assumptions are satisfied by many kernels, such as the triangular kernel or the Epanechnikov kernel.

	Though many contributions considered convergence rates for such estimators, as in \cite{hansen2008uniform}, there do not exist precise assumptions that can be direclty checked on a specific DGP when the predictors contain both lag-values of the response and exogenous covariates. We will fill this gap in this section when the dynamic is given by a MCRE in a stationary and strongly mixing environment.

	We will use the following assumptions. For simplicity of the statements, we limit our results to the long-term contractivity condition for the drift parameters. More general results are possible, using suitable transfer mixing properties from the process $Y$ to the pair $(Y,X)$ such as that given in Lemma \ref{conc}.
	
	\begin{description}
	\item {\bf E1} There exists $s>2$ such that, setting $V(x)=\vert x\vert^s$, for any $y\in\R^d$, 
	$$Q(y)V\leq \gamma(y)V+K(y),$$
	for some mappings $\gamma:\R^d\rightarrow \R_+$ and $K:\R^d\rightarrow [1,\infty)$ such that $\E\left[\gamma(Y_0)+K(Y_0)\right]<\infty$ and 
	$$\limsup_{t\rightarrow \infty}\E^{1/t}\left[\gamma(Y_t)\cdots \gamma(Y_1)K(Y_0)\right]<1.$$
	\item {\bf E2} The minorization condition is satisfied for the function $V$ defined above.
	\item {\bf E3} The environment is strongly mixing with $\alpha_Y(m)\leq C m^{-\beta}$ for some 
	$$\beta>\frac{1+(s-1)(1+d+d/q)}{s-2}\mbox{ and } q>0.$$
	\item {\bf E4}
	The probability distribution of $Z_0$ has a density w.r.t. the Lebesgue measure, denoted by $f_{Z_0}$ and such that 
	$$\sup_{z\in\R^{d+1}}\E\left[1+\vert X_1\vert^s\vert Z_0=z\right]f_{Z_0}(z)<\infty.$$
	\item {\bf E5}
	For any positive integer $j$, the probability distribution of the pair $(Z_0,Z_j)$ admits a density denoted by $f_{Z_0,Z_j}$. Moreover, there exists a positive integer $j^*$ and a positive constant $B$ such that for any integer $j\geq j^{*}$
	$$\sup_{z_0,z_j\in \R^{d+1}}\E\left[\left\vert X_1X_{j+1}\right\vert \vert Z_0=z_0,Z_j=z_j\right]\cdot f_{Z_0,Z_j}(z_0,z_j)\leq B.$$ 
	\end{description}
	
	The following result is a direct consequence of Theorem $8$ in \cite{hansen2008uniform}. 
	
	\begin{theorem}
	\label{cvunif}
Suppose that Assumptions {\bf E1--E5} are satisfied. Suppose furthermore that the second derivatives of $f_{Z_0}$ and $r\times f_{Z_0}$ are bounded and uniformly continuous and that $\inf_{\vert z\vert\leq c}f_{Z_0}(z)>0$ for some $c>0$.
Then we get
$$\sup_{\vert z\vert\leq c}\left\vert \hat{r}(z)-r(z)\right\vert=O_P\left(\left(\frac{\log n}{n}\right)^{\frac{2}{d+5}}\right).$$
	\end{theorem}
	
\paragraph{Proof of Theorem \ref{cvunif}.}

It is simply needed to check Assumption $2$ and condition $(10)$ of Theorem $8$ given in \cite{hansen2008uniform}. First, Assumptions {\bf E1} and {\bf E2} guaranty the existence of a unique stationary solution $(X_t)_{t\in\Z}$ such that $\E\vert X_0\vert^s<\infty$. Assumption {\bf E3} guarantees that $\alpha_{X,Y}(m)=O\left(m^{-\kappa \beta}\right)$ for any $\kappa\in (0,1)$. Choosing $\kappa$ close to $1$, we obtain that $\kappa \beta > \frac{1+(s-1)(1+d+d/q)}{s-2}$ which is precisely condition $(10)$ in \cite{hansen2008uniform}. Conditions $(5),(6)$ and $(7)$ given in Assumption $2$ of  \cite{hansen2008uniform} follows directly from {\bf E4-E5}.$\square$

Checking Assumptions {\bf E4-E5} as well as other regularity conditions on marginal densities will depend on specific properties of the transition kernels $Q(y,\cdot,\cdot)$. It is difficult to provide general guidelines. 
For simplicity, we will study more specifically the model
	\begin{equation}\label{autoreg}
	X_t=r\left(X_{t-1},Y_{t-1}\right)+\sigma\left(X_{t-1},Y_{t-1}\right)\varepsilon_t,\quad t\in\Z,
	\end{equation}
	when the $\varepsilon_t's$ are i.i.d., centered and mutually independent from the $Y_s$'s.
	
For this specific model, we will give sufficient conditions for getting {\bf E1-E2-E4-E5}. Our aim is to only make assumptions on the transition kernels and the environment. We will use the following assumptions.

\begin{description}
\item{\bf F1}
The volatility parameter satisfies
	$$0<\sigma_{-}:=\inf_{(x,y)\in \R^{d+1}}\sigma(x,y)\leq \sigma_+:=\sup_{(x,y)\in\R^{d+1}}\sigma(x,y)<\infty.$$
\item{\bf F2}
$\varepsilon_0$ has a probability distribution with a continuous density, positive everywhere and denoted by $f_{\varepsilon}$. Moreover there exists a $s'>2$ such that
$$\E\vert\varepsilon_0\vert^{s'}<\infty,\quad \sup_{u\in \R}\left(1+\vert u\vert^{s'}\right)f_{\varepsilon}(u)<\infty.$$
\item{\bf F3}	
	The following sub-linearity condition holds true, 
	$$\left\vert r(x,y)\right\vert\leq a(y)\vert x\vert+b(y),$$
	where $a,b:\R^d\rightarrow \R_+$ are measurable functions such that $b\geq 1$. Moreover, for some $s>2$, we have 
	$$\E\left[a(Y_t)^{s'}\cdots a(Y_1)^{s'} b(Y_0)^{s'}\right]\leq C\kappa^t,$$
	for a pair of constants $(C,\kappa)\in (0,\infty)\times (0,1)$.
\item{\bf F4}
The distribution of $Y_t$ conditionally on $(Y_{t-s})_{s\geq 1}$ has a probability density denoted by $g_t$ and which is uniformly bounded, i.e. there exists $M>0$ such that $\sup_{y\in\R^d}g_t(y)\leq M$ a.s. Moreover, setting 
$$\alpha_j=a(Y_j)+a(Y_j)b(Y_{j-1})+\cdots+a(Y_j)\cdots a(Y_2)b(Y_1),$$
$$\beta_j=a(Y_j)\cdots a(Y_1),\quad K_j=b(Y_j)b(Y_0)+\alpha_j b(Y_0)+\beta_j b(Y_0)^2,$$
$$L_j=b(Y_j)a(Y_0)+\alpha_j a(Y_0),\quad M_j=\beta_j a(Y_0)^2,$$
$$p_j(y)=\E\left[K_j\vert Y_0=y\right],\quad q_j(y)=\E\left[L_j\vert Y_0=y\right],\quad r_j(y)=\E\left[M_j\vert Y_0=y\right],$$
we have 
\begin{equation}\label{cond1}
\sup_{j\geq 3}\sup_{y\in\R^d}\left\{p_j(y)\E\left[g_0(y)\right]+\E\left[\left(p_j(y)g_0(y)\right)^{\frac{s'}{s'-1}}\right]+\E\left[\left(r_j(y)g_0(y)\right)^{\frac{s'}{s'-2}}\right]\right\}<\infty.
\end{equation}
Moreover, there exists $s\in (2,s')$ such that
\begin{equation}\label{cond2}
\sup_{y\in\R^d}\left\{ b(y)^s E\left[g_0(y)\right]+\E\left[\left(a(y)^sg_0(y)\right)^{\frac{s'}{s'-s}}\right]\right\}<\infty.
\end{equation}
\end{description}

\begin{proposition}\label{conseq}
 Assumptions {\bf F1-F2-F3-F4} entail Assumptions {\bf E1-E2-E4-E5}.
\end{proposition}

Let us comment on these assumptions. 
\begin{enumerate}
    \item 
    Assumption {\bf F1} means that the conditional variance 
    $$\var\left(X_t\vert X_{t-j},Y_{t-j};j\geq 1\right)=\sigma^2\left(X_{t-1},Y_{t-1}\right)$$
    is lower and upper bounded. Assuming such a lower bound is classical. The existence upper bound is only used to simplify the statement of other assumptions but it can be removed.
    \item
    Assumption {\bf F2}  requires positiveness and continuity of $f_{\varepsilon}$, which are classical conditions for studying autoregressive models with Markov chain techniques. In particular, under these conditions, any compact interval is a small set. We also assume that this density is bounded and satisfies additional growth and moment conditions.
    \item
    Assumption {\bf F3} imposes a sub-linear growth of the regression function with coefficients depending on the environment as well as a long-time contractivity for getting existence of a stationary MCRE and with a moment of order $s'>2$.
    \item
    Assumption {\bf F4} is more difficult to check. The existence of a uniformly bounded conditional density is satisfied for instance for any stationary and non-anticipative solution of the recursions
    $$Y_t=m\left(Y_{t-1},Y_{t-2},\ldots\right)+v\left(Y_{t-1},Y_{t-2},\ldots\right)\eta_t,$$
    where the $\eta_t'$s are i.i.d. with an absolutely continuous probability distribution with a bounded density and a lower bounded volatility function $v$. Such a model specification covers many models used in Econmetrics, ARMA and GARCH models among other. The two other conditions (\ref{cond1}) and (\ref{cond2}) are more technical to check. They are satisfied when $a(y)\leq a\in (0,1)$ and $b(y)\leq b\in (0,\infty)$ for all $y\in\R^d$.
   When the mapping $b(\cdot)$ is not necessarily bounded, a control of $\E\left[b(Y_j)\vert Y_0=y\right]$ is needed. 
   \item
    Under our assumptions, one can show that $f_{Z_0}(z)=\E\left[f_0(x)g_0(y)\right]$, where $f_0$ is the conditional density of $X_0$ given $(Y_s)_{s\leq -1}$ and $g_0$ is defined in {\bf F4}. As shown in the proof of Proposition \ref{conseq}, $f_0$ is continuous a.s. and  $f_0(x)>0$ a.s. for any $x\in\R$.  If we assume that $g_0$ is continuous a.s. and that for any point $y\in \R^d$ such that $\vert y\vert\leq R$, $g_0(y)$ is not equal to zero a.s., then for any $e>0$, $\inf_{\vert x\vert\leq e,\vert y\vert\leq R}f_{Z_0}(x,y)>0$,  a condition which is required for applying Theorem \ref{cvunif}. Additional smoothness conditions on $f_{Z_0}$, as required in the statement of Theorem \ref{cvunif}, can be obtain from that on $f_0$ and $g_0$ and differentiability of $f_0$ can be obtained directly from differentiability properties of $f_{\varepsilon}$, in addition to suitable moment conditions. For simplicity we do not provide additional specific conditions.
    \end{enumerate}

\paragraph{Proof of Proposition \ref{conseq}.}

Setting for $(x,x',y)\in \R^{d+2}$,
	$$f(y,x',x)=\frac{1}{\sigma(x',y)}f_{\varepsilon}\left(\frac{x-r(x',y)}{\sigma(x',y)}\right),$$
	we get 
	$$Q(y,x,A)=\int_A f(y,x,x')dx',\quad (y,x,A)\in \R^d\times\R\times\mathcal{B}(\R).$$

Let us first check {\bf E1-E2}. Setting $\gamma(y)=(1+\epsilon)^{s'-1}a(y)^{s'}$, where $\epsilon>0$ is such that $\widetilde{\kappa}:=(1+\epsilon)^{s'-1}\kappa<1$, and $V(x)=\vert x\vert^{s'}$, we get
	$$Q(y)V\leq \gamma(y)V+K(y),$$
	with 
	$$K(y)=\left(\frac{1+\epsilon}{\epsilon}\right)^{s'-1}\E\left(b(y)+\sigma_{+}\vert\varepsilon_1\vert\right)^{s'}$$
	and it is straightforward to show that {\bf E1} is verified. The minorization condition {\bf E2} is also satisfied since if $\vert x'\vert\leq R^{1/{s'}}$, we have
	$$Q(y,x',A)\geq \sigma_{+}^{-1}\inf\left\{f_{\varepsilon}(u): \vert u\vert \leq \frac{1+\sup_{\vert w\vert\leq R^{1/s'}}\left\vert r(w,y)\right\vert}{\sigma_{-}}\right\}\cdot\lambda(A),$$
	where $\lambda$ denotes the Lebesgue measure over $[0,1]$. We deduce {\bf E2} from positivity and continuity of $f_{\varepsilon}$ which ensure that this density is lower bounded by a positive constant on each compact interval.
	
	We next check {\bf E4}. We recall that for a stationary solution obtained from our assumptions, the exists a randomly invariant sequence of random measures $\left(\pi_t\right)_{t\in\Z}$ such that $\pi_t=\pi_{t-1}P_{t-1}$ where we set 
	$P_t(x,A)=Q\left(Y_t,x,A\right)$. Moreover, $\pi_t$ is measurable w.r.t. $\sigma\left(Y_s:s\leq t-1\right)$. Note that $\pi_t$ has automatically a density w.r.t. the Lebesgue measure. This random density denoted by $f_t$ is given by 
	$$f_t(x)=\int \pi_{t-1}(dx')f(Y_{t-1},x',x).$$
	Note that the assumptions on $f_{\varepsilon}$ and $\sigma$ guaranty that $f_t$ is continuous and positive a.s.
	We have the expression $f_{Z_t}(x,y)=\E\left(f_{t-1}(x)g_{t-1}(y)\right)$,
        where $g_t$ denotes the density of $Y_t$ given $(Y_{t-s})_{s\geq 1}$.
        Since $\Vert f_{t-1}\Vert_{\infty}\leq \Vert f_{\varepsilon}\Vert_{\infty}/\sigma_{-}$ and $g_{t-1}$ is bounded from {\bf A4}, the density of $Z_t$ is automatically bounded. The other condition to check is 
        $$\sup_{z\in \R^{d+1}}\E\left[\vert X_1\vert^s\vert Z_0=z\right]f_{Z_0}(z)<\infty,$$
        for some $s>2$. Using our sub-linearity condition, it is sufficient to show that
        $$\sup_{(x,y)\in\R^{d+1}}\left\{a(y)^s\vert x\vert^s+b(y)^s\right\}\E\left[f_0(x)g_0(y)\right]\infty.$$
It is possible to show that there exists a positive constant $C$ such that 
$$\sup_{x\in\R}\vert x\vert^sf_0(x)\leq C\left(1+\E\left[\left\vert r(X_{-1},Y_{-1}\right\vert^s\vert Y_{-1},Y_0,\ldots\right]\right).$$
Moreover
$$\E\left[\left\vert r(X_{-1},Y_{-1})\right\vert^sg_0(y)\right]\leq \E^{s/s'}\left[\left\vert r(X_{-1},Y_{-1})\right\vert^{s'}\right]\cdot \E^{\frac{s'-s}{s'}}\left[g_0(y)^{\frac{s'}{s'-s}}\right].$$
The required condition then follows from (\ref{cond2}).

Finally, we have to check {\bf E5}. Set
$$U:=\sup_{j\geq 3}\sup_{(z_0,z_j)}\E\left[\left\vert X_{j+1}X_1\right\vert \vert Z_0=z_0,Z_j=z_j\right]\ell_j(z_0,z_j)<\infty,$$
where $\ell_j$ denotes the density of the pair $(Z_0,Z_j)$. We are going to check the bound
\begin{equation}\label{eff}
U\leq C\sup_{j\geq 3}\sup_{(x_0,y_0)\in \R^{d+1}}\left\{p_j(y_0)+q_j(y_0)\vert x_0\vert+r_j(y_0)x_0^2\right\}f_{Z_0}(x_0,y_0),
\end{equation}
where $C$ denote a positive constant and the mappings $p_j,q_j,r_j$ are defined in {\bf F4}. To prove (\ref{eff}), we simplify our notations. Then, for $z_j=(x_j,y_j)$, we have
$$\E\left[\left\vert X_{j+1}X_1\right\vert \vert Z_0=z_0,Z_j=z_j\right]=
L(z_j)\E\left[X_1\Vert Z_0=z_0,Z_j=z_j\right],$$
where $L(z_j)=\E\left[\left\vert X_{j+1}\right\vert \vert Z_j=z_j\right] X_0$
 We denote by $g(\cdot\vert y_j^{-})$ the density of $Y_j$ conditionally on $Y_s=y_s$ for $s\leq j-1$, which is uniformly bounded from {\bf F4}, and 
by $f(\cdot\vert y_0^{-})$ the density of $X_0$ conditionally on $Y_s=y_s$ for $s\leq 0$. Moreover, we denote by $f_j(y_{1:j-1};x_1,x_j)$ the density of $(X_1,X_j)$ conditionally on $Y=y$, evaluated at point $(x_1,x_j)$. 
For $z_0=(x_0,y_0)$ and $z_j=(x_j,y_j)$, we also set 
$$I(y_{1:j},y_0^{-},x_j)=\int \vert x_1\vert L(z_j)f_j\left(y_{1:j-1};x_1,x_j\right)f(y_0,x_0,x_1)f\left(x_0\vert y_0^{-}\right)dx_1.$$
If $\mu$ denotes the probability measure of $Y_0^{-}=(Y_{-1},Y_{-2},\ldots)$, we
have 
\begin{eqnarray*}
&&\E\left[\left\vert X_{j+1}X_1\right\vert \vert Z_0=z_0,Z_j=z_j\right]\ell_j(z_0,z_j)\\
&=&\int I(y_{1:j},y_0^{-},x_j)_0,x_0,x_1)\prod_{i=0}^j
g(y_i\vert y_i^{-})d y_1\cdots dy_{j-1}\mu(d y_0^{-}).
\end{eqnarray*}
To simplify our notations, for a fixed $y\in (\R^d)^{\Z}$ and $j\in \N$, we simply denote $a(y_j),b(y_j)$ by $a_j,b_j$. Setting $\overline{b}_j=b_j+\E\vert \varepsilon_0\vert$, we have $\overline{b}_j\leq C b_j$. Moreover, 
$$L(z_j)\leq a_j\vert x_j\vert+\overline{b}_j.$$
Moreover,
\begin{eqnarray*}
I(y_{1:j},y_0^{-},x_j)_0,x_0,x_1)&\leq& a_j \int \vert x_1 x_j\vert f_j(y_1,\ldots,y_{j-1};x_1,x_j)f(y_0,x_0,x_1)f\left(x_0\vert y_0^{-}\right)dx_1\\
&+& \overline{b}_j\int \vert x_1\vert f_j(y_{1:j-1};x_1,x_j)f(y_0,x_0,x_1)f\left(x_0\vert y_0^{-}\right)dx_1\\
&:=& A+B
\end{eqnarray*}
We have the bound 
\begin{eqnarray*}
\vert x_j\vert f_j\left(y_{1:j-1};x_1,x_j\right)&=& \int \vert x_j\vert f(y_{j-1},x_{j-1},x_j)f_{j-1}\left(y_{1:j-2};x_1,x_{j-1}\right)dx_{j-1}\\
&\leq& C\left\{1+\int r(x_{j-1},y_{j-1})f_{j-1}\left(y_{1:j-2};x_1,x_{j-1}\right)dx_{j-1}\right\}\\
&=& C\left\{a_{j-1}\E\left[\left\vert X_{j-1}\right\vert\vert Y=y,X_1=x_1\right]+b_{j-1}\right\},
\end{eqnarray*}
where we used the fact that $\sup_{u\in\R}\vert u\vert f_{\varepsilon}(u)<\infty$ from {\bf F2}. We also use the bound
$$\E\left[\left\vert X_{j-1}\right\vert\vert Y=y,X_1=x_1\right]\leq \beta_{j-2}\vert x_1\vert+d_{j-2},$$
where
$$d_{j-2}=b_{j-2}+\sum_{s=2}^{j-2}a_{j-2}\cdots a_{j-s}b_{j-s-1}.$$
We then obtain
\begin{eqnarray*}
A&\leq& C\int \vert x_1\vert\left[a_j+a_j b_{j-1}+\beta_j\vert x_1\vert+a_j a_{j-1}d_{j-2}\right]f(y_0,x_0,x_1)dx_1f_0\left(x_0\vert y_0^{-}\right)\\
&\leq& C\int \left(\alpha_j\vert x_1\vert+\beta_j x_1^2\right)f(y_0,x_0,x_1)d x_1f\left(x_0\vert y_0^{-}\right).
\end{eqnarray*}
Using the inequality
$$\E\left[X_1^2\vert Y=y,X_0=x_0\right]\leq 2 a_0^2 x_0^2+\overline{b}_0^2,$$
we get
$$A\leq Cf\left(x_0\vert y_0^{-}\right)\left(\alpha_j b_0+\beta_jb_0^2+\alpha_j a_0\vert x_0\vert+\beta_j a_0^2 x_0^2\right).$$
Moreover using the fact that $f_j$ is uniformly bounded by $\sup_{u\in\R}f_{\varepsilon}(u)/\sigma^{-}$, we get
$$B\leq C b_j L(z_0)f\left(x_0\vert y_0^{-}\right).$$
Collecting all the previous bounds, and using the expression 
$$f_{Z_0}(x_0,y_0)=\int f\left(x_0\vert y_0^{-}\right)g\left(y_0\vert y_0^{-}\right)\mu(d y_0^{-}),$$
we get (\ref{eff}). Using the fact that $f\left(x_0\vert y_0^{-}\right)$ is uniformly bounded, the inequalities
$$\vert x_0\vert^i f\left(x_0\vert y_0^{-}\right)\leq C\left(1+\E\left[\left\vert r(X_{-1},Y_{-1})\right\vert^i \vert Y_0^{-}=y_0^{-}\right]\right)$$
for $i=1,2$, {\bf E5} follows from Holder's inequality and (\ref{cond1}). The proof is now complete.$\square$

	\section{Stochastic gradient Langevin algorithm}\label{sec:langevin}
	
	We consider a recursive stochastic algorithm known as \textit{Stochastic Gradient Langevin Dynamics (SGLD)}, originally proposed by Welling and Teh~\cite{welling2011bayesian}, and defined by the iteration
	\begin{align}\label{eq:SGLDit}
		\begin{split}
		X_{n+1} &= X_n - \lambda H(X_n, Y_n) + \sqrt{\frac{2\lambda}{\beta}} \, \xi_{n+1},\quad n\ge 0,
		\end{split}
	\end{align}
	where $H : \mathbb{R}^d \times \mathbb{R}^m \to \mathbb{R}^d$ is called update function, $\lambda > 0$ is the step size, and $\beta>0$ is the so-called inverse temperature parameter. Furthermore, $(\xi_n)_{n \geq 1}$ is an i.i.d. sequence of standard $d$-dimensional Gaussian random variables, and $(Y_n)_{n \in \N}$ is an $\mathbb{R}^m$-valued process that models the data stream fed into the algorithm. We assume that the processes $(\xi_n)_{n \geq 1}$ and $(Y_n)_{n \in\N}$ are independent of each other, and $X_0=x_0\in\R^d$ is a deterministic initial value.
	
	The SGLD algorithm and its variants have demonstrated remarkable efficiency in finding global minima of potentially complex high-dimensional objective functions provided suitable regularity conditions hold for the gradient (see Welling and Teh \cite{welling2011bayesian}, and Raginski et al. \cite{raginsky} among others). The fundamental idea behind SGLD is that the problem of finding the global minimum of a possibly non-convex objective function, denoted by $U$, is intrinsically linked to the task of sampling from a target distribution characterized by a density proportional to $e^{-\beta U(x)}$, which is concentrated to the global minimum of $U$ when $\beta>0$ is large. 
	
	More precisely, suppose that $(Y_n)_{n \in \mathbb{Z}}$ is a stationary process and $H$ is an unbiased estimator of the gradient of a differentiable function $U:\mathbb{R}^d \to \mathbb{R}_+$, that is,  
\[
\mathbb{E}[H(x, Y_0)] = \nabla U(x), \quad x \in \mathbb{R}^d.
\]  
Under suitable regularity assumptions, for sufficiently small step size $\lambda$ and large $n$, the distribution of $X_n$ is expected to approximate closely the probability measure
\begin{equation}\label{eq:pibeta}
  \pi_{\beta}(A) \propto \int_A e^{-\beta U(x)} \, dx, \quad A \in \mathcal{B}(\mathbb{R}^d),
\end{equation}
which coincides with the invariant distribution of the Langevin stochastic differential equation
\begin{equation}\label{eq:LangevinSDE}
  d X_t = -\nabla U(X_t)\, dt + \sqrt{2\beta^{-1}}\, dB_t, \qquad t \ge 0,
\end{equation}
as $t \to \infty$, where $(B_t)_{t \ge 0}$ denotes the standard Brownian motion.

Consequently, if $U$ has a unique minimum at $x^* \in \mathbb{R}^d$, for large $\beta$, we can expect that $\mathbb{E}[X_n] \approx x^*$, provided $n$ is sufficiently large and $\lambda$ is small enough. This concept is discussed in details, for instance, in \cite{welling2011bayesian,6}. When the gradient is not estimated, the SGLD iteration \eqref{eq:SGLDit} coincides with the Euler--Maruyama discretization of the Langevin SDE \eqref{eq:LangevinSDE}. Consequently, even if $\law{X_n}$ converges to a limiting distribution as $n \to \infty$, due to the discretization error this limit, denoted by $\pi_{\beta,\lambda}$, differs from $\pi_{\beta}$. The bias introduced by the discretization in the Wasserstein-1 distance is of order $O(\sqrt{\lambda})$ under suitable conditions, as detailed in \cite{5author}.
	
The literature on SGLD is extensive but most of the available studies, assume that $Y_n$, $n\in\Z$ are i.i.d. However, this does not hold true in several applications, prominently in the case of financial times series, processing natural languages or when the data come from sensors.
	
	The first work that explicitly deals with non-i.i.d. data streams $(Y_n)_{n \in \mathbb{N}}$ is Section 3.1 of \cite{6}. It only assumes that $(Y_n)_{n \in \mathbb{N}}$ has the L-mixing property. Under the assumption that $U$ is continuously differentiable, $\nabla U$ is Lipschitz, and $U$ is strongly concave, the authors obtained an upper bound for the Wasserstein-2 distance between the law of the iterates and the target distribution, with constants depending explicitly on the Lipschitz and strong convexity constants of the potential, as well as the dimension of the space.
	
	The article \cite{5author}, similar to \cite{6}, examines the convergence of the distribution of iterates under the assumption that the data stream $(Y_n)_{n \in \mathbb{N}}$ is conditional L-mixing. 
	The main result of this paper is that, instead of assuming a strong convexity condition, it imposes the following more weaker dissipativity condition:  
	\begin{equation}\label{eq:dissipativity}
		\langle x, H(x, y) \rangle \geq \Delta \|x\|^2 - b,
	\end{equation}  
	where $\Delta, b > 0$ are fixed constants. Authors proved that $\law{X_n}$ converges to the invariant distribution $\pi_{\beta, \lambda}$ in the Wasserstein-1 metric at an exponential rate. Furthermore, as previously mentioned, they established that the distance between $\pi_{\beta}$ and $\pi_{\beta, \lambda}$ in the $W_1$ metric can be bounded by the square root of the step size.
	
	In situations where the random variables $Y_{n}$, $n \in \mathbb{Z}$, are independent, $(X_{n})_{n \in \mathbb{N}}$ is a Markov chain. The article \cite{lovas} was the first to point out that, in the more general scenario of dependent data streams, $(X_{n})_{n \in \mathbb{N}}$ can still be treated as a Markov chain \emph{in a random environment}. Easily seen that the parametric kernel 
	\begin{equation}\label{eq:SGLD_Q}
		Q(y, x, A) = \mathbb{P} \left(x - \lambda H(x, y) + \sqrt{\frac{2\lambda}{\beta}} \xi_0 \in A \right)
	\end{equation}  
	corresponds to the iteration \eqref{eq:SGLDit},
	where $\xi_0$, as in the recursion \eqref{eq:SGLDit}, is a standard $d$-dimensional Gaussian random variable. 
	
	In \cite{lovas}, the authors analyzed the SGLD iteration in the context of strongly stationary and bounded data streams $(Y_n)_{n \in \mathbb{Z}}$. They assumed that $H$ grows at most linearly, $\Delta: \mathbb{R}^m \to \mathbb{R}$ is a bounded and measurable function in the dissipativity condition \eqref{eq:dissipativity}, and that  
	\begin{equation}\label{eq:GEcond}
	\Gamma(\alpha) := \lim_{n \to \infty} \frac{1}{n} \ln \mathbb{E} e^{\alpha (\Delta(Y_1) + \dots + \Delta(Y_n))}
	\end{equation}
	exists for all $\alpha \in (-\eta, \eta)$ for some $\eta > 0$, with $\Gamma$ being continuously differentiable on $(-\eta, \eta)$.  
	Under these assumptions, the total variation distance satisfies the bound  
	\begin{equation}\label{eq:SGLD:oldrate}
	\dtv (\mathcal{L}(X_n), \pi_{\beta, \lambda}) \leq c_1 e^{-c_2 n^{1/3}},
	\end{equation}
	where $c_1, c_2 > 0$ are appropriate constants. Moreover, when $(Y_n)_{n \in \mathbb{N}}$ is ergodic, for any bounded and measurable function $\Phi: \mathbb{R}^d \to \mathbb{R}$, the sequence $(\Phi(X_n))_{n \in \mathbb{N}}$ satisfies the law of large numbers in the $L^p$ sense, where $1 \leq p < \infty$.
	Rásonyi and Tikosi, in \cite{rasonyi2022stability}, also studied the iteration \eqref{eq:SGLDit} with merely stationary and unbounded $(Y_n)_{n\in\N}$. In their analysis, they assumed that in the dissipativity condition \eqref{eq:dissipativity}, the parameter $b$ depends on $y$, while $\Delta$ is constant. As a result, the G\"artner-Ellis type condition \eqref{eq:GEcond} could be omitted.
	
	In our joint paper with Rásonyi \cite{lovasCLT}, we studied the stochastic gradient Langevin dynamics with stationary and weakly dependent data streams $(Y_n)_{n \in \mathbb{Z}}$.  
	We proved that $\law{X_n}$ converges to a limiting distribution at an exponential rate as $n \to \infty$.  
	We also established the law of large numbers and the functional central limit theorem for $\Phi(X_n)$, $n \in \mathbb{N}$, where $\Phi: \mathbb{R}^d \to \mathbb{R}$ is an at most polynomially growing function.  
	Our arguments build upon the results from \cite{herrndorf1984}, which required the verification of $\alpha$-mixing properties for the process $(X_n)_{n \in \mathbb{N}}$.
	
	The aforementioned studies all assume stationarity of $(Y_n)_{n\in\mathbb{N}}$. Additionally, they either rely on some mixing property (such as conditional L-mixing or $\alpha$-mixing) of $(Y_n)_{n\in\mathbb{N}}$ or address the issue through alternative means. For example, \cite{lovas} circumvents the need for mixing properties by employing the \eqref{eq:GEcond} G\"artner-Ellis-type condition, whereas \cite{rasonyi2022stability} bypasses it by assuming that $\Delta$ in the \eqref{eq:dissipativity} dissipativity condition is constant.
	In the \eqref{eq:dissipativity} dissipativity condition, we allow both $\Delta$ and $b$ to depend on $y$. Furthermore, we maintain the assumption that the update function $H$ grows at most linearly in its first argument. 
     
	\begin{assumption}\label{as:SGLD:dis_lin}
		We assume that there are measurable functions $\Delta, b:\R^m\to\R$ and a measurable $v:\R^m\to [0,\infty)$ such that for all $x\in\R^d$ and $y\in\R^m$,
		\begin{equation*}
			\left<x,H(x,y)\right> \ge \Delta (y)\|x\|^2-b(y).
		\end{equation*}
		Furthermore, 
		\begin{equation*}
			\| H(x,y)\|\le L(\|x\|+v(y)+1),\quad x\in\R^d,\,y\in\R^m
		\end{equation*}
		holds with some $L>0$.
	\end{assumption}

    Compared to the previously cited works, we establish several substantial improvements. For instance, the second part of Assumption~\ref{as:SGLD:dis_lin} is satisfied with $v(\cdot) = \|\cdot\|$, in particular when $H$ is Lipschitz-continuous. In contrast to earlier contributions, we do not require explicitly the boundedness of either $\Delta$ or $b$, and the update function $H$ is allowed to grow faster than linearly in its second argument. Moreover, we impose no boundedness assumption on the data stream $(Y_n)_{n \in \mathbb{N}}$. Relaxing these restrictive conditions significantly enhances the practical relevance of our results. Finally, in practical applications stationarity of $(Y_n)_{n \in \mathbb{N}}$ cannot generally be assumed. Assuming only weak dependence, we establish the transition of $\alpha$-mixing properties from the sequence $(Y_n)_{n \in \mathbb{N}}$ to $(X_n)_{n \in \mathbb{N}}$, by exploiting the results of Section~\ref{sec:main}.



	\begin{lemma}\label{lem:SGLD}
	Under Assumption \ref{as:SGLD:dis_lin}, for $\lambda > 0$, the parametric kernel $Q$ in \eqref{eq:SGLD_Q} satisfies the drift condition with $V(x) = \|x\|^2$, i.e.
	\begin{align*}
		[Q(y)V](x) \leq \gamma(y) V(x) + K(y),
	\end{align*}
	where $\gamma(y) = 3L^2 \lambda^2 - 2\lambda \Delta(y) + 1$ and
	\[
	K(y) = 3L^2 \lambda^2 (v(y)^2 + 1) + 2\lambda (b(y) + d/\beta).
	\]
	Furthermore, the parametric kernel $Q$ satisfies the following minorization condition for any $R > 0$:
	\begin{align*}
		Q(y,x,A) \geq 2^{-d/2} \exp\left(-\frac{\beta}{2\lambda}(1 + \lambda L)^2(1 + \sqrt{R} + v(y))^2\right) 
			\mathbb{P}\left(\sqrt{\lambda/\beta} \xi_0 \in A\right),
	\end{align*}
	where $y \in \mathbb{R}^m$, $x \in \stackrel{-1}{V}([0,R])$, and $A \in \mathcal{B}(\mathbb{R}^d)$.
	\end{lemma}
	\begin{proof}
		By Assumption \ref{as:SGLD:dis_lin}, we can write 
		\begin{align}
			\begin{split}
				\|x-\lambda H(x,y)\|^2 &= \|x\|^2 - 2\lambda \langle x, H(x,y) \rangle + \lambda^2 \|H(x,y)\|^2 \\
				&\leq \|x\|^2 - 2\lambda (\Delta(y)\|x\|^2 - b(y)) + 
				3L^2\lambda^2 (v(y)^2 + \|x\|^2 + 1) \\
				&= (3L^2\lambda^2 - 2\lambda \Delta(y) + 1)\|x\|^2 
				+ 3L^2\lambda^2 (v(y)^2 + 1) + 2\lambda b(y).
			\end{split}
		\end{align}
		
		Using this, we immediately obtain the desired form of the drift condition:
		\begin{align*}
			\begin{split}
				[Q(y)V](x) &= 
				\mathbb{E}\left[\left\|x-\lambda H(x,y) + \sqrt{\frac{2\lambda}{\beta}}\xi_0\right\|^2\right] 
				= \left\|x-\lambda H(x,y)\right\|^2 + 
				\frac{2\lambda}{\beta}\mathbb{E}[\|\xi_0\|^2] \\
				&\leq 
				(3L^2\lambda^2 - 2\lambda \Delta(y) + 1)\|x\|^2 
				+ 3L^2\lambda^2 (v(y)^2 + 1) + 2\lambda (b(y) + d/\beta).
			\end{split}
		\end{align*}
		
		To verify the minorization condition, let $y \in \mathbb{R}^m$, $x \in \mathbb{R}^d$ such that $V(x)=\|x\|^2 \leq R$, and let $A \in \mathcal{B}(\mathbb{R}^d)$ be arbitrary. Then, by employing the substitution 
        $u = \sqrt{2}z+\sqrt{\frac{\beta}{\lambda}}\left(x-\lambda H(x,y)\right)$, 
        we can write:
		\begin{align}\label{eq:Qmin}
			\begin{split}
				Q(y,x,A) &= \int_{\mathbb{R}^d} \ind_{A}\left(x - \lambda H(x,y) + \sqrt{\frac{2\lambda}{\beta}} z\right) f_{\xi_0}(z)\, \mathrm{d}z \\
				&= 
                \int_{\mathbb{R}^d} \ind_{A}\left(\sqrt{\frac{\lambda}{\beta}} u\right)
				f_{\xi_0}\left(\frac{1}{\sqrt{2}}u-\frac{1}{\sqrt{2}}\sqrt{\frac{\beta}{\lambda}}(x-\lambda H(x,y))\right)2^{-d/2}\, \mathrm{d}u,
			\end{split}
		\end{align}
		where $f_{\xi_0}$ is the density function of $\xi_0$. Using the inequalities 
		$\|u - (x - \lambda H(x,y))\|^2 \leq 2\|u\|^2 + 2\|x - \lambda H(x,y)\|^2$ and 
		$\|x - \lambda H(x,y)\| \leq \|x\| + \lambda L (v(y) + \|x\| + 1)$, we obtain
		\begin{align*}
			f_{\xi_0}\left(\frac{1}{\sqrt{2}}u-\frac{1}{\sqrt{2}}\sqrt{\frac{\beta}{\lambda}}(x-\lambda H(x,y))\right)
			&\geq f_{\xi_0}(u)\exp\left(-\frac{\beta}{2\lambda}\|x-\lambda H(x,y)\|^2\right)
            \\
            &\geq \exp\left(-\frac{\beta}{2\lambda}(1 + \lambda L)^2(1 + \sqrt{R} + v(y))^2\right)  f_{\xi_0}(u)
		\end{align*}
		Substituting this back into \eqref{eq:Qmin} yields
		\begin{align*}
			Q(y,x,A) \geq 2^{-d/2} \exp\left(-\frac{\beta}{2\lambda}(1 + \lambda L)^2(1 + \sqrt{R} + v(y))^2\right) 
			\mathbb{P}\left(\sqrt{\lambda/\beta} \xi_0 \in A\right),
		\end{align*}
		which is exactly what we aimed to prove.
	\end{proof}

    In Assumption~\ref{as:SGLD:dis_lin}, without loss of generality, we may assume that $b$ is non-negative. Then, by applying the Cauchy--Schwarz inequality, we obtain for any fixed $y \in \mathbb{R}^m$ that
\begin{align*}
  \Delta(y)\|x\| - b(y) \;\le\; \langle x, H(x,y)\rangle \;\le\; L \|x\| \bigl(1+\|x\|+v(y)\bigr),
\end{align*}
for all $x \in \mathbb{R}^d$. This inequality can only hold if
\begin{align*}
  L^2 \bigl(v(y)+1\bigr)^2 \;\le\; 4 b(y)\,\bigl(L - \Delta(y)\bigr),
\end{align*}
which, together with the fact that $b(y) \ge 0$, yields $\Delta(y) \le L$.  
Consequently, for $\gamma$ as defined in Lemma~\ref{lem:SGLD} and any positive step size $\lambda > 0$, we have
\[
  \gamma(y) \;\ge\; 3L^2\lambda^2 - 2L\lambda + 1 \;\ge\; \tfrac{2}{3} \;>\; 0, 
  \quad y \in \mathbb{R}^m,
\]
which ensures that $\gamma(y) \in (0,\infty)$ for all $y \in \mathbb{R}^m$.  
Similarly, we may replace $K(\cdot)$ with $K(\cdot)+1$; hence, without loss of generality, we may assume that $K(\cdot) \ge 1$ in the drift condition.


    In the literature on stochastic gradient Langevin dynamics, the minorization measure \( \kappa_R(y, \cdot) \) appearing on the right-hand side of the minorization condition is typically a suitable restriction of the Lebesgue measure. In Lemma \ref{lem:SGLD}, however, we use a standard Gaussian distribution as the minorization measure. Moreover, in our construction, \( \kappa_R(y, A) \) does not depend on \( R \) or \( y \).

	\begin{assumption}\label{as:SGLD:coefficients}
		We assume that there exists $\delta\in (0,1)$ such that 
        $$
        \sup_{n\in\N}\E \left[v(Y_n)^{2\delta}+|b(Y_n)|^{\delta}\right]<\infty.$$ 
        Furthermore, we suppose that
        $$
        \tilde{\Delta}:=\inf_{n\in\N}\E [\Delta (Y_n)]>0.
        $$
	\end{assumption}

    Note that \(\E[\gamma(Y_n)] = 3L^2\lambda^2 - 2\lambda \E[\Delta(Y_n)] + 1 < 1\) holds whenever \(0 < \lambda < \frac{2\E[\Delta(Y_n)]}{3L^2}\). Therefore, under Assumption~\ref{as:SGLD:coefficients}, for any $0 < \lambda < \frac{2}{3L^2} \tilde{\Delta}$, we have \(\sup_{n \in \mathbb{N}} \E[\gamma(Y_n)] < 1\). In the following, the step size \(\lambda > 0\) is always chosen accordingly.

    Additionally, if the data stream \((Y_n)_{n \in \mathbb{Z}}\) is strongly stationary, then Assumption~\ref{as:SGLD:coefficients} reduces to the following set of conditions:
    \[
    \E\left[ v(Y_0)^{2\delta} + |b(Y_0)|^\delta \right] < \infty,
    \quad \text{for some } \delta \in (0,1),
    \quad \text{and} \quad
    \E[\Delta(Y_0)] > 0.
    \]
    Thus, by choosing the step size smaller than $\frac{2}{3L^2}\E[\Delta(Y_0)]$, we can ensure that 
    \[
    \sup_{n \in \mathbb{N}} \E[\gamma(Y_n)] = \E [\gamma (Y_0)]< 1.
    \]
    
    
	\begin{theorem}\label{thm:SGLD:main}
        Let Assumption~\ref{as:SGLD:coefficients} hold and choose \(\lambda > 0\) such that \(\sup_{n \in \mathbb{N}} \E[\gamma(Y_n)] < 1\).  
In addition, suppose that the data stream \((Y_n)_{n \in \mathbb{N}}\) satisfies at least one of the following conditions:
\begin{enumerate}
    \item There exists \(M > 0\) such that for all \(t > 0\) and \(n \ge 1\), the following concentration inequality holds:
    \[
    \sup_{j \in \mathbb{N}} \P \left(
        \sum_{k=1}^n \log \gamma(Y_{k+j}) > t + \sum_{k=1}^n \E[\log \gamma(Y_{k+j})]
    \right) \le 
    \exp \left(-\frac{t^2}{Mn}\right);
    \]

    \item \((Y_n)_{n \in \mathbb{N}}\) is \(\phi\)-mixing, and \(\sup_{n \in \mathbb{N}} \|\Delta(Y_n)\|_\infty < \infty\);

    \item  \((Y_n)_{n \in \mathbb{N}}\) is \(\psi\)-mixing;

    \item 
    $$
    \Delta_{-}:=\inf_y\Delta (y)>0.
    $$
\end{enumerate}

Then the following hold:
\begin{enumerate}

\item There exists $s \in (0,1)$ such that for any measurable function $\Phi:\mathbb{R}^d \to \mathbb{R}$ satisfying
\begin{equation}\label{eq:Phi}
|\Phi(x)|^p \le C(1 + \|x\|^{2s})
\end{equation}
for some $C > 0$ and $p \ge 1$, we have $\sup_{n \in \mathbb{N}} \mathbb{E}[|\Phi(X_n)|^p] < \infty$.

\item There exists $\kappa\in (0,1)$ and $c>0$ such that
$$
\alpha^{X,Y}(n)\leq c \inf_{1\leq i\leq q\leq n}\left\{r_i+\kappa^{n/q}+\alpha^Y(q+1-i)\right\},
$$
where $(r_i)_{i\in\N}$ is as in Corollary \ref{cor:A12fromMixing}.

In particular, 
\begin{itemize}
 \item if $\alpha^Y(n) = O(n^{-a})$ for some $a > 1$, then $\alpha^X(n) = O\left(\frac{\log^a(n)}{n^a}\right)$. 
 
 \item If $\alpha^Y(n) = O(\kappa^n)$ for some $\kappa \in (0,1)$, then there exists $\bar{\kappa} \in (0,1)$ such that $\alpha^X(n) = O(\bar{\kappa}^{\sqrt{n}})$.
\end{itemize}

\item Furthermore, if $(Y_n)_{n \in \mathbb{Z}}$ is stationary, then there exists a stationary process $((X_n^*, Y_n))_{n \in \mathbb{Z}}$ whose distribution is unique, and such that $\operatorname{Law}(X_n) \to \operatorname{Law}(X_0^*)$ as $n \to \infty$, in total variation distance, with the same rates as in point 2. That is,
\[
\|\operatorname{Law}(X_n) - \operatorname{Law}(X_0^*)\|_{\mathrm{TV}} = O\left(\frac{\log^a(n)}{n^a}\right)
\]
if $\alpha^Y(n) = O(n^{-a})$ for some $a > 1$, and
\[
\|\operatorname{Law}(X_n) - \operatorname{Law}(X_0^*)\|_{\mathrm{TV}} = O(\bar{\kappa}^{\sqrt{n}})
\]
for some $\bar{\kappa} \in (0,1)$ if $\alpha^Y(n) = O(\kappa^n)$ for some $\kappa \in (0,1)$.

\end{enumerate}
\end{theorem}

\paragraph{Note.} 
For the SGLD, algorithm, the convergence of the marginal distribution to the stationary distribution of a MCRE has been studied in previous references. See in particular, \cite{lovas} who obtained a subexponential rate or \cite{lovasCLT} for an exponential rate, without assuming a mixing condition on the environment. However, their results require the boundedness of the process $(Y_t)_{t\in\Z}$, which is an important limitation. Working with unbounded environments leads to a substantial difficulty, with drift/minorization conditions depending on the environment, as shown in Lemma \ref{lem:SGLD}. The fact that the mixing coefficients of the process $(Y_t)_{t\in\Z}$ influence the convergence rate is mainly due to our general proof technique for getting Theorem \ref{thm:A12fromMixing}, which aim is to control how fast the environment goes back to stable regions. Our results are then much more general, though one cannot recover directly the rates obtained under more restrictive assumptions.

\begin{proof}
We show that there exists an exponent $s \in (0,1)$ such that the parametric kernel $Q$ defined in \eqref{eq:SGLD_Q} with $V(x)=\|x\|^{2s}$ satisfies part A1 of Assumption~\ref{as:main_assumptions} in the long-term contractivity condition form, i.e., $d_n = O(\rho^n)$ for some $\rho \in (0,1)$. We shall see that Part A2 follows easily from Assumption \ref{as:SGLD:coefficients}. Under these conditions, point 1 follows from Lemma~\ref{mom}, point 2 from part 4 of Lemma~\ref{conc}, and point 3 from the forward coupling argument with the unique stationary solution, as described in Section~\ref{sec:stac_environment}.

Let $s \in (0,1)$ be arbitrary and suppose that $V^s(x) \le R$. Then, by Lemma~\ref{lem:SGLD}, we have
\[
Q(y, x, A) \ge 2^{-d/2} \exp\left(
-\frac{\beta}{2\lambda}(1 + \lambda L)^2 \left(1 + R^{\frac{1}{2s}} + v(y)^2 \right)
\right) \mathbb{P} \left(
\sqrt{\lambda / \beta} \, \xi_0 \in A
\right),
\]
for all $y \in \mathbb{R}^m$ and $A \in \mathcal{B}(\mathbb{R}^d)$ hence for the minorization coefficient, we have 
\[
\beta(R, y) = 1 - 2^{-d/2} \exp\left(
-\frac{\beta}{2\lambda}(1 + \lambda L)^2 \left(1 + R^{\frac{1}{2s}} + v(y)^2 \right)
\right).
\]
By Assumption~\ref{as:SGLD:coefficients}, there exists $\delta \in (0,1)$ such that $\sup_{n \in \mathbb{N}} \mathbb{E}[v(Y_n)^{2\delta}] < \infty$. Hence, by Markov's inequality,
\[
\sup_{n \in \mathbb{N}} \mathbb{P}(v(Y_n) > v') \le \frac{\mathbb{E}[v(Y_n)^{2\delta}]}{(v')^{2\delta}} \to 0 \quad \text{as } v' \to \infty,
\]
which implies
\[
\lim_{\bar{\beta} \uparrow 1} \sup_{n \in \mathbb{N}} \mathbb{P}(\beta(R, Y_n) > \bar{\beta}) = 0,
\]
and thus part A2 of Assumption~\ref{as:main_assumptions} is satisfied.

Since $\sup_{n \in \mathbb{N}} \mathbb{E}[\gamma(Y_n)] < 1$, for any $j \in \mathbb{N}$ we have
\[
- \max_{1 \le k \le n} \mathbb{E}[\log \gamma(Y_{k+j})] 
\ge 1 - \max_{1 \le k \le n} \mathbb{E}[\gamma(Y_{k+j})] 
\ge 1 - \sup_{n \in \mathbb{N}} \mathbb{E}[\gamma(Y_n)] > 0,
\]
thus, under condition~1, Lemma~\ref{useful} yields the existence of $\delta_0 > 0$ such that for any $0 < s < \delta_0$,
\begin{equation} \label{eq:concl}
\sup_{j \in \mathbb{N}} \mathbb{E} \left[ \gamma(Y_{j+1})^s \cdots \gamma(Y_{j+n})^s \right] = O(\rho^n)
\end{equation}
for some $\rho \in (0,1)$. If condition~2 or 3 holds, the same conclusion follows from Lemma~\ref{lem:ga}. Let $0 < s < \frac{1}{2} \min(\delta, \delta_0)$ be arbitrary, where $\delta$ is as in Assumption~\ref{as:SGLD:coefficients}. Then it is straightforward to verify that for the functions $\gamma(y)$ and $K(y)$ defined in Lemma~\ref{lem:SGLD}, we have
\[
d_0 = \sup_{n \ge 0} \mathbb{E} \left[ K(Y_n)^s + \gamma(Y_n)^s \right] < \infty.
\]
Moreover, by the Cauchy–Schwarz inequality and \eqref{eq:concl}, we obtain
\[
d_l = \sup_{t \ge -1} \mathbb{E} \left[
K(Y_t)^s \prod_{i=1}^l \gamma(Y_{t+i})^s
\right]
\le 
\sup_{t \in \mathbb{N}} \mathbb{E}^{1/2}[K(Y_t)^{2s}] \cdot 
\sup_{t \in \mathbb{N}} 
\mathbb{E}^{1/2} \left[ \prod_{i=1}^l \gamma(Y_{t+i})^{2s} \right]
= O(\rho^l),
\]
which completes the proof.

\end{proof}

Combining the result above with the theorems on strongly mixing sequences from Appendix A of \cite{lovas2024transition}, we obtain several interesting and useful consequences. For instance, applying Theorems A.1 and A.3 together with Remark A.2 in \cite{lovas2024transition}, and under the assumptions of Theorem~\ref{thm:SGLD:main}, if $\alpha^Y(n) = O(n^{-\frac{r}{r - 2}})$ for some $r > 2$, and $\Phi:\mathbb{R}^d \to \mathbb{R}$ satisfies condition~\eqref{eq:Phi} with $p > r/2$, then
\[
\frac{1}{n}\sum_{k=1}^n \left(\Phi(X_k) - \mathbb{E}[\Phi(X_k)]\right) \to 0, \quad \text{as } n \to \infty,
\]
almost surely and also in $L^1$. Similarly, under suitable summability assumptions on the decay rate of $\alpha^Y(n)$ and an appropriate choice of $p$ in \eqref{eq:Phi}, a central limit theorem for the sequence $(\Phi(X_n))_{n \in \mathbb{N}}$ can be established by applying Corollary A.6 and Remark A.7 in \cite{lovas2024transition} or directly by applying Theorem 1.5 in \cite{lovas2024transition}.

\paragraph{Example: Online logistic regression.} 
Example~1.1 in \cite{lovasCLT} presents a regularized logistic regression model, where $m \geq 2$, $d := m-1$, and $Y_n=(Q_n,Z_n)\in \{0,1\}\times\R^{d}$, $n\in\Z$, is a stationary sequence of random variables. 
The goal is to optimize the regression parameter $\theta \in \R^d$ such that the $L^2$-regularized binary cross-entropy
\begin{equation}\label{eq:logiU}
U(\theta) \;=\; -\E \!\left[\log \!\Big( \sigma(\langle \theta, Z_0\rangle)^{Q_0} \,
\big(1-\sigma(\langle \theta, Z_0\rangle)\big)^{1-Q_0}\Big)\right] \;+\; c\|\theta\|^2 ,
\end{equation}
is minimized, where $\sigma(x)=1/(1+e^{-x})$ denotes the sigmoid function and $c>0$ is a regularization parameter. Thus, one aims to predict the binary variable $Q$ from the covariates $Z$. 

Despite its simplicity, the logistic regression model outlined above proves to be an effective and natural choice for analyzing online data. In \cite{tyagi2018sentiment}, the authors employ a logistic regression classifier with unigram feature vectors for sentiment analysis of Twitter posts. Similarly, \cite{Dey2017OnlineLearning} applies this model to the Amazon product review dataset, where, based on the occurrences of key words (represented by the coordinates of $Z$), each review is classified as positive or negative ($Q = 1$ or $Q = 0$). Here, the feature vector consists of word counts of important terms, and the study compares the convergence of stochastic gradient descent versus batch gradient descent. Finally, \cite{aliman2022sentiment} presents a study evaluating various machine learning techniques for tweet classification, concluding that logistic regression achieves the highest accuracy and serves as the best-fitted algorithm for predicting potential mental health crisis tweets.

An unbiased estimator of the gradient $\nabla U$ is
\begin{equation}\label{eq:logiH}
H(\theta,(q,z)) = -\big(q-\sigma(\langle \theta, z\rangle)\big)z + 2c\theta.
\end{equation}
A straightforward calculation shows that $H$ satisfies Assumption~\ref{as:SGLD:dis_lin}.
Noting that $q\in \{0,1\}$ and $\sigma(x)\in (0,1)$ for all $x$, the Cauchy--Schwarz inequality yields
\begin{align*}
\langle \theta, H(\theta,(q,z)) \rangle 
&\ge 2c \|\theta\|^2 - 2 \|\theta\| \|z\| \\
&\ge (2c-1) \|\theta\|^2 - \|(q,z)\|^2.
\end{align*}
Hence, the dissipativity condition in the first part of Assumption~\ref{as:SGLD:coefficients} holds with $\Delta(y)=2c-1$ and $b(y)=\|y\|^2$, provided $c>1/2$.  
Similarly, by the triangle inequality,
\begin{align*}
\|H(\theta,(q,z))\| &\le |q-\sigma(\langle \theta, z\rangle)| \|z\| + 2c\|\theta\| \\
&\le 2(c+1) \left( \|\theta\| + \|(q,z)\| + 1 \right),
\end{align*}
so the linear growth condition in the second part of Assumption~\ref{as:SGLD:dis_lin} is satisfied with $L=2(c+1)$ and $v(\cdot)=\|\cdot\|$.  
Since $\Delta(y) = 2c-1>0$ for $c>1/2$, the second part of Assumption~\ref{as:SGLD:coefficients} holds automatically.  
Finally, Assumption~\ref{as:SGLD:coefficients} 
reduces to
\begin{equation}\label{eq:Ybdmom}
\sup_{n\in\N}\E \left[\|Y_n\|^\delta\right] < \infty
\end{equation}
for some $\delta\in(0,1)$, which is typically satisfied for data streams encountered in applications.

In this model, the quantity $\gamma(y) = 3L^2\lambda^2 - 2\Delta \lambda + 1$ appearing in Lemma~\ref{lem:SGLD} does not depend on $y$, since for a fixed regularization parameter $c>1/2$ we have constant values $\Delta = 2c-1$ and $L = 2(c+1)$. This allows us to determine an optimal step size $\lambda^*$ as a function of $c$, for which the gamma contraction parameter is minimized, i.e., $\gamma^* = \gamma(\lambda^*)$. Explicitly, we have
\begin{equation*}
\lambda^* = \frac{2c-1}{12 (c+1)^2}, \quad \text{and} \quad \gamma^* = \frac{8 c^2 + 28 c + 11}{12 (c+1)^2}.
\end{equation*}

\begin{figure}[ht]
    \centering
    \begin{subfigure}[b]{0.48\textwidth}
        \centering
        \includegraphics[width=\textwidth]{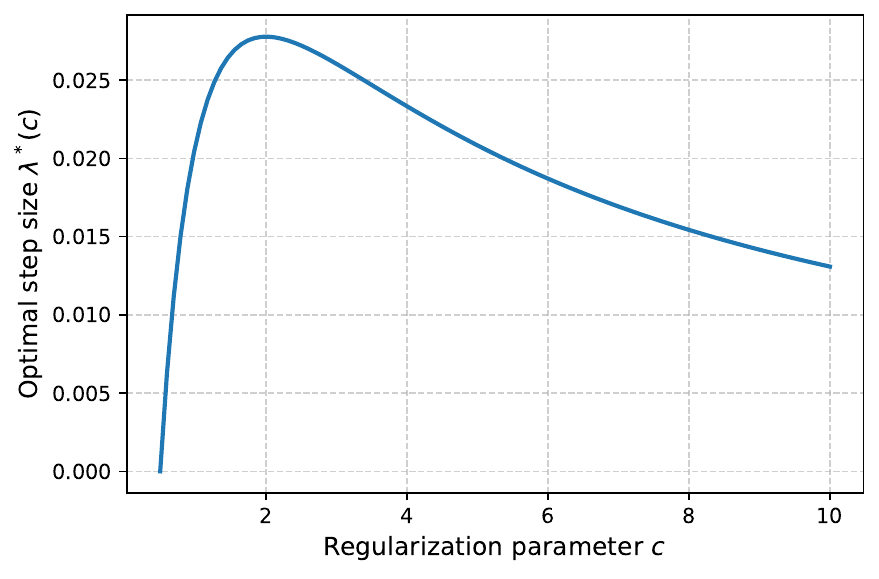}
        \caption{Optimal step size $\lambda^*(c)$ as a function of the regularization parameter.}
        \label{fig:lambda}
    \end{subfigure}
    \hfill
    \begin{subfigure}[b]{0.48\textwidth}
        \centering
        \includegraphics[width=\textwidth]{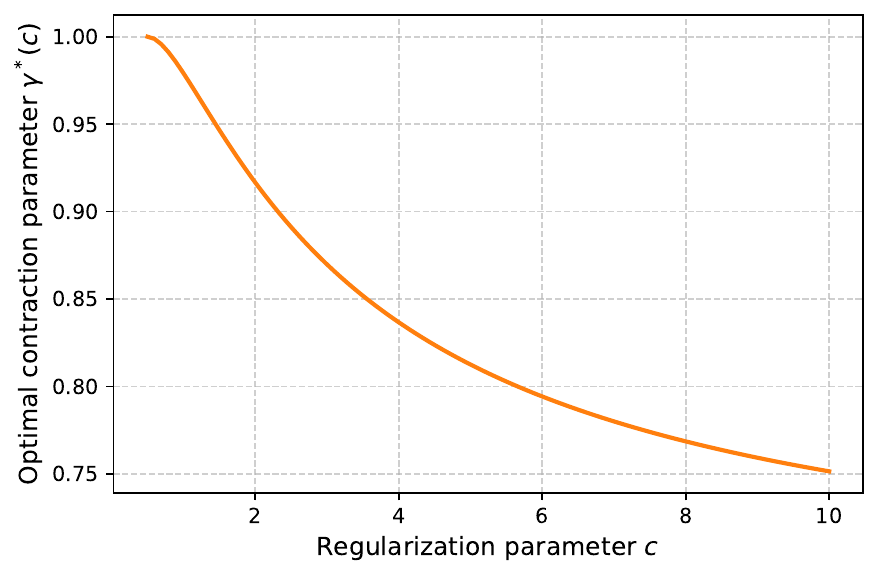}
        \caption{Optimal contraction parameter $\gamma^*(c)$ as a function of the regularization parameter.}
        \label{fig:gamma}
    \end{subfigure}
    \caption{Behavior of the optimal parameters $\lambda^*(c)$ and $\gamma^*(c)$ with respect to the regularization parameter $c$.}
    \label{fig:lambda-gamma}
\end{figure}
Figure~\ref{fig:lambda-gamma} shows that the optimal step size in $c$ admits a maximum at $c=2$ ($\lambda^* (2) = \frac{1}{36}$, $\gamma^* (2) =\frac{11}{12}$); thus, increasing the regularization parameter beyond this point forces a smaller step size. The choice of $c$ further affects the functional $U$ in \eqref{eq:logiU} and the optimizer $\theta^*$, thereby introducing bias. Estimating this bias is beyond the scope of the present work, where we restrict attention to the mixing properties of the SGLD iteration defined by the update function \eqref{eq:logiH}.

Compared to the results in \cite{lovasCLT}, a key advancement is that we do not need to assume the sequence $(Y_n)_{n\in\Z}$ to be stationary. In this setting, one can establish that the $\alpha$-mixing properties of $(Y_n)_{n\in\Z}$ are inherited by the sequence of learning parameters $(\theta_n)_{n\in\N}$. Although there is no well-defined functional $U$ in the form of \eqref{eq:logiU} whose minimum is sought, Theorem~\ref{thm:SGLD:main} still allows us to study the sequence $(\theta_n)_{n\in\N}$.

\begin{proposition}
For any fixed $\beta>0$ and $c>1/2$, if the step size satisfies
$$
0<\lambda<\frac{2\Delta}{3L^2}=\frac{2c-1}{6(c+1)^2},
$$
then, starting from any deterministic initial condition $\theta_0\in\R^d$, the SGLD iteration
$$
\theta_{n+1} = \theta_n - \lambda H (\theta_n, Y_n) + \sqrt{\frac{2\lambda}{\beta}}\,\xi_{n+1},\,\, n\in\N,
$$
with $(\xi_n)_{n\ge 1}$ i.i.d.\ standard $d$-dimensional Gaussian variables, satisfies the following for $(\theta_n)_{n\in\N}$:
\begin{enumerate}
    \item For any measurable $\Phi:\R^d\to\R$ with $|\Phi (\theta)|^p \le C (1+\|\theta\|^\delta)$ for some $C>0$ and $p\ge 1$, 
    $$
    \sup_{n\in\N} \E \left[|\Phi (\theta_n)|^p\right] < \infty,
    $$
    where $\delta\in (0,1)$ is as in \eqref{eq:Ybdmom}.
    
    \item There exist $\eta\in (0,1)$ and $c>0$ such that
    $$  
    \alpha^{Y,\theta} (n) \le c \inf_{1\le i\le q\le n} \{\eta^{n/q} + \alpha^Y (q+1-i)\}.
    $$
    
    In particular,  
        \begin{itemize}
            \item if $\alpha^Y (n) = O(n^{-a})$ for some $a>1$, then $\alpha^\theta (n) = O \left(\frac{\log^a (n)}{n^a}\right)$;

            \item if $\alpha^Y (n) = O(\kappa^n)$ for some $\kappa\in (0,1)$, then there exists $\bar{\kappa}\in (0,1)$ such that $\alpha^\theta (n) = O(\bar{\kappa}^{\sqrt{n}})$.
        \end{itemize}
\end{enumerate}
\end{proposition}
\begin{proof}
The claim follows directly from Theorem~\ref{thm:SGLD:main}.
\end{proof}


	\bibliographystyle{siam}
	\bibliography{MCRE}

\begin{thebibliography}{10}

\bibitem{aliman2022sentiment}
{\sc G.~Aliman, T.~F.~S. Nivera, J.~C.~A. Olazo, D.~J.~P. Ramos, C.~D.~B. Sanchez, T.~M. Amado, N.~M. Arago, R.~L. Jorda~Jr, G.~C. Virrey, and I.~C. Valenzuela}, {\em Sentiment analysis using logistic regression}, Journal of Computational Innovations and Engineering Applications, 7 (2022), pp.~35--40.

\bibitem{6}
{\sc M.~Barkhagen, N.~H. Chau, E.~Moulines, M.~Rasonyi, S.~Sabanis, and Y.~Zhang}, {\em On stochastic gradient {L}angevin dynamics with dependent data streams in the logconcave case}, Bernoulli, 27 (2021), pp.~1--33.

\bibitem{bosq1993bernstein}
{\sc D.~Bosq}, {\em Bernstein-type large deviations inequalities for partial sums of strong mixing processes}, Statistics, 24 (1993), pp.~59--70.

\bibitem{Bradley2005}
{\sc R.~C. Bradley}, {\em Basic properties of strong mixing conditions. a survey and some open questions}, Probability Surveys, 2 (2005), pp.~107 -- 144.

\bibitem{chamberlain1982general}
{\sc G.~Chamberlain}, {\em The general equivalence of granger and sims causality}, Econometrica: Journal of the Econometric Society,  (1982), pp.~569--581.

\bibitem{5author}
{\sc H.~N. Chau, E.~Moulines, M.~R\'asonyi, S.~Sabanis, and Y.~Zhang}, {\em On stochastic gradient {L}angevin dynamics with dependent data streams: the fully non-convex case}, {S}{I}{A}{M} Journal on Mathematics and Data Science, 3 (2021), pp.~959--986.

\bibitem{cogburn1984ergodic}
{\sc R.~Cogburn}, {\em The ergodic theory of markov chains in random environments}, Zeitschrift f{\"u}r Wahrscheinlichkeitstheorie und verwandte Gebiete, 66 (1984), pp.~109--128.

\bibitem{cogburn91}
\leavevmode\vrule height 2pt depth -1.6pt width 23pt, {\em On the central limit theorem for {M}arkov chains in random environments}, The Annals of Probability, 19 (1991), pp.~587--604.

\bibitem{Dey2017OnlineLearning}
{\sc S.~Dey}, {\em Online learning: Sentiment analysis on amazon product review dataset with logistic regression via stochastic gradient ascent in python}.
\newblock \url{https://sandipanweb.wordpress.com/2017/03/}, Mar. 2017.
\newblock Blog post, accessed 2025-09-13.

\bibitem{douc2014nonlinear}
{\sc R.~Douc, E.~Moulines, and D.~Stoffer}, {\em Nonlinear time series: Theory, methods and applications with R examples}, CRC press, 2014.

\bibitem{doukhan1995mixing}
{\sc P.~Doukhan}, {\em Mixing}, in Mixing: Properties and Examples, Springer, 1995, pp.~15--23.

\bibitem{doukhan1994functional}
{\sc P.~Doukhan, P.~Massart, and E.~Rio}, {\em The functional central limit theorem for strongly mixing processes}, in Annales de l'IHP Probabilit{\'e}s et statistiques, vol.~30, 1994, pp.~63--82.

\bibitem{doukhan2023stationarity}
{\sc P.~Doukhan, M.~H. Neumann, and L.~Truquet}, {\em Stationarity and ergodic properties for some observation-driven models in random environments}, The Annals of Applied Probability, 33 (2023), pp.~5145--5170.

\bibitem{rasonyi2018}
{\sc B.~Gerencs{\'e}r and M.~R{\'a}sonyi}, {\em On the ergodicity of certain {M}arkov chains in random environments}, Journal of Theoretical Probability,  (2023), pp.~1--33.

\bibitem{gyorfi2002}
{\sc L.~Gy\"orfi and G.~Morvai}, {\em Queueing for ergodic arrivals and services}, in Limit theorems in probability and statistics, {I. Berkes, et al.}, ed., vol.~2, 2002, pp.~127--141.
\newblock Fourth Hungarian colloquium on limit theorems in probability and statistics, Balatonlelle, Hungary, 1999.

\bibitem{hansen2008uniform}
{\sc B.~E. Hansen}, {\em Uniform convergence rates for kernel estimation with dependent data}, Econometric Theory, 24 (2008), pp.~726--748.

\bibitem{herrndorf1984}
{\sc N.~Herrndorf}, {\em {A Functional Central Limit Theorem for Weakly Dependent Sequences of Random Variables}}, The Annals of Probability, 12 (1984), pp.~141 -- 153.

\bibitem{kifer1}
{\sc Y.~Kifer}, {\em Perron-{F}robenius theorem, large deviations, and random perturbations in random environments}, Math. Zeitschrift, 222 (1996), pp.~677--698.

\bibitem{kifer1998limit}
{\sc Y.~Kifer}, {\em Limit theorems for random transformations and processes in random environments}, Transactions of the American Mathematical Society, 350 (1998), pp.~1481--1518.

\bibitem{lindvall2002lectures}
{\sc T.~Lindvall}, {\em Lectures on the coupling method}, Courier Corporation, 2002.

\bibitem{lovas2024transition}
{\sc A.~Lovas}, {\em Transition of $\alpha$-mixing in random iterations with applications in queuing theory}, arXiv preprint,  (2024).

\bibitem{lovas2021ergodic}
{\sc A.~Lovas and M.~R{\'a}sonyi}, {\em Ergodic theorems for queuing systems with dependent inter-arrival times}, Operations Research Letters, 49 (2021), pp.~682--687.

\bibitem{lovas}
{\sc A.~Lovas and M.~R{\'a}sonyi}, {\em {M}arkov chains in random environment with applications in queuing theory and machine learning}, Stochastic Processes and their Applications, 137 (2021), pp.~294--326.

\bibitem{lovasCLT}
{\sc A.~Lovas and M.~R\'asonyi}, {\em Functional central limit theorem and strong law of large numbers for stochastic gradient {L}angevin dynamics}, Appl Math Optim,  (2023)), p.~78.

\bibitem{merlevede2011bernstein}
{\sc F.~Merlev{\`e}de, M.~Peligrad, and E.~Rio}, {\em A bernstein type inequality and moderate deviations for weakly dependent sequences}, Probability Theory and Related Fields, 151 (2011), pp.~435--474.

\bibitem{mt}
{\sc S.~P. Meyn and R.~L. Tweedie}, {\em {M}arkov chains and stochastic stability}, Springer-Verlag, 1993.

\bibitem{orey1991markov}
{\sc S.~Orey}, {\em Markov chains with stochastically stationary transition probabilities}, The Annals of Probability, 19 (1991), pp.~907--928.

\bibitem{raginsky}
{\sc M.~Raginsky, A.~Rakhlin, and M.~Telgarsky}, {\em Non-convex learning via stochastic gradient {L}angevin dynamics: a nonasymptotic analysis}, in Proceedings of the 2017 Conference on Learning Theory, S.~Kale and O.~Shamir, eds., vol.~65 of Proceedings of Machine Learning Research, Amsterdam, Netherlands, 07--10 Jul 2017, pp.~1674--1703.

\bibitem{rasonyi2022stability}
{\sc M.~R{\'a}sonyi and K.~Tikosi}, {\em On the stability of the stochastic gradient langevin algorithm with dependent data stream}, Statistics \& Probability Letters, 182 (2022), p.~109321.

\bibitem{Rio}
{\sc E.~Rio}, {\em Asymptotic theory of weakly dependent random processes}, vol.~80, Springer, 2017.

\bibitem{rosenthal1995minorization}
{\sc J.~S. Rosenthal}, {\em Minorization conditions and convergence rates for markov chain monte carlo}, Journal of the American Statistical Association, 90 (1995), pp.~558--566.

\bibitem{stenflo}
{\sc O.~Stenflo}, {\em {M}arkov chains in random environments and random iterated function systems}, Trans. American Math. Soc., 353 (2001), pp.~3547--3562.

\bibitem{tjostheim1990non}
{\sc D.~Tj{\o}stheim}, {\em Non-linear time series and markov chains}, Advances in applied probability, 22 (1990), pp.~587--611.

\bibitem{tong1980threshold}
{\sc H.~Tong and K.~S. Lim}, {\em Threshold autoregression, limit cycles and cyclical data}, Journal of the Royal Statistical Society: Series B (Methodological), 42 (1980), pp.~245--268.

\bibitem{Truquet1}
{\sc L.~Truquet}, {\em Ergodic properties of some {M}arkov chains models in random environments}, 2021.

\bibitem{TRUQUET2023294}
{\sc L.~Truquet}, {\em Strong mixing properties of discrete-valued time series with exogenous covariates}, Stochastic Processes and their Applications, 160 (2023), pp.~294--317.

\bibitem{tyagi2018sentiment}
{\sc A.~Tyagi and N.~Sharma}, {\em Sentiment analysis using logistic regression and effective word score heuristic}, International Journal of Engineering and Technology (UAE), 7 (2018), pp.~20--23.

\bibitem{welling2011bayesian}
{\sc M.~Welling and Y.~W. Teh}, {\em Bayesian learning via stochastic gradient {L}angevin dynamics}, in Proceedings of the 28th International Conference on Machine Learning, 2011, pp.~681--688.

\end{thebibliography}
	
	\newpage
	\appendix
	\section{Proof of Theorem \ref{thm:A12fromMixing}}\label{ap:proofs}
	
	The proof follows the approach outlined in \cite{Truquet1}. We first represent the Markov chain in random environment $(X_n)_{n \in \mathbb{N}}$ with a random iteration of the form \eqref{eq:iter}. Then, for a fixed realization of the environment $\mathbf{y} \in \mathcal{Y}$, we define a sequence of time indices $j < \tau_0 < \tau_1 < \dots$ which depends solely on the trajectory $\mathbf{y}$. We then show that there exist absolute constants $M > 0$ and $\rho \in (0,1)$ such that the following quenched coupling inequality holds:
    \begin{equation}\label{coupling}
	\P \left( Z_{j,j+n}^{x_1,\mathbf{y}} \ne Z_{j,j+n}^{x_2,\mathbf{y}} \right) \le M \left(1 + V(x_1) + V(x_2)\right) \rho^{L_{n,j}}, \quad x_1,x_2 \in \mathcal{X},
	    \end{equation}
	where the random variable $L_{n,j}$ appearing in the exponent depends on the sequence $(\tau_i)_{i \in \mathbb{N}}$ and, ultimately, on the trajectory $\mathbf{y}$.
	
	Finally, to obtain a bound for the non-coupling probability 
	\[
	\sup_{j \in \mathbb{N}} \P \left( Z_{j,j+n}^{X_j} \ne Z_{j,j+n}^{x_0} \right)
	\]
	satisfying \eqref{eq:coupling_cond}, we integrate the previously derived quenched bound with respect to the law of $Y$. To this end, it is necessary to control the deviation of $L_{n,j}$, or at least that of a suitable lower bound. This can be achieved by applying Assumption~\ref{as:main_assumptions}.

	\subsection{Basic lemmas and notations}\label{subsec:lemmas_not}
	
	In our earlier papers (see Lemma 7.1 in \cite{lovas} and Lemma 3.9 in \cite{lovasCLT}), we relied on special cases of the following lemma. It is also a variant of Lemma 6.1 in \cite{rasonyi2018}, albeit in a somewhat broader context. Such representations of parametric kernels satisfying the minorization condition \eqref{eq:smallset} can be considered standard. The following general version appears as Lemma C.2 in \cite{lovas2024transition}, where a detailed proof can also be found.
	\begin{lemma}\label{lem:T}
		Suppose the parametric kernel $Q:\Y\times\X\times\B(\X)\to [0,1]$ satisfies the minorization condition given by \eqref{eq:smallset} with $R>0$. Then there exists a measurable mapping $T^R:\X\times\Y\times [0,1]\to\X$ such that
		$$
		Q(y,x,A)=\int_{[0,1]}\ind_{T^R (x,y,u)\in A}\,\dint u,
		$$
		for all $x\in\X$, $A\in\B(\X)$ and $y\in\Y$. 
		Furthermore, exists a Borel set $U\in \B ([0,1])$
		with Lebesgue measure
		$
		\leb_1 (U)\ge 1-\beta (R,y)
		$
		such that for $u\in U$,
		\begin{equation}\label{eq:Tconst}
			T^R (y,x_1,u) = T^R (y,x_2,u),\,\,y\in\Y,\,x_1,x_2\in\stackrel{-1}{V}([0,R]).
		\end{equation}
	\end{lemma}
	\begin{proof}
		The proof can be found in Appendix C of \cite{lovas2024transition}.  
	\end{proof}
	
	One key difference from the arguments described in \cite{lovas2024transition} is that, for now, we only assume that the parametric kernel $Q$ satisfies the minorization condition \eqref{eq:smallset} with $R>0$, which will be specified later. 
	Furthermore, unlike in \cite{lovas2024transition}, we do not impose the condition $\sup_{y\in\Y} \beta (R,y) < 1$ on the minorization constant. It is crucial to emphasize this distinction. Otherwise, we define the measurable mapping  
	\begin{equation}\label{eq:fgood}
		(x,y,u) \mapsto f(x,y,u) := T^{R}(y,x,u), \quad
		x \in \X, \, y \in \Y, \, u \in [0,1].
	\end{equation}
	
	Let $(\eps_t)_{t\in\N}$ be a sequence of i.i.d. variables uniformly distributed on $[0,1]$, such that the sigma-algebras $\F^{\eps}_{0,\infty}$ and $\sigma (Y_t, X_t, t\in\N)$ are independent.  
	Furthermore, for $s\in\N$ and $x\in\X$, let us introduce the family of auxiliary processes  
	\begin{equation}\label{eq:Zaux}
		Z_{s,t}^{x,\mathbf{y}} = \begin{cases}
			x &\text{ if } t\le s, \\
			f(Z_{s,t-1}^{x,\mathbf{y}},y_{t-1},\eps_t) &\text{ if } t>s,
		\end{cases}
	\end{equation}
	where $\mathbf{y}=(y_0,y_1,\ldots)\in\Y^\N$ can be any fixed trajectory, exactly as we did in \cite{lovas2024transition}.  
	It also holds that for $\mathbf{Y}=(Y_n)_{n\in\N}$, the process $Z_{s,t}^{X_s,\mathbf{Y}}$, ${t\ge s}$ is a version of $(X_t)_{t\ge s}$.
	
	Let $x_1, x_2 \in \mathcal{X}$ be two distinct initial values, and let $\mathbf{y} \in \mathcal{Y}^\mathbb{N}$ be a fixed trajectory. Our first goal is to estimate the non-coupling probability 
	$\P (Z_{j,j+n}^{x_1,\mathbf{y}} \ne Z_{j,j+n}^{x_2,\mathbf{y}}),$
	for a fixed $j\in\N$. 
	
	According to the following lemma, if a parametric kernel $Q$ satisfies the drift condition \eqref{eq:drift}, then the iterates of $Q$ also satisfy the drift condition, albeit with different coefficients.
	\begin{lemma}\label{lem:Viter}
		Assume that the parametric kernel $Q: \Y \times \X \times \B (\X) \to [0,1]$ satisfies the drift condition \eqref{eq:drift} with $\gamma, K: \Y \to (0, \infty)$. Then for $x \in \X$, $\mathbf{y} \in \Y^\mathbb{N}$, and $k, l \in \mathbb{N}$ with $l < k$, we have
		\[
		[Q_{l,k}(\mathbf{y})V](x) := [Q(y_{k-1}) \ldots Q(y_l) V](x) \le \prod_{r=l}^{k-1} \gamma(y_r) V(x) + \sum_{r=l}^{k-1} K(y_r) \prod_{j=r+1}^{k-1} \gamma(y_j).
		\]
	\end{lemma}
	
	\begin{proof}
		The detailed proof can be found in the Appendix of \cite{lovas2024transition} (see Lemma C. 1.) or in \cite{lovas} (see Lemma 7.4).
	\end{proof}
	
	\subsection{Coupling at random times}\label{subsec:cpl_random}
	
	For notational convenience, we set
	\begin{equation}\label{eq:gaKbeta}
	(\ga_t, K_t, \beta_t(R)) = (\ga(y_t), K(y_t), \beta(R, y_t)),\quad t\in\N.
	\end{equation}
	For $0 \leq j \leq t$ and $C>1$, we further introduce
	\begin{equation}\label{eq:gaS}
	\ga_{t,j}(C) = \sup_{C \leq l \leq t-j} \ga_{t-1} \cdots \ga_{t-l},
	\quad
	S_{t,j} = K_{t-1} + \sum_{l=1}^{t-j-1} \ga_{t-1} \cdots \ga_{t-l} K_{t-l-1}.
	\end{equation}
	Given a pair $(C, \bar{\beta})$, we define the sequence of random times $C+j = \tilde{\tau}_0 < \tilde{\tau}_1 < \dots$, depending solely on the trajectory $\mathbf{y}$, such that for every $i \geq 1$,
	\begin{equation}\label{eq:tau_tilde}
	\ga_{\tilde{\tau}_i,j}(C) \leq 1 - \frac{1}{C},
	\quad
	S_{\tilde{\tau}_i,j} \leq C,
	\quad
	\beta_{\tilde{\tau}_i}(R_C) \leq \bar{\beta},
	\end{equation}
	where $R_C > 0$ is to be specified later. 
	At this point, it is not clear whether $\tilde{\tau}_i$ is almost surely finite for every $i \in\N$. This will follow from Assumption \ref{as:main_assumptions}, as will be explained later.
	
	Let us define the subsequence $\tau_i=\tilde{\tau}_{iC}$, $i=0,1,\ldots$, and consider the submapping $Z_i^\iota := Z_{j,\tau_i}^{x_{\iota},\mathbf{y}}$ for $\iota=1,2$ and $i \in \mathbb{N}$. Furthermore, for brevity, we introduce the shorthand notation 
	\[
	\overline{Z}_i := (Z_i^1, Z_i^2), \quad \lVert \overline{Z}_i \rVert := \max \left( V(Z_i^1), V(Z_i^2) \right).
	\]
	Note that, $\left(\overline{Z}_i\right)_{i\in\N}$ also forms a time-inhomogeneous Markov chain. By Lemma \ref{lem:Viter} and the definition of the random time sequence $(\tau_i)_{i\ge 0}$, we have for $\iota=1,2$:
	\begin{align}\label{eq:tauchain_drift}
		\begin{split}	
		\E [V(Z_i^\iota)\mid Z_{i-1}^\iota] &=
		[Q_{\tau_{i-1},\tau_i}(\mathbf{y})V](Z_{j,\tau_{i-1}}^{x_\iota,\mathbf{y}})= [Q(y_{\tau_i-1})\dots Q(y_{\tau_{i-1}})V](Z_{j,\tau_{i-1}}^{x_\iota,\mathbf{y}})
		\\
		&\le \prod_{k=\tau_{i-1}}^{\tau_i-1}\ga (y_k)V(Z_{j,\tau_{i-1}}^{x_\iota,\mathbf{y}})
		+
		\sum_{k=\tau_{i-1}}^{\tau_i-1}K(y_k)\prod_{l=k+1}^{\tau_i-1}\ga (y_l)
		\\
		&\le\ga_{\tau_i,j}(C) V(Z_{j,\tau_{i-1}}^{x_\iota,\mathbf{y}}) + S_{\tau_i,j}
		\le \left(1-\frac{1}{C}\right) V(Z_{i-1}^\iota) + C, \quad i\ge 1.
		\end{split}
	\end{align}
	
	Furthermore, the non-coupling probabilities on small sets of the form $\{\lVert \overline{Z}_i \rVert \le R_C\}$ can be estimated using Lemma~\ref{lem:T} and the deifnition of random times $(\tilde{\tau}_i)_{i\ge 1}$~\eqref{eq:tau_tilde}, as follows:
	\begin{align}\label{eq:glue}
		\begin{split}
			\P \left(Z_{i+1}^1\ne Z_{i+1}^2\mid \overline{Z}_i\right)
			\ind_{\lVert \overline{Z}_i \rVert\le R_C}
			&\le 
			\P \left(Z_{j,\tau_i+1}^{x_1,\mathbf{y}}\ne Z_{j,\tau_i+1}^{x_2,\mathbf{y}}\mid \overline{Z}_i\right)
			\ind_{\lVert \overline{Z}_i \rVert\le R_C}
			\\
			&=\P \left(f(Z_i^1,y_{\tau_i},\eps_{\tau_i+1})\ne f(Z_i^2,y_{\tau_i},\eps_{\tau_i+1})\mid \overline{Z}_i\right)\ind_{\lVert \overline{Z}_i \rVert\le R_C}
			\\
			&\le \P (\eps_{\tau_i+1}\notin U) = 1-\leb_1 (U)\le \beta_{\tau_i} (R_C)\le \bar{\beta}.
		\end{split}
	\end{align}
	
	\begin{lemma}\label{lem:tau_chain}
		There exist constants $M > 0$ and $\rho \in (0,1)$, depending on the choice of $C$ and $\bar{\beta}$, such that
		$$
		\P (Z_{j,j+n}^{x_1,\mathbf{y}}\ne Z_{j,j+n}^{x_2,\mathbf{y}})
		\le 
		M(1+V(x_1)+V(x_2))\rho^{L_{n,j}},\quad n\in\N,
		$$
		where the quantity $L_{n,j}$ denotes the number of such random times not exceeding $n+j$:
		\begin{equation}\label{eq:Lnj}
			L_{n,j} = \max\{i \mid \tau_i \leq n+j\}.
		\end{equation}
	\end{lemma}
	\begin{proof}
		
		Let us fix $R_C:=4C^2$, and introduce the sequence of successive visiting times
		$$
		\sigma_0:=0,\,\sigma_{k+1} = \min\left\{i>\sigma_k\middle|  \lVert \overline{Z}_i\rVert\le R_C
		\right\},\,k\in\N
		$$
		that are obviously stopping times with respect to the natural filtration of the process $(\overline{Z}_i)_{i\in\N}$. 
		Note that, for every $i\ge 1$, on the event $\{ \lVert \overline{Z}_i\rVert>R_C\}$, 
		\begin{equation}\label{eq:RCV}
		R_C \le \max (V(Z_i^1),V(Z_i^2))\le V(Z_i^1)+V(Z_i^2)	
		\end{equation}
		hence by the definition of $R_C$, we have
		\begin{equation}\label{eq:contr}
		\left(1-\frac{1}{C}\right)(V(Z_i^1)+V(Z_i^2))+2C\le 
		\left(1-\frac{1}{2C}\right)(V(Z_i^1)+V(Z_i^2)).
		\end{equation}
		By Markov inequality and \eqref{eq:RCV}, for $k\ge 1$ and $s\ge 0$, we obtain
		\begin{align*}
		\P (\sigma_{k+1}-\sigma_{k}>s\mid \overline{Z}_{\sigma_k})
		&=
		\E \left[ \prod_{l=1}^{s}
		\ind_{\lVert \overline{Z}_{\sigma_k+l}\rVert>R_C}
		\middle| \overline{Z}_{\sigma_k}\right]
		\\
		&\le
		\frac{1}{R_C}\E \left[\left(V(Z_{\sigma_k+s}^1)+V(Z_{\sigma_k+s}^2)\right) \prod_{l=1}^{s-1}
		\ind_{\lVert \overline{Z}_{\sigma_k+l}\rVert>R_C}
		\middle| \overline{Z}_{\sigma_k}\right].
		\end{align*}
		
		Using the tower rule and the strong Markov property of the sequence $(\overline{Z}_i)_{i\in\N}$, by the drift property \eqref{eq:tauchain_drift} and the estimate \eqref{eq:contr}, we can write
		\begin{align*}
			\E \left[\left(V(Z_{\sigma_k+s}^1)+V(Z_{\sigma_k+s}^2)\right) \prod_{l=1}^{s-1}
			\ind_{\lVert \overline{Z}_{\sigma_k+l}\rVert>R_C}
			\middle| \overline{Z}_{\sigma_k}\right] &=
			\\
			\E \left[\E\left(V(Z_{\sigma_k+s}^1)+V(Z_{\sigma_k+s}^2)
			\mid \F_{0,\sigma_{k}+s-1}^{\overline{Z}}
			\right) \prod_{l=1}^{s-1}
			\ind_{\lVert \overline{Z}_{\sigma_k+l}\rVert>R_C}
			\middle| \overline{Z}_{\sigma_k}\right] 
			&=
			\\
			\E \left[\left(
			\E\left[V(Z_{\sigma_k+s}^1)\mid Z_{\sigma_k+s-1}^1\right] 
			+
			\E\left[V(Z_{\sigma_k+s}^2)\mid Z_{\sigma_k+s-1}^2\right]
			\right) \prod_{l=1}^{s-1}
			\ind_{\lVert \overline{Z}_{\sigma_k+l}\rVert>R_C}
			\middle| \overline{Z}_{\sigma_k}\right] 
			&\le 
			\\
			\E \left[\left(\left(1-\frac{1}{C}\right)(V(Z_{\sigma_k+s-1}^1)+V(Z_{\sigma_k+s-1}^2))+2C
			\right) \prod_{l=1}^{s-1}
			\ind_{\lVert \overline{Z}_{\sigma_k+l}\rVert>R_C}
			\middle| \overline{Z}_{\sigma_k}\right] 
			&\le 
			\\
			\left(1-\frac{1}{2C}\right)
			\E \left[\left(V(Z_{\sigma_k+s-1}^1)+V(Z_{\sigma_k+s-1}^2)\right) \prod_{l=1}^{s-2}
			\ind_{\lVert \overline{Z}_{\sigma_k+l}\rVert>R_C}
			\middle| \overline{Z}_{\sigma_k}\right]&.
		\end{align*}
		
		Iteration of this argument in $s-1$ steps leads to the following estimation:
		\begin{align}\label{eq:delta_sig}
			\begin{split}
			\P (\sigma_{k+1}-\sigma_{k}>s\mid \overline{Z}_{\sigma_k}) 
			&\le
			\frac{1}{R_C}\left(1-\frac{1}{2C}\right)^{s-1}
				\E \left[V(Z_{\sigma_k+1}^1)+V(Z_{\sigma_k+1}^2) 
			\middle| \overline{Z}_{\sigma_k}\right]
			\\
			&\le 
			\left(1-\frac{1}{2C}\right)^{s-1}\frac{\left(1-\frac{1}{C}\right) \left[V( Z_{\sigma_k}^{1})+ V( Z_{\sigma_k}^{2})\right]+2C}{R_C} \\
			&\le 
			\left(1-\frac{1}{2C}\right)^{s-1}\frac{\left(1-\frac{1}{C}\right)2R_C+2C}{R_C}
			\le
		 2\left(1-\frac{1}{2C}\right)^{s}.
		 \end{split} 
		\end{align}
		Along similar lines, we can show that 
		$$
		\P (\sigma_1>s)\le \frac{1}{R_C}\left(1-\frac{1}{2C}\right)^{s-1}\left[\E[V(Z_1^1)]+\E[V(Z_1^2)]\right],
		$$
		where since $C>1$ and $\tau_1=\tilde{\tau}_{C}$, by Lemma \ref{lem:Viter}, \eqref{eq:gaS} and also by \eqref{eq:tau_tilde}, we can estimate further by writing
		\begin{align*}
			\E[V(Z_1^{\iota})]
			&=\E\left[V\left(Z_{j,\tau_1}^{x_{\iota},\mathbf{y}}\right)\right]
			=
			\left[Q_{j,\tau_1}(\mathbf{y})V\right]\left(Z_{j,j}^{x_{\iota},\mathbf{y}}\right)
			=
			\left[Q_{j,\tau_1}(\mathbf{y})V\right]\left(x_{\iota}\right)
			\\
			&\le \prod_{k=j}^{\tau_1-1}\ga (y_k) V(x_{\iota})
			+
			\sum_{k=j}^{\tau_1-1}K(y_k)\prod_{l=k+1}^{\tau_1-1}\ga (y_l)
			\\
			&\le
			\ga_{\tilde{\tau}_C,j}(C)V(x_{\iota})+S_{\tilde{\tau}_C,j}
			\le \left(1-\frac{1}{C}\right)V(x_{\iota})+C,\quad \iota=1,2
		\end{align*}
		hence we have
		\begin{align}\label{eq:sig_ge_s}
		\begin{split} 
		\P (\sigma_1>s)&\le \frac{1}{R_C}\left(1-\frac{1}{2C}\right)^{s-1}\left[\left(1-\frac{1}{C}\right)\left(V(x_1)+V(x_2)\right)+2C\right]
		\\
		&\le 
		\frac{2C}{4C^2}\left(1-\frac{1}{2C}\right)^{s-1}
		\left[V(x_1)+V(x_2)+1\right]
		\\
		&\le 
		\left[V(x_1)+V(x_2)+1\right]\left(1-\frac{1}{2C}\right)^{s}.
		\end{split}
		\end{align}
		
		Using the estimate \eqref{eq:delta_sig}, for the generating function of the time elapsed between the $k$th and $(k+1)$th visits, we get
		\begin{align}\label{eq:mom_gen1}
		\begin{split}
			\E \left(\frac{1}{\left(1-\frac{1}{4C}\right)^{\sigma_{k+1}-\sigma_k}}\middle| \F_{-\infty,\tau_{\sigma_k}}^\eps\right) 
			&= \sum_{l=1}^{\infty}
			\frac{1}{\left(1-\frac{1}{4C}\right)^{l}} \P (\sigma_{k+1}-\sigma_{k}=l\mid \overline{Z}_{\sigma_k}) 
			\\
			&\le 
			\sum_{l=1}^{\infty} 	\frac{2 \left(1-\frac{1}{2C}\right)^{l-1}}{\left(1-\frac{1}{4C}\right)^{l}} = 8C,\,k\ge 1,
		\end{split}
		\end{align} 
		and similarly, by \eqref{eq:sig_ge_s}, for $k=0$,
		\begin{align}\label{eq:mom_gen2}
		\begin{split}
		\E \left(\frac{1}{\left(1-\frac{1}{4C}\right)^{\sigma_1}}\right)
		&=\sum_{l=1}^{\infty}
		\frac{1}{\left(1-\frac{1}{4C}\right)^{l}}
		\P (\sigma_1=l)
		\le 
		\sum_{l=1}^{\infty}
		\frac{1}{\left(1-\frac{1}{4C}\right)^{l}}
		\P (\sigma_1>l-1)
		\\
		&\le 
		\left[V(x_1)+V(x_2)+1\right]
		\sum_{l=1}^{\infty}
		\frac{\left(1-\frac{1}{4C}\right)^{l-1}}{\left(1-\frac{1}{4C}\right)^{l}}
		= 4C \left[V(x_1)+V(x_2)+1\right].
		\end{split}
		\end{align}
		
		Using the recently obtained bounds \eqref{eq:mom_gen1} and \eqref{eq:mom_gen2},
		by the Markov inequality and the tower rule, for $0<m<n$, we have
		\begin{align*}
			\P (\sigma_m\ge n)&\le \left(1-\frac{1}{4C}\right)^n \E \left(\frac{1}{\left(1-\frac{1}{4C}\right)^{\sigma_m}}\right) 
			\\
			&=
			\left(1-\frac{1}{4C}\right)^n \E \left( \E \left[\frac{1}{\left(1-\frac{1}{4C}\right)^{\sigma_m-\sigma_{m-1}}}
			\middle|\F_{-\infty,\tau_{\sigma_{m-1}}}^\eps 
			\right]  \frac{1}{\left(1-\frac{1}{4C}\right)^{\sigma_{m-1}}}\right)
			\\
			&\le
			\left(1-\frac{1}{4C}\right)^n \times 8C\times\E \left(\frac{1}{\left(1-\frac{1}{4C}\right)^{\sigma_{m-1}}}\right)
			\le\ldots\\
			&\le 
			\frac{1}{2}\left[V(x_1)+V(x_2)+1\right](8C)^m\left(1-\frac{1}{4C}\right)^n
		\end{align*}
		Let us define
		$$
		m_n := \left\lfloor n\frac{\log \left(1-\frac{1}{8C}\right)-\log \left(1-\frac{1}{4C}\right)}{\log (8C)}\right\rfloor.
		$$
		Obviously, $(8C)^{m_n}\left(1-\frac{1}{4C}\right)^n\le \left(1-\frac{1}{8C}\right)^n$, and thus for $n$ is so large such that $m_n\ge 1$, we have
		$$
		\P (\sigma_{m_n}\ge n)\le 
		\frac{1}{2}\left[V(x_1)+V(x_2)+1\right]\left(1-\frac{1}{8C}\right)^n.
		$$
		
		For every $l$, $\eps_{\tau_{\sigma_l}+1}$ is independent of $\F_{0,\tau_{\sigma_l}}^\eps$ hence by \eqref{eq:glue}, we can estimate the probability of no-coupling on events when the small set is visited at least $m_n$-times as follows:
		\begin{align*}
			\P (Z_n^1\ne Z_n^2,\sigma_{m_n}<n)\le \bar{\beta}^{m_n}.
		\end{align*}
		
		Combine this bound with that one what we got for the tail probability of the visiting times, and obtain
		\begin{align*}
			\P (Z_n^1\ne Z_n^2)&\le \P (Z_n^1\ne Z_n^2,\sigma_{m_n}<n) + \P (\sigma_{m_n}\ge n) 
			\\
			&\le
			\bar{\beta}^{m_n}
			+
			\frac{1}{2}\left[V(x_1)+V(x_2)+1\right]\left(1-\frac{1}{8C}\right)^n,
		\end{align*}
		for $n$ being so large such that $m_n\ge 1$. From this, we can conclude that
		$$
			\P (Z_n^1\ne Z_n^2)\le M \left[V(x_1)+V(x_2)+1\right] \rho^n,\quad n\in\N
		$$
		holds with appropriate constants $M > 0$ and $\rho \in (0,1)$ depending on $\bar{\beta}\in (0,1)$ and $C>1$.
			
		Finally, using the definition of $L_{n,j}$ \eqref{eq:Lnj}, we get
		\begin{align*}
			\P (Z_{j,j+n}^{x_1,\mathbf{y}}\ne Z_{j,j+n}^{x_2,\mathbf{y}})
			&\le \P (Z_{j,\tau_{L_{n,j}}}^{x_1,\mathbf{y}}\ne Z_{j,\tau_{L_{n,j}}}^{x_2,\mathbf{y}})
			\\
			&=\P (Z_{L_{n,j}}^1\ne Z_{L_{n,j}}^2)
			\le
			M \left[V(x_1)+V(x_2)+1\right] \rho^{L_{n,j}},\quad n\in\N
		\end{align*}
		which completes the proof.
	\end{proof}
	
	\subsection{Controlling the deviation of $\tilde{L}_{n,j}$}

Using the notations already introduced in the previous section, we set for some $R=R_C$, 
$$q_{t,j}(C,\beta,R_C)=\P\left(\gamma_{j,t}(C)\leq 1- C^{-1}, S_{t,j}\leq C, \beta_t(R_C)\leq \beta\right).$$

\begin{lemma}\label{lowb}
There exist $C_0>1$ such that for any $C\geq C_0$, one can find $\beta=\beta_C\in (0,1)$ such that $\inf_{j\geq 0}\inf_{t\geq C+j}q_{t,j}\left(C,\beta,R_C\right)>0$.
\end{lemma}

\paragraph{Proof of Lemma \ref{lowb}.} Let $C>1$ and set $p_{t,j}(C)=\P\left(\gamma_{j,t}(C)\leq 1- C^{-1}, S_{t,j}\leq C\right)$. 
We have 
$$\sup_{j\geq 0}\sup_{t\geq C+j-1}\left\{p_{t,j}(C)-q_{t,j}(C,\beta,R_C)\right\}\leq \sup_{t\geq 0}\P\left(\beta_t(R_C)>\beta\right)$$
which goes to $0$ when $\beta\nearrow 1$.
Moreover
$$p_{t,j}(C)\geq  1-\P\left(\gamma_{t,j}(C)>1-C^{-1}\right)-\P\left(S_{t,j}>C\right).$$
Moreover, from Markov inequality and the union bound, we have for $C\geq 2$,
$$\P\left(\gamma_{t,j}(C)>1-C^{-1}\right)\leq \sum_{\ell\geq C} \left(1-C^{-1}\right)^{-s}d_{\ell}(s)\leq 2^s r_C(s) .$$
and
$$\P\left(S_{t,j}>C\right)\leq C^{-s}\E\left(S_{t,j}^s\right)\leq C^{-s}r_0(s).$$
We deduce that if $C$ is large enough, say $C\geq C_0$, then $\inf_{j\geq 0}\inf_{t\geq C+j-1}p_{t,j}(C)>0$. For such a $C$, the first part of the proof 
allows to conclude the proof, choosing $\beta=\beta_C$ sufficiently close to $1$.$\square$

Next, for a pair $(C,\beta)$ ensuring that $\inf_{j\geq 0}\inf_{t\geq C+j}q_{t,j}\left(C,\beta,R_C\right)>0$, we denote by $\widetilde{L}_{n,j}(\omega)$ the number of time points $C+j\leq t\leq n+j$ such that 
$$\gamma_{t,j}(C)_{\omega}\leq 1-1/C,\quad S_{t,j}(\omega)\leq C\mbox{ and }\beta_t(R_C)_{\omega}\leq \beta.$$ 
Our aim is  to control the deviation of $\widetilde{L}_{n,j}$. One could do this by using the $\alpha-$mixing property of $Y$. 

We recall the notation
$$r_n:=\sum_{\ell \geq n}d_{\ell}.$$
In the following lemma, we give a general upper bound for the quantity $b(n)$ and then for the probability that the coupling time exceeds $n$. This upper bound involves quantities of type $r_i$ and we also give sufficient conditions which entail the validity of {\bf A1}. Controlling the decay of this quantity is difficult in general because it involves expectations of products of dependent random variables. We then precise some rates under specific conditions on the drift parameters, e.g. when the random drift coefficients takes values in $[0,1]$ (but without being necessarily uniformly smaller than $1$) or when the environment satisfies uniformly mixing conditions such as $\phi-$mixing properties.

\begin{lemma}\label{conc}
Suppose that conditions (\ref{eq:drift}) and (\ref{eq:smallset}) are valid for a drift function $V$ and that Assumption {\bf A2} is satisfied.  Set $H_{x_1}(X_j)=M\left(1+V(x_1)+V(X_j)\right)$ where $x_1\in \X$ and 
$M$ is the positive constant defined in Lemma \ref{lem:tau_chain}. The following assertions are valid.
\begin{enumerate}
\item
There exists two three positive constants $c_1,c_2,c_3$ and a constant $b\in (0,1)$ such that for any $n\geq C$,  
$$\P\left(\widetilde{L}_{n,j}\leq c_1 n\right)\leq c_2\min_{C\leq i<\widetilde{q}<c_3 n}\left\{r_i+b^{n/\widetilde{q}}+\alpha^Y\left(\widetilde{q}+1-i\right))\right\}.$$
In particular, we deduce that for the coupling of Lemma \ref{lem:tau_chain},
$$b(n) \leq \sup_{j\geq 0}\E\left[H_{x_1}(X_j)\right]\rho^{c_1 n}+ c_2\min_{C\leq i<\widetilde{q}<c_3 n}\left\{r_i+b^{n/\widetilde{q}}+\alpha^Y\left(\widetilde{q}+1-i\right))\right\}.$$
\item
Suppose that $\sup_{j\geq 0}\E K_j^k<\infty$ for some $k>1$ and that the sequence $\left(\gamma_j\right)_{j\geq 1}$ takes values in the interval $[0,1]$ and that 
there exists $\eta\in (0,1)$ such that $\inf_{j\geq 1}\P\left(\gamma_j<\eta\right)>0$.  If $\alpha_Y(m)=O(m^{-a})$ for some $a>1$, then for $b=a(k-1)/k$, we get $b(n)=O\left(\left(\log n\right)^{2b-1}  n^{-b+1}\right)$. 
Now if $\alpha_Y(m)=O\left(\kappa^m\right)$ for some $\kappa\in (0,1)$, then there exists $\overline{\kappa}\in (0,1)$ such that $b(n)=O\left(\overline{\kappa}^{\frac{\sqrt{n}}{\log n\log\log n}}\right)$.
\item 
Suppose that $\sup_{j\geq 0}\Vert \gamma_j\Vert_{\infty}<\infty$, there exists $k>0$ such that $\sup_{j\geq 0}\E K_j^k<\infty$ and $\sup_{j\geq 0}\E\left[\gamma_j\right]<1$. Suppose furthermore that $\lim_{n\rightarrow \infty}\phi^{Y}(n)=0$.
Then, replacing $V$ by $V^s$ with $s>0$ sufficiently small, we have $d_{\ell}=O\left(\rho^{\ell}\right)$ for some $\rho\in (0,1)$ and there exists $\overline{\kappa}\in (0,1)$ such that $b(n)=O\left(\overline{\kappa}^{\sqrt{n}}\right)$. 
\item 
Finally, suppose that the long-term contractivity condition holds true, i.e. 
$d_{\ell}=O\left(\rho^{\ell}\right)$ for some $\rho\in (0,1)$.  If $\alpha_Y(m)=O(m^{-a})$ for some $a>1$, then $b(n)=O\left(\log^a(n) n^{-a}\right)$. 
Now if $\alpha^Y(m)=O\left(\kappa^m\right)$ for some $\kappa\in (0,1)$, then there exists $\overline{\kappa}\in (0,1)$ such that $$b(n)=O\left(\overline{\kappa}^{\sqrt{n}}\right).$$

\end{enumerate}
\end{lemma}

\paragraph{Note.}
Point $4$ of Lemma \ref{conc} is satisfied for instance when 
$\Vert \gamma_t\Vert_{\infty}\leq \lambda_t$ with
$$\chi:=\limsup_n \sup_{j\geq 0}\left\{\frac{1}{n}\sum_{t=j+1}^{j+n} \log\left(\lambda_t\right)\right\}<0.$$
In particular if $\lambda_t\leq \lambda<1$ for any $t\geq 0$ but one can consider some inhomogeneous cases for which $\lambda_t$ exceeds $1$ from time to time.

\paragraph{Proof of Lemma \ref{conc}.}
\begin{enumerate}
\item
Since  $\widetilde{L}_n$ is a partial sum of binary variables, each involving an unbounded number of coordinates, we first approximate 
$\widetilde{L}_n$. We  set $S_{t,j}=S_{t,j,i}+T_{t,j,i}$ with
 $$S_{t,j,i}=\sum_{\ell=1}^{\min(i,t-j-1)} \gamma_{t-1}\cdots \gamma_{t-\ell} K_{t-\ell-1},$$ with $C\leq i\leq n$ to fix latter and 
$$H_{t,j,i}=\sup_{C\leq \ell\leq \min(i,t-j)}\gamma_{t-1}\ldots\gamma_{t-\ell},\quad G_{t,j,i}=\sup_{\min(i+1,t-j+1)\leq \ell\leq t-j}\gamma_{t-1}\ldots\gamma_{t-\ell},$$
 with the convention that $T_{t,j,i}=G_{j,j,i}=0$ if $t\leq i+j$. One can use the inequalities
$$\mathds{1}_{S_{t,j}\leq C}\geq \mathds{1}_{S_{t,j,i}\leq C/2}-\mathds{1}_{T_{t,j,i}\geq C/2},$$ 
$$\mathds{1}_{\gamma_{t,j}\leq 1-1/C}\geq \mathds{1}_{H_{t,j,i}\leq 1-1/C}-\mathds{1}_{G_{t,j,i}\geq 1-1/C}$$
to get
\begin{eqnarray*}
\widetilde{L}_{n,j}&\geq &\sum_{t=C+j}^{n+j} \mathds{1}_{S_{t,j,i}\leq C/2,H_{t,j,i}\leq 1-1/C,\beta_t(R_C)\leq \beta}-\sum_{t=C+j}^{n+j}\mathds{1}_{T_{t,j,i}\geq C/2}-\sum_{t=C+j}^{n+j} \mathds{1}_{G_{t,j,i}\geq 1-1/C}\\
&:=&L_{n,j,1}-L_{n,j,2}-L_{n,j,3}.
\end{eqnarray*}
For some positive integers $k$ and $k'$, we have
\begin{eqnarray*}
\P\left(\widetilde{L}_{n,j}\leq k\right)&\leq& \P\left(L_{n,j,1}\leq k+L_{n,j,2}+L_{n,j,3}\right)\\
&\leq& \P\left(L_{n,j,1}\leq k+2k'\right)+\P\left(L_{n,j,2}>k'\right)+\P\left(L_{n,j,3}>k'\right):=p_{j,1}+p_{j,2}+p_{j,3}.
\end{eqnarray*}

From Markov's inequality
$$p_{j,2}\leq \sum_{t=i+j}^{n+j}\P\left(T_{t,j,i}\geq C/2\right)/k'\leq \frac{2n}{Ck'}r_i.$$
We deduce that 
$$p_{j,2}\leq \frac{D n r_i}{k'}$$
for a suitable constant $D>0$. 
On the other hand, we have 
\begin{eqnarray*}
p_{j,3}&\leq& \sum_{t=i+j+1}^{n+j}\P\left(G_{t,j,i}\geq 1-1/C\right)/k'\\
&\leq& \frac{n}{k'}\sum_{\ell=i+1}^{t-j}\P\left(\gamma\left(Y_{t-1}\right)\cdots\gamma\left(Y_{t-\ell}\right)\geq 1-1/C\right)\\
&\leq & \frac{n}{k'(1-1/C)}\sum_{\ell=i+1}^{t-j}d_{\ell}\\
&\leq & \frac{D'n r_i}{k'}
\end{eqnarray*}
for another suitable constant $D'>0$. We have used the fact that $K\geq 1$ and Markov inequality here.
It now remains to bound $p_1$. Setting
\begin{eqnarray*}
q&:=&\inf_{j\geq 0, t\geq j+C}\P\left(S_{t,j}\leq C/2,\gamma_{t,j}(C)\leq 1-1/C,\beta_t(R_C)\leq\beta\right)\\
&\leq& \inf_{j\geq 0, t\geq j+C}\P\left(S_{t,j,i}\leq C/2,H_{t,j,i}\leq 1-1/C,\beta_t(R_C)\leq \beta\right),
\end{eqnarray*}
we have $q>0$ for a large $C$, and 
$$U_{t,j,i}=\mathds{1}_{S_{t,j,i}\leq C/2,H_{t,j,i}\leq 1-1/C,\beta_t(R_C)\leq \beta}-\P\left(S_{t,j,i}\leq C/2,H_{t,j,i}\leq 1-1/C,\beta_t(R_C)\leq \beta\right),$$
we have, setting $\widetilde{n}=n-C+1$,
$$p_{j,1}\leq  \P\left(\sum_{t=j+C}^{j+n} U_{t,j,i}\leq k+2k'-\widetilde{n}q\right)\leq \P\left(\sum_{t=j+C}^{j+n} U_{t,j,i}/\widetilde{n}\leq d-q\right),$$
where we chose $k=k'=[d\widetilde{n}/3]$ with some $0<d<q$. Note that $k\geq c_1 n$ for a suitable constant $c_1>0$.

We then deduce the existence of  $\overline{D}>0$ such that
\begin{equation}\label{goodinter}
\P\left(\widetilde{L}_n\leq c_1 n\right)\leq  \inf_{C\leq i\leq n}\left\{\overline{D} r_i+\P\left(\sum_{t=j+C}^{n+j} U_{t,j,i}/\widetilde{n}\leq d-q\right)\right\}.
\end{equation}

We next use an exponential inequality for strongly mixing sequences. We first note that 
$$\alpha_{U_{\cdot,j,i}}(m)\leq \alpha_Y(m-i),\quad m\geq i.$$

Using inequality \eqref{sufexp1}, we deduce that there exist two positive constants $\widetilde{C}_1,\widetilde{C}_2$ such that
$$\P\left(\sum_{t=1}^{\widetilde{n}} U_{t,i}/\widetilde{n}\leq d-q\right)\leq\widetilde{C}_1\inf_{i<\widetilde{q}<2(q-d)\widetilde{n}/7} \left\{\exp\left(-\widetilde{C}_2 \widetilde{n}/\widetilde{q}\right)+\alpha_Y\left(\widetilde{q}+1-i\right)\right\}.$$
Finally, we use the fact that $\inf_{n\geq C}\widetilde{n}/n$ is lower bounded by $1/C$ to get the first bound. 

The second bound of the lemma can be obtained using the event $\left\{\widetilde{L}_{n,j}\leq c_1n\right\}$ and its complement and using the fact that 
$$\E\left[\min\left\{H_{x_1}(X_j)\rho^{L_{n,j}},1\right\}\right]\leq \P\left(L_{n,j}\leq c_1n\right)+\E\left[H_{x_1}(X_j)\right]\rho^{c_1n}.$$ 
\item
The second point follows from Lemma \ref{products} $4.$ and the previous point, choosing $i=[\widetilde{q}/2]$ and $\widetilde{q}\sim c n/\log n$  for power decays and $\widetilde{q}\sim\sqrt{n}$ for geometric decays.
\item 
For $s\in (0,1)$ to be fixed later, we denote by $d_{\ell}(s)$ and $r_i(s)$ the same quantities as $d_{\ell}$ and $r_i$, replacing $(\gamma_t,K_t)$ by $\left(\gamma_t^s,K_t^s\right)$ for any $t\geq -1$. From Holder inequality, we have
$$d_{\ell}(s)\leq \sup_{t\geq 0}\E^{\frac{s}{k}}\left[K_t^k\right]\cdot \E^{\frac{k-s}{k}}\left[\gamma_{t+1}^{\frac{ks}{k-s}}\cdots \gamma_{t+\ell}^{\frac{ks}{k-s}}\right].$$
From Lemma \ref{lem:ga}, one can choose $s$ small enough  
such that $d_{\ell}(s)=O\left(\kappa^{\ell}\right)$ for some $\kappa\in (0,1)$. Fixing such a value for $s$, we impose $i\sim \sqrt{n}$ and we conclude that 
$$\P\left(\widetilde{L}_n\leq c_1 n\right)\leq D'\kappa^{\sqrt{n}}$$
for some $D'>0$ and $\kappa\in (0,1)$. We then get the bound on $b(n)$, proceeding as in the proof of the first point of the lemma.
\item 
For the last point, we use the general bound on $b(n)$ given in the first point of the lemma. Since $r_i=O\left(\rho^i\right)$, we choose $i=[\widetilde{q}/2]$ and $\widetilde{q}\sim c n/\log n$ for power decays of the mixing coefficients and $\widetilde{q}\sim \sqrt{n}$ for geometric decays, which leads to the result.$\square$
\end{enumerate}

$\square$
\bigskip

The following lemma gives some bound for $r_n(s)$ in the specific case of random coefficients $\gamma(Y_t)$ bounded by $1$. This quantity can be bounded from exponential inequalities for mixing random variables. We investigate 
two cases with a power decay or a geometric decay for the $\alpha-$mixing coefficients.

\begin{lemma}\label{products}
Suppose that the sequence $\left(\gamma_j\right)_{j\geq 1}$ takes values in the interval $[0,1]$ with $\sup_{j\geq 1}\E\left[\gamma_j\right]<1$. In what follows, we set 
$$p_n(\delta)=\sup_{j\geq 0}\P\left(\gamma_{j+1}\cdots\gamma_{j+n}>\exp(-n\delta)\right).$$
The following statements are valid.
\begin{enumerate}
\item
There exists some positive constants $\delta_1, \delta_2,c_1$ and $c_2$ such that
$$p_n(\delta_1)\leq c_1\left\{\exp\left(-c_2 n/q\right)+\alpha^{\gamma}(q+1)\right\},$$
for all $q\in (1,\delta_2 n)$.
\item
Suppose that $\alpha_{\gamma}(n)=O\left(e^{-cn}\right)$ for some $c>0$. In this case, there exist some positive constants $c_1,c_2,\delta_1$  such that 
$$p_n(\delta_1)\leq c_1 \exp\left(-\frac{c_2 n}{\log n\log\log n}\right).$$
\item
Suppose that there exists $\lambda>0$ such that $L:=\sup_{j\geq 0}\E\left[\exp(\lambda K_j)\right]<\infty$. Then, for any $\delta_1>0$, we have the bound
$$d_n\leq \sup_{j\geq 0}\E(K_j)e^{-n\delta_1}-\frac{L}{\lambda}p_n(\delta_1)\log p_n(\delta_1)+L p_n(\delta_1).$$
\begin{itemize}
\item If $\alpha^Y(m)=O(m^{-a})$ then $d_n=O\left(n^{-a}\log^{a+1}n\right)$ and $r_n=O\left(n^{-a+1}\log^{a+1} n \right)$. 
\item  If $\alpha^Y(m)=O\left(\kappa^m\right)$ for some $\kappa\in (0,1)$, then there exists $\rho\in (0,1)$ such that $r_n=O\left(\rho^{\frac{n}{\log n\log\log n}}\right)$.
\end{itemize}
\item
Suppose that $L:=\sup_{j\geq 0}\E K_j^k<\infty$ for some $k>1$. Then for any $\delta_1>0$, we have
$$d_n\leq L^{1/k}\left\{p_n(\delta_1)^{\frac{k-1}{k}}+e^{-n s\delta_1}\right\}.$$
\begin{itemize}
\item If $\alpha^Y(m)=O(m^{-a})$, then if any $b:=a(k-1)/k>1$, we have $r_n=O\left(\left(\log n\right)^b/n^{b-1}\right)$. 
\item  If $\alpha^Y(m)=O\left(\kappa^m\right)$ for some $\kappa\in (0,1)$, then there exists $\rho\in (0,1)$ such that $r_n=O\left(\rho^{\frac{n}{\log n\log\log n}}\right)$.
\end{itemize}
\end{enumerate}
\end{lemma}

\paragraph{Proof of Lemma \ref{products}.}
\begin{enumerate}
\item
Under our assumptions, $m:=\sup_{j\geq 1}\E\left[\gamma_j\right]<1$. 
Taking $\eta\in (m,1)$, Markov inequality yields to $\sup_{j\geq 1}\P\left(\gamma_j\geq \eta\right)<1$ and then $\inf_{j\geq 1}\P\left(\gamma_j<\eta\right)>0$.
Let $\epsilon\in (0,1)$ and set $W_j=\log \max(\epsilon,\gamma_j)$, $Z_j=W_j-m_j$ with $m_j=\E W_j$.  Observe that $\sup_{j\geq 1}m_j\leq -h$ where 
$$h=-\log\left(\max(\epsilon,\eta)\right)\inf_{j\geq 1}\P\left(\gamma_j<\eta\right)>0.$$
Define $\delta_1=h':= h/2$. We have 
$$\P\left(\gamma_{j+1}\cdots\gamma_{j+n}>\exp(-\delta_1 n)\right)\leq \P\left(\sum_{i=1}^n Z_{j+i}>nh'\right).$$
The result is then a consequence of inequality (\ref{sufexp1}).
\item
The proof is similar to that of the first point but using Proposition \ref{exp2} instead of inequality (\ref{sufexp1}).
\item
For the bound on $d_n$, we use the inequality
$$\E\left[K_j\gamma_{j+1}\cdots \gamma_{j+n}\right]\leq \E(K_j)e^{-n \delta}+\int_0^{\infty}\min\left\{p_n(\delta_1),\P(K_j>t)\right\}dt.$$
Setting $t_n=-\frac{\log p_n(\delta_1)}{\lambda}$ and using the bound $\P(K_j>t)\leq L e^{-\lambda t}$, we obtain
$$\E\left[K_j\gamma_{j+1}\cdots \gamma_{j+n}\right]\leq  \E(K_j)e^{-n \delta}+ L t_n p_n(\delta_1)+L\int_{t_n}^{\infty}\exp(-\lambda t)dt,$$
which leads to the result. The two other bounds are obtained from the two first points and the comparison between a series and an integral. 
For power decays, one can apply the second point by choosing $q\sim n/\log n$ in order to get the result.  
\item
From Hölder inequality, we get setting $p=k/(k-1)$,
$$\E\left[K_j\gamma_{j+1}\cdots \gamma_{j+n}\right]\leq \E^{1/k}\left(K_j^k\right)\E^{1/p}\left[\gamma_{j+1}^p\cdots\gamma_{j+n}^p\right]$$
and the result follow by considering the event $\left\{\gamma_1\cdots \gamma_n\leq e^{-n \delta}\right\}$ and its complement.
The bounds for $r_n$ follows as in the previous point. 

\end{enumerate}
$\square$

\subsection{End of the proof of Theorem \ref{thm:A12fromMixing}}

From Lemma \ref{conc}, we know that there exists a coupling such that 
$$b(n)\leq \widetilde{c} \rho^{c_1 n}+c_2\min_{C\leq i<q<c_3n}\left\{r_i+b^{n/q}+\alpha^{Y}(q+1-i)\right\},$$
for some positive constants $\widetilde{c},c_1,c_2,c_3$ and $C$ and $\rho,b\in (0,1)$. Setting $\kappa=\max(\rho^{c_1},b)$ and $c=2\max(c_2,\widetilde{c})$, we get 
$$b(n)\leq c\min_{C\leq i<q<c_3n}\left\{r_i+b^{n/q}+\alpha^{Y}(q+1-i)\right\}.$$
Since $r_i$ does not go to $0$ if $i\leq C$ and $b_{n/q}$ does not go to $0$ when $q\geq c_3n$, one can replace the condition $C\leq i<q<c_3 n$ by $1\leq i<q\leq n$ in the expression of the minimum, changing the constant $c$ is necessary. This leads to the result. \qed

\subsection{Proof of Corollary \ref{cor:A12fromMixing}}
The proof is a consequence of Theorem \ref{thm:A12fromMixing} and Lemma \ref{lem:coupl_trans_mix}, choosing $m=[n/2]$. In this case, we have $\alpha^{Y}(m+1)\leq \alpha^{Y}(q-i+1)$ for any $1\leq i<q\leq n-m$. Changing the constant $c$ given in Theorem \ref{thm:A12fromMixing} by $2c$ leads to the result. \qed

\subsection{Proof of Theorem \ref{thm:stac_coupling}}\label{ap:stac_coupling:proof}

According to Lemma \ref{lem:tau_chain}, there  exist constants $M > 0$ and $\rho \in (0,1)$, depending on the choice of $C$ and $\bar{\beta}$, such that
	$$
	\P (Z_{0,n}^{x_1,\mathbf{y}}\ne Z_{0,n}^{x_2,\mathbf{y}})
	\le 
    \min
    \left\{1,
	M(1+V(x_1)+V(x_2))\rho^{L_{n,0}}
    \right\},\quad n\in\N,
	$$
where the exponent $L_{n,0}$ heavily depends on $\mathbf{y}$ but it is independent of $x_1$ and $x_2$. Proceeding in the same way as in the proof of point 1 of Lemma \ref{conc}, we can estimate
\begin{align*}
  \min\left\{1,
	M(1+V(x_1)+V(x_2))\rho^{L_{n,0}}
    \right\}\le \ind_{L_{n,0}(\mathbf{y})\le c_1 n}
    +
    M(1+V(x_1)+V(x_2))\rho^{c_1 n}.
\end{align*}

Using the above bound, by the measure decomposition theorem, we can write
\begin{align*}
\P (\tau\ge n)&\le \P (Z_{0,n}^{X_0}\ne Z_{0,n}^{X_0^*}) \\
&=
\int_{\Y^\N} 
\left[ 
\int_{\X\times\X}
\P (Z_{0,}^{x_0,\mathbf{y}}\ne Z_{0,n}^{x_0^*,\mathbf{y}}) 
\law{X_0}(\dint x_0)\otimes \law{X_0|\mathbf{Y}=\mathbf{y}}(\dint x_0^*)
\right]
\law{\mathbf{Y}}(\dint \mathbf{y}) 
\\
&\le \P (L_{n,0}\le c_1 n) + M(1+\E[V(X_0)]+\E[V(X_0^*)])\rho^{c_1 n}. 
\end{align*}

In the proof of Lemma 2.6, we showed that \(\E[V(X_0^*)]<\infty\). Using this result together with part 1 of Lemma~\ref{conc}, and following the concluding steps in the proof of Theorem~\ref{thm:A12fromMixing}, we obtain the stated upper bound for the tail probability \(\P(\tau \ge n)\), \(n \ge 1\).
\qed

\section{Annex}

The following result can be found in \cite{Rio}, see Theorem $6.1$.

\begin{proposition}\label{exp1}
Let $(X_i)_{i\geq 1}$ be a sequence of real-valued random variables such that $\Vert X_i\Vert_{\infty}\leq M$ for any positive $i$
and $\left(\alpha_n\right)_{n\geq 0}$ be its sequence of strong mixing coefficients.  Set $X_i=0$ for for any $i>n$ and let 
$S_n=\sum_{i=1}^n\left(X_i-\E X_i\right)$.  Let $q>1$ be an integer, and $v_q$ be any positive real such that
$$v_q \geq \sum_{i\geq 1}\E\left[\left(X_{iq-q+1} +\cdots + X_{iq}\right)^2\right].$$
Set $M(n)=\sum_{i=1}^n \Vert X_i\Vert_{\infty}$ . Then, for any $\lambda \geq qM$,
$$\P\left(\left\vert S_n\right\vert\geq  7\lambda/2\right)\leq 4\exp\left(-\frac{v_q}{2qM}\log\left(1+\frac{\lambda q M}{v_q}\right)\right)+4M(n)\frac{\alpha_{q+1}}{\lambda}.$$
\end{proposition}

We will use the following version of the previous exponential inequality. Using the same notations as in Proposition \ref{exp1}, there exist positive constants $c_1,c_2$ such that for any positive integer $n$ and any $\delta>0$ and $q\in(1,n\delta/M)$,
\begin{equation}\label{sufexp1}
\P\left(S_n\geq n\delta\right)\leq c_1\left\{\exp\left(-c_2 n\delta/q\right)+\alpha_{q+1}\right\}.
\end{equation}

The next result can be found in \cite{merlevede2011bernstein}, Theorem $1$.

\begin{proposition}\label{exp2}

Let us keep the notations of Proposition \ref{exp1} and assume that $\alpha_n=O\left(\exp(-cn)\right)$ for some $c>0$. There exists a positive constant $C$ such that
Then for any $\delta>0$ and any positive integer $n$, we have
$$\P\left(\left\vert S_n\right\vert\geq n\delta\right)\leq \exp\left(-\frac{C \delta^2 n}{M^2+M\delta \log n\log\log n}\right).$$
\end{proposition}

We end this section with a useful lemma which shows that an expectation of product of dependent random variables, not necessarily smaller than one, can goes to zero exponentially fast, provided an Hoeffding type exponential inequality for partial sums is avalailable.
\begin{lemma}\label{useful}
Let $W_1,\ldots,W_n$ be some non-negative random variables such that 
$$\ell:=-\max_{1\leq i\leq n}\E\log W_i>0.$$
Setting $Z_j=\log W_j-\E\left[\log W_j\right]$, we assume that there exists $M>0$ such that
for all $t>0$ and $n\geq 1$,
\begin{equation}\label{first}
\P\left(\sum_{j=1}^n Z_j>t\right)\leq \exp\left(-\frac{t^2}{Mn}\right).
\end{equation}
Then for any $\delta>0$, we have
$$q_{\delta}(n):=\E\left[\left(W_1\cdots W_n\right)^{\delta}\right]\leq e^{-{\delta} n\ell}+2{\delta}\sqrt{\pi}\sqrt{Mn}e^{\frac{{\delta}^2 Mn}{4}-{\delta} n\ell}.$$
Then we have $\lim_{n\rightarrow \infty}q_{\delta}(n)=0$ as soon as ${\delta}<\frac{4\ell}{M}$.
\end{lemma}
\paragraph{Proof of Lemma \ref{useful}.}
To prove the lemma, set $S_n=\sum_{j=1}^n Z_j$. We note that 
\begin{eqnarray*}
q_{\delta}(n)&\leq&\int_0^{\infty}\P\left(S_n>n\ell+\log(t)/{\delta}\right)dt\\
&\leq&\int_{-\infty}^{\infty}{\delta}e^{{\delta}u}\P\left(S_n>n\ell+u\right)du\\
&\leq& \int_{-\infty}^{-n\ell}{\delta}e^{{\delta}u}du+\int_{-n\ell}^{\infty} {\delta}e^{{\delta}u}\exp\left(-\frac{(u+n\ell)^2}{Mn}\right)du\\
&\leq& e^{-{\delta}n\ell}+\sqrt{Mn}{\delta}e^{-{\delta}n\ell}\int_{-\infty}^{\infty}e^{{\delta}\sqrt{Mn} v}e^{-v^2}dv.
\end{eqnarray*}
We conclude by using the equality $\int_{-\infty}^{\infty} e^{\alpha v}e^{-v^2}dv=2\sqrt{\pi}e^{\frac{\alpha^2}{4}}$.$\square$

\medskip
In the following, we consider three notions of strong mixing--$\alpha$-, $\phi$-, and $\psi$-mixing--and a sequence of positive (possibly bounded) random variables $(\gamma_n)_{n \in \mathbb{N}}$ with $\mathbb{E}[\gamma_n] < 1$. Our goal is to derive elementary upper bounds for $\mathbb{E}[\gamma_1 \cdots \gamma_n]$ in terms of $\sup_{n \in \mathbb{N}} \mathbb{E}[\gamma_n]$ and the corresponding mixing coefficients.

\begin{lemma}\label{lem:Th}
	Let $\Theta_1,\ldots,\Theta_n$, $n\ge 1$ be almost surely positive random variables satisfying 
    $$
	\hat{\Theta}:=\max_n \E [\Theta_n]<1.
    $$

    We introduce the shorthand notations \(\alpha := \alpha^\Theta(1)\), \(\phi := \phi^\Theta(1)\), and \(\psi := \psi^\Theta(1)\). Then, the expectation of the product \(\Theta_1 \cdot \ldots \cdot \Theta_n\) can be bounded in terms of \(\alpha\), \(\phi\), and \(\psi\) as follows:
    \begin{enumerate}
        \item If $\max_n \|\Theta_n\|_{\infty}\le \bar{\Theta}<\infty$, then
        \[
        \E[\Theta_1 \ldots \Theta_n] \le
        \begin{cases}
		\displaystyle \frac{\alpha \bar{\Theta}^n}{1 - \frac{\hat{\Theta}}{\bar{\Theta}}} + \hat{\Theta}^n, \\
		\displaystyle (\phi \bar{\Theta} + \hat{\Theta})^{n-1} \hat{\Theta}.
	\end{cases}
    \]
    
        \item \[
        \E[\Theta_1 \ldots \Theta_n] \le (1 + \psi)^{n-1} \hat{\Theta}^n.
        \]
    \end{enumerate}
\end{lemma}
\begin{proof}
	By the layer-cake representation, we can write
	\begin{equation}\label{eq:cake}
		\E [\Theta_1\ldots\Theta_n] = \int_{[0,\infty)^n}
		\P \left(\Theta_k\ge\theta_k,\,k=1,\ldots,n\right)
		\dint \theta_1\ldots\dint\theta_n
	\end{equation}
	For $n\ge 2$, the probability inside the integral can be estimated as
	\begin{align*}
	\P \left(\Theta_k\ge\theta_k,\,k=1,\ldots,n\right) &\le  \alpha +	\P \left(\Theta_k\ge\theta_k,\,k=1,\ldots,n-1\right)
	\P(\Theta_n\ge\theta_n)\\[1em]
	\P \left(\Theta_k\ge\theta_k,\,k=1,\ldots,n\right) &\le 
	\begin{cases}
		 \P (\Theta_1\ge\theta_1)\prod_{k=2}^{n} \left(\phi+\P (\Theta_k\ge\theta_k)\right) & \\
		 (1+\psi)^{n-1} \prod_{k=1}^{n}\P (\Theta_k\ge\theta_k). &
	\end{cases}
	\end{align*}
	
	Substituting these into \eqref{eq:cake} for $n\ge 2$, yields
	\begin{align}\label{eq:toiter}
			\E [\Theta_1\ldots\Theta_n] &\le \alpha \bar{\Theta}^n + \hat{\Theta}\E [\Theta_1\ldots\Theta_{n-1}]\\[1em] 
			\E [\Theta_1\ldots\Theta_n] &\le 
			\begin{cases}
				(\phi\bar{\Theta}+\hat{\Theta})^{n-1}\hat{\Theta} & \\
				(1+\psi)^{n-1}\hat{\Theta}^n. &
			\end{cases} \nonumber
	\end{align}
	where by induction easily seen that the iteration of \eqref{eq:toiter} results
	$$
	\E [\Theta_1\ldots\Theta_n] 
	\le 
	\alpha \sum_{i=0}^{n-1}\bar{\Theta}^{n-i}\hat{\Theta}^i
	+\hat{\Theta}^n
	\le 
	\frac{\alpha\bar{\Theta}^n}{1-\frac{\hat{\Theta}}{\bar{\Theta}}}
	+
	\hat{\Theta}^n
	$$
	which completes the proof.
\end{proof}

In the following, we adapt the approach outlined in Bosq \cite{bosq1993bernstein}, using a decomposition based on residue classes to derive a uniform upper bound in $j$ for the expectations $\E[\ga_{j+1}^\delta \cdot \ldots \cdot \ga_{j+n}^\delta]$, where $n \ge 1$ and $\delta>0$ is an appropriate exponent.

\begin{lemma}\label{lem:ga}
Assume that the sequence $(\gamma_n)_{n \in \mathbb{N}}$ consists of positive random variables and satisfies \(\hat{\gamma} := \sup_{n \in \mathbb{N}} \mathbb{E}[\gamma_n] < 1\). Then there exist constants $\delta \in (0,1)$, $c > 0$, and $\kappa \in (0,1)$ such that
\[
\sup_{j \in \mathbb{N}} \mathbb{E}\left[\gamma_{j+1}^\delta \cdot \ldots \cdot \gamma_{j+n}^\delta\right] \le c \kappa^n, \quad n \ge 1,
\]
provided that at least one of the following conditions holds:
\begin{enumerate}
    \item the sequence $(\gamma_n)_{n \in \mathbb{N}}$ is $\phi$-mixing and uniformly bounded;
    \item the sequence $(\gamma_n)_{n \in \mathbb{N}}$ is $\psi$-mixing.
\end{enumerate}
\end{lemma}
\begin{proof}
	Let $p, q$ be integers such that $pq \le n < p(q+1)$. Then we can write
	\[
	\ga_{j+1} \cdot \ldots \cdot \ga_{j+n} = \eta_1 \cdot \ldots \cdot \eta_p \, \Delta_n,
	\]
	where $\eta_k = \prod_{l=0}^{q-1} \ga_{j + k + lp}$ for $k = 1, \ldots, p$, and $\Delta_n = \prod_{l=pq+1}^{n} \ga_{j+l}$. Since $0\le n - pq < p$, by applying Hölder's inequality, we obtain
    \begin{align*}
    \E[\ga_{j+1}^{\frac{1}{2p-1}} \cdot \ldots \cdot \ga_{j+n}^{\frac{1}{2p-1}}]
    &\le \prod_{k=1}^p \E^{\frac{1}{2p-1}}[\eta_k] \times 
    \prod_{l=pq+1}^n \E^{\frac{1}{2p-1}} [\ga_{j+l}]
    \\
    &\le \left(\max_{1\le k\le p} \E [\eta_k]\right)^{\frac{p}{2p-1}}
    \hat{\ga}^{\frac{n-pq}{2p-1}}
    \le
     \left(\max_{1\le k\le p} \E [\eta_k]\right)^{\frac{1}{2}}.
    \end{align*}

    Clearly, for each fixed $k$, the variables $\Theta_{l+1} := \gamma_{j+k+lp}$ for $l = 0, \ldots, q-1$ satisfy the assumptions of Lemma~\ref{lem:Th}. Moreover, the mixing coefficients satisfy $\phi^\Theta \le \phi^\gamma(p)$ and $\psi^\Theta \le \psi^\gamma(p)$. 

In the bounded case, applying part 1 of Lemma~\ref{lem:Th}, we obtain
\[
\mathbb{E}[\eta_k] \le (\phi^\gamma(p)\bar{\gamma} + \hat{\gamma})^{q-1} \hat{\gamma} \le (\phi^\gamma(p)\bar{\gamma} + \hat{\gamma})^q,
\]
where $\bar{\gamma} := \sup_{n \in \mathbb{N}} \|\gamma_n\|$.

Similarly, in the unbounded case, by part 2 of Lemma~\ref{lem:Th}, we have
\[
\mathbb{E}[\eta_k] \le (1 + \psi^\gamma(p))^{q-1} \hat{\gamma}^q.
\]

In both cases (1 and 2), one can choose a fixed $p \ge 2$ large enough such that
\[
\max_{1 \le k \le p} \mathbb{E}[\eta_k] \le c \kappa^n, \quad n \in \mathbb{N},
\]
for some constants $c > 0$ and $\kappa \in (0,1)$. Consequently, for any $\delta \in \left(0, \frac{1}{2p - 1}\right]$, the claimed inequality holds.

\end{proof}

\section{Proof of Lemma \ref{lem:coupl_trans_mix}}\label{ap:proof_of_coupl_trans_mix}
The proof essentially follows the argument of Lemma~1.2 in \cite{lovas2024transition}.  
Let $j \in \N$ and $n \ge p-1$ be arbitrary. Then there exist $k, k' \in \N$ such that
\[
[j, j+n] \cap p\N = \{\, lp \mid k \le l \le k' \,\}.
\]
As a first step, we extend the process $Z_{s,t}^x$ to indices $s$ and $t$ such that either $s = 0$ or $s \in [j, j+n] \cap p\N$, while $t \ge s$ is allowed to take any integer value strictly smaller than $kp$ or strictly greater than $k'p$. This construction depends on $j$, but for simplicity of notations, we omit the index $j$ in the definition of $Z_{s,t}^x$.

Since \(\X\) and \(\Y\) are standard Borel spaces, there exists a measurable function
\[
f : \X \times \Y \times [0,1] \to \X
\]
such that
\[
Q(y, x, B) = \mathbb{E}\left[ \ind_{\{ f(x, y, \eps_0) \in B \}} \right], \quad y \in \Y,\, x \in \X,\, B \in \B(\X).
\]
However, unlike the \(p\)-step dynamics given by \(f^p\), the function \(f\) does not necessarily satisfy any coupling inequality.  
Using the mapping \(f\), we define the following extension of the process \(Z\). For $s\in [j,j+n]\cap p\N$, we set $Z_{s,s}^x=x$ and for $t>s$:
\begin{equation}\label{eq:Zext}
Z_{s,t}^x = \begin{cases}
f^p\bigl(Z_{s,t-p}^x, Y_{t-p}^p, \eps_t\bigr) & \text{if } t = \tilde{k} p \text{ with } k < \tilde{k} \le k', \\
f\bigl(Z_{s,t-1}^x, Y_{t-1}, \eps_t\bigr) & \text{otherwise if $t\geq k'p+1$},
\end{cases}
\end{equation}
When $j\geq 1$, we also define $Z_{0,0}^x=x$, $Z_{0,t}=f\left(Z_{0,t-1}^x,Y_{t-1},\varepsilon_t\right)$ for $1\leq t\leq j$ and $Z_{0,t}^x$ defined as in (\ref{eq:Zext}) for $t>j$.
where \(s = 0\) or \(s \in [j, j+n] \cap p\N\) and \(t > s\), with the initial condition \(Z_{s,s}^x = x\).  
In words, this means that up to time \(k p\), the dynamics evolve by the one-step mapping \(f\); between times \(k p\) and \(k' p\) we iterate by steps of size \(p\) via \(f^p\); and beyond time \(k' p\) the evolution proceeds again according to the one-step dynamics defined by \(f\). Note that $Z_{s,t}^x$ is not always defined for a time point $t\in [s,j+n]$ which is not a multiple of $p$, such a definition is not indeed for evaluating the mixing coefficients. Moreover, we define $Z_{0,0}^x=x$ and $Z_{0,t}=f\left(Z_{0,t-1}^x,Y_{t-1},\varepsilon_t\right)$ for $1\leq t\leq j$. 

Let $A \in \mathcal{F}_{0,j}^{X,Y}$ and $B \in \mathcal{F}_{j+n, \infty}^{X,Y}$ be arbitrary events. By the definition of the generated $\sigma$-algebra, there exist collections of Borel sets  
$(A_l)_{l \le j}$, $(B_l)_{l \ge j+n} \subseteq \B (\X \times \Y)$, such that  
\begin{equation}\label{eq:defAB}
A = \big\{ (X_l, Y_l) \in A_l \ \text{for all}\ 0 \le l \le j \big\},  
\quad  
B = \big\{ (X_l, Y_l) \in B_l \ \text{for all}\ l \ge j+n \big\}.
\end{equation}
Since the processes \((X_l, Y_l)\) for \(0 \le l \le j\) or \(l \ge j+n\) and \((Z_{0,l}^{X_0}, Y_l)\) for the same indices are versions of each other, we may replace \(X_l\) by \(Z_{0,l}^{X_0}\) in the definition of the events \(A\) and \(B\) in \eqref{eq:defAB} without loss of generality. 
We do not introduce a separate notation for this substitution and understand \(A\) and \(B\) accordingly in the sequel.

Let \( k \le k'' \le k'\) be arbitrary, and introduce the event
$
\tilde{B} = \left( (Z_{k''p,l}^{x_0}, Y_l) \in B_l, \, l \ge j+n \right).
$
We can estimate
\begin{equation}\label{eq:covest}
	|P(A \cap B) - P(A)P(B)| = |\text{cov}(\ind_A, \ind_B)| 
	\le 
	|\text{cov}(\ind_A, \ind_{\tilde{B}})| + 
	|\text{cov}(\ind_A, \ind_B - \ind_{\tilde{B}})|.
\end{equation}
Using that $A$ is $\sigma (X_0)\vee\F_{0,j}^Y\vee\F_{1,j}^\eps$ and $\tilde{B}$ is $\F_{k''p,\infty}^Y\vee\F_{k''p+1,\infty}^\eps$-measurable, and that $X_0$ is independent of $\sigma (Y_n,\eps_{n+1}\mid n\in\N)$, we have
\begin{align*}
	|\text{cov}(\ind_A, \ind_{\tilde{B}})| 
	&\le 
	\E\left| 
	\E \left[\ind_{A}(\ind_{\tilde{B}}-\P(\tilde{B}))
	\middle|X_0
	\right]
	\right|
	\le 
	\alpha (\F_{0,j}^Y\vee\F_{1,j}^\eps,\F_{k''p,\infty}^Y\vee\F_{k''p+1,\infty}^\eps).
\end{align*}
Since the process $(\eps_n)_{n \in \N}$ is i.i.d. and independent of $(Y_n)_{n \in \N}$, by point a) of Theorem 5.1 in \cite{Bradley2005},
$\alpha (\F_{0,j}^Y\vee\F_{1,j}^\eps,\F_{k''p,\infty}^Y\vee\F_{k''p+1,\infty}^\eps) = \alpha (\F_{0,j}^Y,\F_{k''p,\infty}^Y)\le \alpha^Y (k''p-j)$, and thus for the first term in \eqref{eq:covest}, we obtain
$$
|\text{cov}(\ind_A, \ind_{\tilde{B}})| \le \alpha^Y (k''p-j)\le \alpha^Y((k''-k)p).
$$

As for the second term, observe that on the event $\{Z_{0,k'p}^{X_0}=Z_{k''p, k'p}^{x_0}\}$, we have $\ind_{B} = \ind_{\tilde{B}}$, and thus we can estimate by writing	
\begin{align*}
|\cov(\ind_{A},\ind_{B}-\ind_{\tilde{B}})|
&=
\left|\E [(\ind_{A}-\P (A))(\ind_{B}-\ind_{\tilde{B}})]\right|
\le
\P \left(Z_{0, k'p}^{X_{0}}\ne Z_{k''p, k'p}^{x_0}\right)
\\
&=\P \left(Z_{k''p, k'p}^{X_{k''p}}\ne Z_{k''p, k'p}^{x_0}\right)
\le b(k'-k'').
\end{align*}

Given that $A\in\F_{0,j}^{X,Y}$ and $B\in\F_{j+n,\infty}^{X,Y}$ were arbitrary, we have shown that 
$$
\alpha (\F_{0,j}^{X,Y},\F_{j+n,\infty}^{X,Y})
\le
\alpha^{Y} ((k''-k)p)+b(k'-k'').
$$
Noting that for any \(j, n \in \N\) the interval \([j, j+n]\) contains at least \(\lfloor \frac{n}{p} \rfloor\) integers divisible by \(p\) , we have \(k' - k \ge \lfloor \frac{n}{p} \rfloor - 1\).  
It follows that, for any \(0 \le m < \lfloor \frac{n}{p} \rfloor\), we may choose \(k''\) such that \(k'' - k = m\).  
In this case,
\[
k' - k'' = k' - k - (k'' - k) \ge \lfloor \tfrac{n}{p} \rfloor - 1 - m,
\]
and by the monotonicity of the sequence \(b\) we obtain
\[
\alpha(\F_{0,j}^{X,Y}, \F_{j+n,\infty}^{X,Y})
\le \alpha^{Y}(mp) + b\!\left(\lfloor \tfrac{n}{p} \rfloor - 1 - m\right).
\]
Since the upper bound does not depend on \(j\), taking the supremum over \(j\) yields the desired inequality.

\qed

\end{document}